\documentclass[10pt]{article}
\usepackage[utf8]{inputenc} 
\usepackage[T1]{fontenc}
\usepackage{lmodern}

\usepackage{amsthm,amsmath,amsfonts,amssymb,stmaryrd,bbm}
\usepackage{mathtools}
\usepackage{enumitem}
\usepackage{mathrsfs} 

\usepackage{graphicx}
\usepackage[dvipsnames]{xcolor}
\usepackage{caption,tikz}
\usepackage{dsfont}
\usepackage{amsthm}

\usepackage{anyfontsize}

\usepackage{mathtools} 

\usepackage[english]{babel}
\usepackage[colorlinks=true, linkcolor=blue, urlcolor=black, citecolor=blue,pdfstartview=FitH]{hyperref}

\numberwithin{equation}{section}

\let\originalleft\left
\let\originalright\right
\renewcommand{\left}{\mathopen{}\mathclose\bgroup\originalleft}
\renewcommand{\right}{\aftergroup\egroup\originalright}

\newlength{\bibitemsep}
\newlength{\bibparskip}
\let\oldthebibliography\thebibliography
\renewcommand\thebibliography[1]{\oldthebibliography{#1}
  \setlength{\parskip}{\bibitemsep}
  \setlength{\itemsep}{\bibparskip}}

\DeclareMathOperator{\re}{Re}
\DeclareMathOperator{\im}{Im}

\DeclareMathOperator{\tr}{tr}
\DeclareMathOperator{\Var}{Var}

\DeclareMathOperator{\oo}{o}

\DeclareMathOperator{\Id}{Id}
\DeclareMathOperator{\diag}{diag}
\DeclareMathOperator{\supp}{supp}

\DeclareMathOperator{\rank}{rank}
\DeclareMathOperator{\Leb}{Leb}

\newcommand{\mc}[1]{\mathcal{#1}}
\newcommand{\mf}[1]{\mathfrak{#1}}

\newcommand{\ii}{\mathrm{i}}

\newcommand{\defeq}{\vcentcolon=}
\newcommand{\eqdef}{=\vcentcolon}

\renewcommand{\epsilon}{\varepsilon}

\renewcommand{\leq}{\leqslant}
\renewcommand{\geq}{\geqslant}

\renewcommand{\P}{\mathbb{P}}
\newcommand{\E}{\mathbb{E}}
\newcommand{\R}{\mathbb{R}}
\newcommand{\C}{\mathbb{C}}
\newcommand{\N}{\mathbb{N}}
\newcommand{\Z}{\mathbb{Z}}

\newcommand{\abs}[1]{\left\lvert #1 \right\rvert}

\newcommand{\vertiii}[1]{{\left\vert\kern-0.25ex\left\vert\kern-0.25ex\left\vert #1 
    \right\vert\kern-0.25ex\right\vert\kern-0.25ex\right\vert}}
\newcommand{\ip}[1]{\left\langle #1 \right\rangle}
\newcommand{\diff}{\mathop{}\!\mathrm{d}}

\theoremstyle{plain} 
\newtheorem{thm}{Theorem}[section]
\newtheorem{lem}[thm]{Lemma}
\newtheorem{cor}[thm]{Corollary}

\newtheorem{prop}[thm]{Proposition}
\newtheorem{assn}{Assumption}
\newtheorem{assnA}{Assumption}

\newtheorem*{theorem*}{Theorem}
\newtheorem{defn}[thm]{Definition}

\newtheorem{rem}[thm]{Remark}

\theoremstyle{definition}
\newtheorem{assump}{Assumption}

\makeatletter
\renewcommand{\subsection}{\@startsection
{subsection}
{2}
{0mm}
{-\baselineskip}
{0 \baselineskip}
{\normalfont\bf\itshape}} 
\makeatother

\makeatletter
\renewcommand{\subsubsection}{\@startsection
{subsubsection}
{3}
{0mm}
{-\baselineskip}
{0 \baselineskip}
{\normalfont\bf\itshape}} 
\makeatother

\def\author#1{\par
    {\centering{\authorfont#1}\par\vspace*{0.05in}}
}

\def\titlefont{\fontsize{13}{15}\bfseries\boldmath\selectfont\centering{}}
\def\authorfont{\fontsize{13}{15}}

\let\affiliationfont\rhfont

\def\address#1{\par
    {\centering{\affiliationfont#1\par}}\par\vspace*{11pt}
}

\def\title#1{
    \thispagestyle{plain}
    \vspace*{-14pt}
    \vskip 79pt
    {\centering{\titlefont #1\par}}%
    \vskip 1em
}

\makeatletter
\newcommand{\setword}[2]{%
  \phantomsection
  #1\def\@currentlabel{\unexpanded{#1}}\label{#2}%
}
\makeatother

\makeatletter
\renewcommand{\section}{\@startsection
{section}
{1}
{0mm}
{-2\baselineskip}
{1\baselineskip}
{\normalfont\large\scshape\centering}} 
\makeatother

\begin{document}
\enlargethispage{1cm}

~\vspace{-3.75cm}

\title{Large deviations for the largest eigenvalue of generalized sample covariance matrices}

\noindent
\begin{minipage}[c]{0.5\textwidth}
\author{Jonathan Husson}
\address{University of Michigan\\
    E-mail: jhusson@umich.edu}
\end{minipage}
\begin{minipage}[c]{0.5\textwidth}
\author{Benjamin McKenna}
\address{Harvard University \\ Center of Mathematical Sciences and Applications \\
    E-mail: bmckenna@fas.harvard.edu}
\end{minipage}

\vspace*{-0.1in}

\begin{abstract}
We establish a large-deviations principle for the largest eigenvalue of a generalized sample covariance matrix, meaning a matrix proportional to $Z^T \Gamma Z$, where $Z$ has i.i.d. real or complex entries and $\Gamma$ is not necessarily the identity. We treat the classical case when $Z$ is Gaussian and $\Gamma$ is positive definite, but we also cover two orthogonal extensions: Either the entries of $Z$ can instead be sharp sub-Gaussian, a class including Rademacher and uniform distributions, where we find the same rate function as for the Gaussian model; or $\Gamma$ can have negative eigenvalues if $Z$ remains Gaussian. The latter case confirms formulas of Maillard in the physics literature. 

We also apply our techniques to the largest eigenvalue of a deformed Wigner matrix, real or complex, where we upgrade previous large-deviations estimates to a full large-deviations principle. Finally, we remove several technical assumptions present in previous related works.
\end{abstract}

\vspace*{0.05in}

\noindent \emph{Date:} February 3, 2023

\vspace*{0.05in}

\noindent \hangindent=0.2in \emph{Keywords and phrases:} Large deviations, sample covariance matrices, Wishart matrices, deformed Wigner matrices

\vspace*{0.05in}


\vspace*{-0.2in}

{
	\hypersetup{linkcolor=black}
	\tableofcontents
}


\vspace{-0.5cm}
\section{Introduction}


\subsection{Our results.}\

In this paper, we give a large-deviations principle at speed $N$ for the largest eigenvalue $\lambda_{\textup{max}}$ of a \emph{generalized sample covariance matrix}, meaning a random matrix of the form
\begin{equation}
\label{eqn:modeldef}
	H_N = \frac{1}{M} Z^T \Gamma Z = \frac{1}{M} \sum_{i=1}^M d_i z_i z_i^T,
\end{equation}
where all the notation is defined below. Informally, this means that we find a function $I$ such that
\begin{equation}
\label{eqn:introldp}
	\P(\lambda_{\textup{max}}(H_N) \approx x) \approx \exp(-NI(x)).
\end{equation}
Here $Z \in \R^{M \times N}$ has i.i.d. entries distributed according to some centered probability measure $\mu$ with unit variance; its rows, which are independent vectors of length $N$, are written in column form (in order to match the literature) as $\{z_i\}_{i=1}^M$; \emph{generalized} means that the deterministic matrix $\Gamma = \diag(d_1, \ldots, d_M) = \diag(d_1^{(M)}, \ldots, d_M^{(M)})$ may not be the identity; and $M$ and $N$ are parameters both tending to infinity in such a way that their ratio is order one.

Our results also hold for the complex case $H_N = M^{-1} Z^\ast \Gamma Z$, but for the sake of simplicity we limit our notation to the real case for most of the paper.

When all $d_i$'s are positive, the matrix $H_N$ is sometimes called an \emph{inhomogeneous} sample covariance matrix (or an inhomogeneous Wishart matrix). In this paper, we can allow for negative $d_i$'s, in the special case when $\mu$ is Gaussian. This part of our result matches a recent result of Maillard in the physics literature, which inspired our work \cite{Mai2021}. We also allow non-Gaussian $\mu$'s, but we require them to lie in a special class of measures called \emph{sharp sub-Gaussian}, which includes as important special cases the Rademacher and uniform distributions, and here we require all the $d_i$'s to be positive; in this case we show that the rate function is the same as in the corresponding Gaussian case. (The rate function probably should \emph{not} match the Gaussian one when some $d_i$'s are negative; see Remark \ref{rem:ssg_negative_di}.)

The case where $\Gamma$ is positive definite is better known. In that case, $H_N$ is a sample covariance matrix of \emph{independent} observations of possibly \emph{correlated} data, i.e., it has the same distribution as $\frac{1}{M} \sum_{i=1}^M x_ix_i^T$, where the vectors $x_i$ are independent, but the components of each $x_i$ have possibly nontrivial covariance matrix $\Gamma$. In this case, $H_N$ of course also has the same nonzero eigenvalues as the variant $\frac{1}{M} \Gamma^{1/2} ZZ^T \Gamma^{1/2}$, and the literature is sometimes phrased in this way. But we emphasize that the case where $\Gamma$ has positive and negative eigenvalues also has genuine statistical interest: it appears, for example, in MANOVA estimation of the above scenario when the observations $x_i$ are no longer independent of one another. For a discussion of this case, we direct the readers to \cite{FanJoh2022}.


\subsection{Previous results on large deviations for random matrices.}\

Our work fits into a recent lineage of papers, starting with \cite{GuiHus2020} and \cite{GuiMai2020}, which establish large-deviations principles for random matrices using the method of \emph{tilting by spherical integrals}. Papers in this line include \cite{BirGui2020, AugGuiHus2021, McK2021, GuiHus2022, Hus2022, BelGuiHua2020} in the math literature, and \cite{Mai2021, MerPot2022} in the physics literature. Compared to these previous works, the main technical novelty in our work is the establishment of rigorous non-Gaussian results beyond the so-called \emph{$x_c$ threshold}. That is, in the model \eqref{eqn:modeldef}, it turns out that there exists some $x_c = x_c(\Gamma)$, which may be finite or infinite depending on $\Gamma$, such that if $x_c < \infty$ then the function $I$ from \eqref{eqn:introldp} is non-analytic at $x_c$. A similar phenomenon was partially shown for additively deformed Wigner matrices in \cite{McK2021}; we will define this model properly below, but informally speaking, it comes in both a Gaussian variety and a sharp sub-Gaussian variety (and other varieties not considered). In \cite{McK2021}, one of the present authors showed the analogue of \eqref{eqn:introldp} for all $x$ in the Gaussian case, but only for $x < x_c$ in the sharp sub-Gaussian case. Dealing with the regime $x \geq x_c$ in the sharp sub-Gaussian case remained a challenge.

Following these works, Maillard considered the present model \eqref{eqn:modeldef}, in the Gaussian case and at the physics level of rigor. He also found an $x_c$ threshold, and proposed an interesting new tilting for establishing \eqref{eqn:introldp} in the regime $x \geq x_c$, which motivated our work. We are able to verify his results, plus add new ones of our own, but without using his techniques. Instead, when considering a model with $x_c < +\infty$, we find an approximating sequence of models which all have $x_c = +\infty$. In the context of recent results on large deviations for random matrices, a similar argument first appeared in \cite{Hus2022}, but to handle a problem other than the $x_c$ threshold. Here we show it can also be useful for $x_c$ thresholds, in a variety of models. Indeed, although the bulk of the paper treats the sample-covariance model \eqref{eqn:modeldef}, in Appendix \hyperlink{sec:wigner}{C} we consider the deformed-Wigner case and complete the analysis started by \cite{McK2021}. 

As special cases of our main theorem, we recover earlier results in the homogeneous case $\Gamma = \Id$, i.e., for usual Wishart matrices. Majumdar and Vergassola computed the rate function in the Gaussian case \cite{MajVer2009}; later, recent work by Guionnet and the first author extended this sharp sub-Gaussian Wishart matrices \cite{GuiHus2020}. The computation that our rate function matches theirs was carried out by Maillard \cite{Mai2021}.

We also develop techniques to remove a technical assumption present in the works cited above. Typically, the results of these works can be written informally as ``Suppose $X$ is a random matrix built from samples of $\mu$, a sharp-sub-Gaussian measure, which is also either compactly supported or satisfies the log-Sobolev inequality. Then $\lambda_{\textup{max}}(X)$ satisfies an LDP.'' We show how to remove the ``compact-or-log-Sobolev'' assumption, ending with theorems of the form ``Suppose $X$ is a random matrix built from samples of $\mu$, a sharp sub-Gaussian measure. Then $\lambda_{\textup{max}}(X)$ satisfies an LDP.'' For example, our main result is written in this latter way, as is our deformed-Wigner result just mentioned. In Appendix \hyperlink{sec:compact_or_log_sobolev}{B}, we give one more example, refining a result of Guionnet and the first author \cite{GuiHus2020}. We show that this extension is nontrivial, in Remark \ref{rem:new_ssg_example}, by giving an example of a sharp sub-Gaussian law that is neither compactly supported nor satisfies the log-Sobolev inequality. As we explain there, our proof also allows one to bypass the use of local laws.

Finally, we make a historical remark on the sharp-sub-Gaussian condition. Of course the notion of ``sub-Gaussian random variable'' is quite classical, but the works in the lineage discussed above are all based on the additional ideas that (a) a particular subclass of this family, called ``sharp-sub-Gaussian'' in Definition \ref{defn:ssg} below, is special, and (b) what makes this subclass special, is that random matrices constructed from its members satisfy bounds with the \emph{same constants} as for Gaussian measure itself. To the best of our knowledge, these LDP works were interested in finding, but not aware of, any other place in the probability literature where these ideas appeared. We recently learned of such a place, namely a remark in recent work of Zhivotovskiy on the operator norm of a sum of independent random matrices \cite[Remark 6]{Zhi2021}. Zhivotovskiy remarks that these same ideas previously appeared implicitly in works of Catoni and collaborators.


\subsection{Previous results on generalized sample covariance matrices.}\

The model $H_N$ and its variants are quite classical. Their first appearance, to the best of our knowledge, is in a 1967 paper of Mar\v{c}enko and Pastur, which proved a global law for the variant $A + \frac{1}{M} Z^T \Gamma Z$, where $A$ is a sequence of deterministic matrices whose empirical measures have some limit, and where $\Gamma$ is random and diagonal with i.i.d. entries. The global law for our precise model appears in (more general) work of Silverstein and Bai \cite{SilBai1995}, which also gives a good summary of the literature on global laws for variants of $H_N$.

In our work we assume that $\Gamma$ has no outlier eigenvalues; for example we do not allow $\Gamma$ to be a finite-rank perturbation of the identity. This restriction is generally believed, and under very mild restrictions proved, to prevent $H_N$ from typically having outlier eigenvalues. That is, we are not in the regime of the Baik-Ben Arous-P\'ech\'e (BBP) transition.

Thus the largest eigenvalue tends to the right endpoint of the asymptotic support of $H_N$. For some variants of $H_N$, it is known to have Tracy-Widom fluctions around this point. For example, in the complex Gaussian case, which is determinantal, El Karoui \cite{ElK2007} introduced an edge regularity condition and treated all $\Gamma$'s satisfying this condition, in the regime $\frac{M}{N} \geq 1$; Onatski \cite{Ona2008} extended this to $\frac{M}{N} < 1$. In the direction of universality, Bao, Pan, and Zhou \cite{BaoPanZho2015} allowed $Z$ to have complex non-Gaussian entries; Lee and Schnelli \cite{LeeSch2016} allowed $Z$ to have real non-Gaussian entries, as long as $\Gamma$ is diagonal; Knowles and Yin \cite{KnoYin2017} gave the (non-trivial) extension of this to non-diagonal $\Gamma$. Although we only treat the largest eigenvalue, it is also of statistical interest to consider the case when $\Gamma$ is such that $H_N$ has asymptotically disconnected support, and show that the rightmost eigenvalues for \emph{each connected component} of the support each have Tracy-Widom fluctuations, which are independent of each other. Hachem, Hardy, and Najim \cite{HacHarNaj2016} obtained such a result in the complex Gaussian case, followed by recent work of Fan and Johnstone \cite{FanJoh2022} in the real Gaussian case. Finally, Wang \cite{Wan2019} recently obtained the first non-trivial speed of convergence to the Tracy-Widom distribution, of order $N^{-1/57}$ (in the case $\Gamma = \Id$, Wang found a much better speed of $N^{-2/9}$, which Schnelli and Xu \cite{SchXu2021} recently improved to almost $N^{-1/3}$).\footnote{
Many of the cited works are written for the model $\frac{1}{M} \Gamma^{1/2} Z^T Z \Gamma^{1/2}$, which of course has the same non-zero eigenvalues as our model, but requires $\Gamma \geq 0$. 
}


\subsection{Organization.}\

The paper is organized as follows: In Section \ref{sec:results}, we define our model and the associated $x_c$ threshold, and state our main results. The proof is given in two stages: First for models with $x_c = +\infty$ (in Section \ref{sec:infinite_xc}), then for models with $x_c < +\infty$ (in Section \ref{SecReg}) by approximation using $x_c = +\infty$ models. In Section \ref{SecComplex} we outline the minor adjustments necessary if $H_N$ is complex Hermitian instead of real symmetric. 

In Appendix \hyperlink{sec:talagrand}{A}, we give a straightforward extension of Talagrand's classic results on concentration for product measures, but one that we were unable to find in the literature: His results are written for independent random variables valued in $[-1,1]$, and we extend his results to independent random variables valued in the $d$-dimensional unit ball for any $d$. The quality of the estimate does not depend on $d$ (essentially because the radius of this ball does not depend on $d$), which may be of independent interest. The particular case $d = 2$ allows us to consider complex random matrices, whose upper-triangular entries are independent of one another, but whose real and imaginary parts are simply \emph{uncorrelated}, rather than \emph{independent} as required in previous results. Finally, in Appendix \hyperlink{sec:compact_or_log_sobolev}{B} we show by example how to remove the ``compact-or-log-Sobolev'' assumption from works in the literature, and in Appendix \hyperlink{sec:wigner}{C} we state and prove our results on deformed Wigner matrices.


\subsection{Notation.}\

We write the Lipschitz and bounded-Lipschitz norms of a function $f : \R \to \R$, respectively, as
\begin{align*}
	\|f\|_{\textup{Lip}} &= \sup_{x \neq y} \abs{\frac{f(x)-f(y)}{x-y}}, \\
	\|f\|_{\mc{L}} &= \|f\|_{\textup{Lip}} + \|f\|_\infty.
\end{align*}
We will need the function classes
\[
	\mc{F}_{\textup{Lip}} = \{f : \R \to \R : \|f\|_{\mc{L}} \leq 1\}
\]
and, for a given compact set $\mc{K} \subset \R$,
\[
	\mc{F}_{\textup{Lip},\mc{K}} = \{f \in \mc{F}_{\textup{Lip}} : \supp(f) \subset \mc{K}\}.
\]
We will also need the bounded-Lipschitz, Wasserstein-$1$, and Kolmogorov-Smirnov distances between probability measures on $\R$, defined respectively by
\begin{equation}
\label{eqn:distances}
\begin{split}
	d_{\textup{BL}}(\mu,\nu) &= \sup_{f \in \mc{F}_{\textup{Lip}}} \abs{ \int_\R f(x) (\mu - \nu)(\diff x)}, \\
	{\rm W}_1(\mu,\nu) &= \sup_{f : \|f\|_{\textup{Lip}} \leq 1} \abs{ \int_\R f(x)(\mu - \nu)(\diff x)}, \\
	d_{\textup{KS}}(\mu,\nu) &= \sup\{\abs{\mu((-\infty,x]) - \nu((-\infty,x])} : x \in \R\},
\end{split}
\end{equation}
the first of which metrizes weak convergence. If the probability measure $\mu$ is compactly supported, we write $G_\mu$ for its Stieltjes transform with the sign convention 
\[
	G_\mu(z) = \int \frac{\mu(\diff \lambda)}{z-\lambda},
\]
and write $\ell(\mu)$ and $r(\mu)$ for its left and right endpoints, respectively.

We introduce the Dyson index $\beta$ as a shorthand for the symmetry class under consideration, either real-symmetric ($\beta = 1$) or complex-Hermitian ($\beta = 2$). 

If $T$ is a matrix, we write its operator norm (with respect to Euclidean distance) as $\|T\|$, its Frobenius norm as $\|T\|_F^2 = \sum_{i,j} |T_{ij}|^2$, and its trace norm as $\|T\|_\ast = \sum_i \sigma_i(T)$. If $T$ is square and $N \times N$, we number its eigenvalues as 
\[
	\lambda_{\textup{min}}(T) = \lambda_1(T) \leq \lambda_2(T) \leq \cdots \lambda_N(T) = \lambda_{\textup{max}}(T),
\]
sometimes dropping the dependence on $T$ from the notation, and write its empirical spectral measure as
\begin{equation}
\label{eqn:muhat}
	\hat{\mu}_{T} = \frac{1}{N} \sum_{i=1}^N \delta_{\lambda_i(T)}.
\end{equation}
We write $\Leb(\cdot)$ for the Lebesgue measure, and in the appendices we will need the semicircle law normalized as
\[
	\rho_{\textup{sc}}(\diff x) = \frac{\sqrt{(4-x^2)_+}}{2\pi} \diff x.
\]


\subsection{Acknowledgements.}\

We wish to thank Alice Guionnet for proposing the problem and for many helpful discussions. We also thank Yizhe Zhu for directing us to the reference \cite{Zhi2021}. This material is based upon work supported by the National Science Foundation under Grant No. DMS-1928930 while the authors were in residence at the Mathematical Sciences Research Institute in Berkeley, California, during the Fall 2021 semester.


\section{Results}
\label{sec:results}


\subsection{Set-up.}\
\label{subsec:setup}

Theorem~\ref{thm:main}, our main result in the real case, is stated in Section~\ref{subsec:main}; the complex analog, Theorem~\ref{thm:mainc}, is given in Section~\ref{subsec:complex}. In order to state them properly, we need to precisely define both our model, in the present Section~\ref{subsec:setup}, as well as the limit $\sigma$ of its empirical measure and a particular function $\overline{G}_\sigma$, which is a sort of ``second branch'' of the Stieltjes transform needed to define the rate function for our LDP, in Section~\ref{subsec:dyson}. In Section~\ref{subsec:degenerate} we discuss certain degenerate cases. 

We recall from \eqref{eqn:modeldef} our fundamental random matrix
\[
	H_N = \frac{1}{M} Z^T \Gamma Z = \frac{1}{M} \sum_{i=1}^M d_i z_i z_i^T,
\]
which we now define precisely. Fix $\alpha > 0$, and choose a sequence $M = (M_N)_{N=1}^\infty$ such that
\[
	\lim_{N \to \infty} \frac{M_N}{N} = \alpha.
\]
For each $M$, we consider a deterministic real $M \times M$ matrix 
\[
	\Gamma = \Gamma_M = \diag(d_1, \ldots, d_M),
\]
where $(d_1, \ldots, d_M) = (d_1^{(M)}, \ldots, d_M^{(M)})$ are ordered without loss of generality as $d_1 \leq \cdots \leq d_M$. We suppose there exists a compactly supported probability measure $\rho$ such that
\[
	\hat{\mu}_{\Gamma_M} \to \rho \quad \text{weakly.}
\]

We need the following definition.

\begin{defn}
\label{defn:ssg}
A centered probability measure $\mu$ on $\R$ is called \emph{sharp sub-Gaussian} if it has unit variance and
\[
	\int_\R e^{tx} \mu(\diff x) \leq e^{\frac{t^2}{2}} \quad \text{ for all } t \in \R.
\]
\end{defn}
Standard Gaussian measure is sharp sub-Gaussian, of course, but so are the Rademacher law $\frac{1}{2}(\delta_{-1}+\delta_{+1})$ and the uniform law on $[-\sqrt{3},\sqrt{3}]$. We fix some sharp sub-Gaussian measure $\mu$, and let $Z \in \R^{M \times N}$ be a matrix with i.i.d. entries distributed according to $\mu$. To match the notation in the literature, we take the \emph{rows} of $Z$, which are independent vectors of length $N$, and write them in \emph{column} form as $\{z_i\}_{i=1}^M$.

The following assumption will be made throughout the paper, although we will only state it explicitly in the main theorem.

\begin{assn}
\label{assn:no_outliers}
We assume that $\Gamma$ has asymptotically no (external) outliers, in the sense that
\begin{align*}
	\lambda_{\textup{min}}(\Gamma) &\to \ell(\rho), \\
	\lambda_{\textup{max}}(\Gamma) &\to r(\rho).
\end{align*}
\end{assn}
We allow internal outliers, in the sense that, if $\rho$ has disconnected support, $\Gamma$ can have eigenvalues in the gaps between the components. 


\subsection{The Dyson equation and the limit of the empirical measure.}\
\label{subsec:dyson}

The following two thresholds will be important:
\begin{defn}
Let
\begin{align}
	\theta_{\textup{max}} &\defeq \begin{cases} \frac{\alpha}{r(\rho)} & \text{if } r(\rho) > 0, \\ +\infty & \text{otherwise.} \end{cases} \label{eqn:def_theta_max} \\
	x_c(\rho) &\defeq \begin{cases} r(\rho)^2G_\rho(r(\rho)) + \left(\frac{1}{\alpha}-1\right)r(\rho) & \text{if } r(\rho) > 0 \text{ and } G_\rho(r(\rho)) < \infty, \\ +\infty & \text{otherwise} \end{cases} \label{eqn:def_xc_rho}
\end{align}
\end{defn}

With such a setup, the following convergence theorem for the empirical measure $\hat{\mu}_{H_N}$ is essentially classical:

\begin{thm}
\label{thm:sigma}
With $H_N$ as above, the sequence of empirical measures $(\hat{\mu}_{H_N})_{N=1}^\infty$ converges weakly in probability toward a compactly supported measure $\sigma$. Furthermore, the Stieltjes transform of $\sigma$ on $(r( \sigma), + \infty)$ satisfies 
\begin{equation}
\label{eqn:dyson}
H_{\rho}( G_{\sigma}(x)) = x,
\end{equation}
where $H_{\rho}$ is defined on $(0, \theta_{\max})$ by 
\begin{equation}
\label{eqn:mp}
	H_{\rho}(y) \defeq \frac{1}{y} + \int_{\R} \frac{\alpha u}{\alpha - yu} \rho(\diff u).
\end{equation}
If $r(\rho) \leq 0$, then $r(\sigma) \leq 0$. (However, it is possible to have $r(\sigma) < 0$ when $r(\rho) > 0$; see Remark \ref{rem:r_sig<0,r_rho>0} for an example.)
\end{thm}
\begin{proof}
Theorem 1.1 in \cite{SilBai1995} gives every claim, other than that the measure $\sigma$ is compactly supported, and that $r(\sigma) \leq 0$ when $r(\rho) \leq 0$. To show the former, note that, since $\rho$ is compactly supported, then $H_\rho$ is meromorphic in a neighborhood of zero, with a simple pole at zero. From the inverse function theorem, $z \mapsto \frac{1}{H_\rho(z)}$ is therefore holomorphic in a neighborhood of zero, so (via \eqref{eqn:mp}) $G_{\sigma}$ is holomorphic in a neighborhood of infinity, meaning $\sigma$ is indeed compactly supported. To show the latter, note that we can always choose $\Gamma$ without outliers, meaning in such a way that $\ell(\hat{\mu}_{\Gamma}) = \ell(\rho)$ and $r(\hat{\mu}_{\Gamma}) = r(\rho)$ at finite $N$, and that for $H_N$ defined with such $\Gamma$ we have $H_N \leq r(\rho) \leq 0$, so that $r(\sigma) \leq 0$.
\end{proof}

\begin{rem}
\label{rem:free_probability}
We will not need it, but the measure $\sigma$ can be interpreted in the language of free probability. Indeed, when $\rho$ is supported on $\R^+$ or $\R^-$, then $\sigma$ is just the free multiplicative convolution of $\rho$ and the corresponding Mar\v{c}enko-Pastur law. The free multiplicative convolution is usually defined only for measures supported on a half-line, but when $\rho$ has two nontrivial components $\rho^\pm(A) = \rho(A \cap \R^\pm)$, we can write $\sigma$ as the free \emph{additive} convolution of the measures $\sigma^{\pm}$, which are respectively the free \emph{multiplicative} convolutions of $\rho^\pm$ with the Mar\v{c}enko-Pastur law.

Indeed, writing $\rho = \rho^+ + \rho^-$ induces a decomposition $\Gamma = \Gamma^+ + \Gamma^-$ into positive and negative $d_i$'s, and thus a decomposition $H_N = H_N^+ + H_N^-$ with $H_N^\pm = \frac{1}{M} Z^T \Gamma^{\pm} Z$. The matrices $H_N^+$ and $H_N^-$ are independent, and their empirical measures tend to $\sigma^\pm$, respectively. This is true regardless of the underlying law of $Z$, which we will choose to be Gaussian; then, due to rotation invariance of the Gaussian law, they have the same joint distribution as $H_N^+$ and $OH_N^-O^T$, where $O$ is a Haar orthogonal/unitary matrix independent of everything else; and since {\color{blue} asymptotically} these matrices are \emph{freely} independent, the empirical measure of their sum tends to the free \emph{additive} convolution of the empirical measures.
\end{rem}

\begin{rem}
\label{rem:r_sig<0,r_rho>0}
Note that we can have $r(\rho) >0$ but $r(\sigma)<0$. For instance, let us choose $M_N=4N$ and let us take 
\[
	\Gamma_N = \text{diag}(\underbrace{2, \dots,2}_{ 2 N \text{ times }}, \underbrace{-2K, \dots,-2K}_{ 2 N \text{ times }})
\]
where $K$ is a constant we will fix later. Then we have for this model $\rho = \frac{1}{2} ( \delta_{-2K} + \delta_2)$. Remark \ref{rem:free_probability} explains that, in this case, $\sigma$ is the free \emph{additive} convolution of two measures, namely the free \emph{multiplicative} convolutions of delta masses at $2$ and $-2K$, respectively, with Mar\v{c}enko-Pastur. Since the free multiplicative convolution in this case just rescales the measures, $\sigma$ is the free additive convolution of two stretched Mar\v{c}enko-Pasturs, one near $2$ and one near $-2K$. The factors of two here are so that the Mar\v{c}enko-Pastur laws are gapped away from zero; since $r(\mu \boxplus \nu) \leq r(\mu) + r(\nu)$ in general, by taking $K$ sufficiently large, we can thus force $r(\sigma) < 0$.
\end{rem}

The equation \eqref{eqn:dyson} is called a \emph{Dyson equation}. The following lemma defines a certain function $\overline{G}_\sigma$, which should be thought of as a second branch of the Stieltjes transform. Its proof will be given in Section \ref{subsec:overline_g}.

\begin{lem}
\label{lem:overline_g}
First, if $x_c(\rho)$ is finite then $x_c(\rho) = H_\rho(\theta_{\textup{max}}) \defeq \lim_{y \uparrow \theta_{\textup{max}}} H_\rho(y)$, and $x_c(\rho) \geq r(\sigma)$. Second, except in the case when $r(\rho) \leq 0$ and $r(\sigma) = 0$ (this degenerate case is handled in Section \ref{subsec:degenerate}), there exists a real-valued, continuous function $\overline{G}_\sigma$, defined on the domain 
\begin{equation}
\label{eqn:domain_d}
	D \defeq \begin{cases} [r(\sigma),+\infty) & \text{if } r(\rho) > 0, \\ [r(\sigma),0) & \text{if } r(\rho) \leq 0, \end{cases}
\end{equation}
with the following properties:
\begin{enumerate}
\item If $x \in D$ and $x \leq x_c(\rho)$, then $H_\rho(\overline{G}_\sigma(x)) = x$, and $\{w : H_\rho(w) = x\} = \{G_\sigma(x), \overline{G}_\sigma(x)\}$.
\item We have $\overline{G}_\sigma(x) > G_\sigma(x)$ for $x > r(\sigma)$, and $\overline{G}_\sigma(r(\sigma)) = G_\sigma(r(\sigma))$.
\item $\overline{G}_\sigma$ is real analytic on $D \setminus \{x_c(\rho)\}$.
\item $\overline{G}_\sigma$ is nondecreasing on $D$, and more specifically
\begin{itemize}
\item If $r(\rho) > 0$ and $x_c(\rho) = +\infty$, then $\overline{G}_\sigma$ is strictly increasing on $D$, with $\lim_{x \to +\infty} \overline{G}_\sigma(x) = \theta_{\textup{max}}$.
\item If $r(\rho) > 0$ and $x_c(\rho) < +\infty$, then $\overline{G}_\sigma$ is strictly increasing on $(r(\sigma),x_c(\rho))$, with $\lim_{x \uparrow x_c(\rho)} \overline{G}_\sigma(x) = \theta_{\textup{max}}$, and $\overline{G}_\sigma(x) = \theta_{\textup{max}}$ for $x \geq x_c(\rho)$.
\item If $r(\rho) < 0$, then $\overline{G}_\sigma$ is strictly increasing on $D$, with $\lim_{x \uparrow 0} \overline{G}_\sigma(x) = +\infty$.
\end{itemize}
\end{enumerate}
\end{lem}


\subsection{Degenerate cases.}\
\label{subsec:degenerate}

For our purposes, there are two possibilities for degenerate behavior. We first explain them informally:
\begin{enumerate}
\item If $\rho(\{0\}) > 0$, this means that one is writing
\[
	H_N = \frac{1}{M} \sum_{i=1}^N d_iz_iz_i^T
\]
where a macroscopic fraction of the $d_i$'s are (asymptotically) zero. Our results do apply to this case as written, but for the sake of completeness, we check below that the rate function remains the same if one instead discards these zero $d_i$'s, which amounts to removing the zero atom from $\rho$, renormalizing to keep it a probability measure, and adjusting $\alpha$ correspondingly. 
\item The case $r(\rho) \leq 0$ and $r(\sigma) = 0$ is degenerate, since then $\lambda_{\textup{max}}(H_N)$ is essentially trapped at zero: On the one hand, $\lambda_{\textup{max}}(H_N)$ cannot push into the bulk at this scale, so it must be asymptotically nonnegative. On the other hand $H_N$ is (asymptotically almost) negative semidefinite, so $\lambda_{\textup{max}}(H_N)$ must be asymptotically nonpositive. This is expressed precisely in a degenerate LDP.
\end{enumerate}

We now formalize the results just explained. All the claims here are proved in Section \ref{subsec:r_sigma_0}.

\begin{lem}
\label{lem:r_sigma_0}
If $\rho$ is a measure on $\R$ such that $r(\rho) \leq 0$ and $\rho(\{0 \}) = 0$, then the following are equivalent:
\begin{itemize}
\item $r(\sigma) = 0$,
\item $\alpha \leq 1$. 
\end{itemize}
\end{lem}
\begin{rem}\label{rem:r_sigma_0}
If $\rho(\{0\}) > 0$, then we claim that the following two procedures give the same result (meaning the same $\sigma$ and the same rate function): (a) applying our results to $\rho$ as written, or (b) creating a new measure $\tau$ that removes the spike at zero, changes $\alpha$ to some $\alpha'$, and considers a corresponding $H'_N$.

Indeed, if $\rho(\{ 0 \}) = \beta >0$ we can write $\rho = (1 - \beta) \rho' + \beta \delta_0$ where $\rho'$ is a measure with $r(\rho) \leq 0$ and $\rho'( \{0 \})=0$. Then it is easy to see that $H_N= H'_N + H''_N$ where  
	$H_N'= Z'^T \Gamma_N' Z'$ and $H_N''= Z''^T \Gamma_N'' Z''$ where $\Gamma'_N$ is a $M'_N \times M'_N$ matrix with $\lim_{N} \frac{M'_N}{N} = \alpha(1 - \beta)$ and where the empirical measure of $\Gamma'_N$ converges toward $\rho'$ and $||\Gamma''_N||$ converges to $0$. So for $\epsilon >0$, $\lim \frac{1}{N}\log \P[ ||H_N''|| \geq \epsilon] = - \infty$. This reduces the problem to stating a large deviation principle for $H_N'$, and furthermore we proved that $r(\sigma)=0$ if and only if $\alpha(1- \beta) \leq 1$ (we state this result in the following corollary). With $\Delta_N \defeq \frac{M'_N}{M_N} \Gamma'_N$, we can rewrite $H_N'$ as $H'_N = \frac{1}{M'_N} Z'^T \Delta_N Z'$. Since the empirical measure of $\Delta_N$ converges toward $\tau \defeq (1- \beta)* \rho'$ and we can apply our main result, Theorem $\ref{thm:main}$, to $H_N'$ with $\tau$ instead of $\rho$ and $\alpha' \defeq \alpha(1 - \beta)$ instead of $\alpha$. It remains to show then that the statement of Theorem \ref{thm:main} remains unchanged when we change $\rho$ into $\rho'$ and $\alpha$ into $\alpha'$. For this we need only to show that the functions $H_{\rho}$ and $H_{\tau}$ are the same, since we will see that the rate function of the large deviation principle only depends on $\rho$ through $H_{\rho}$. Indeed,
	\begin{align*}
		H_{\rho}(y) &= \frac{1}{y} + \int \frac{ \alpha u}{ \alpha - yu } \rho(du) =  \frac{1}{y} + (1- \beta) \int \frac{ \alpha u}{ \alpha - yu } \rho'(du) \\
		&= \frac{1}{y} + (1- \beta) \int \frac{ \alpha (1- \beta)^{-1} u}{ \alpha -(1- \beta)^{-1}  yu } \tau(du) = \frac{1}{y} +  \int \frac{(1- \beta) \alpha  u}{ (1- \beta)\alpha -  yu } \tau(du) = H_{\tau}(y).
	\end{align*}
	Therefore the rate function we obtain by applying Theorem \ref{thm:main} to $H'_N$ is the same as what we obtain by applying Theorem \ref{thm:main} to $H_N$.
\end{rem}
\begin{cor}\label{cor:r_sigma_0}
	If $\rho$ is a measure on $\R$ such that $r(\rho) \leq 0$, then the following are equivalent:
	\begin{itemize}
		\item $r(\sigma) = 0$,
		\item $\alpha ( 1 - \rho( \{0 \})) \leq 1$. 
	\end{itemize}
\end{cor}

\begin{defn}
\label{defn:degenerate}
We will summarize the (equivalent) conditions of Corollary \ref{cor:r_sigma_0} by saying that the pair $(\rho,\alpha)$ is \emph{degenerate}.
\end{defn}

This definition is justified by the following (straightforward) result.

\begin{prop}
\label{prop:degenerate}
Define $I^{\textup{degen}} : \R \to [0,+\infty]$ by
\[
	I^{\textup{degen}}(x) = \begin{cases} 0 & \text{if } x = 0, \\ +\infty & \text{otherwise}. \end{cases}
\]
If $\mu$ is Gaussian, $\rho(\{0\}) = 0$, and the pair $(\rho,\alpha)$ is degenerate, then $\lambda_{\textup{max}}(H_N)$ satisfies a large deviation principle at speed $N$ with the good rate function $I^{\textup{degen}}$. 
\end{prop}

In the following, we will always assume that the pair $(\rho,\alpha)$ is nondegenerate. 


\subsection{Main result in the real setting.}\
\label{subsec:main}

\begin{defn}
Suppose the pair $(\rho,\alpha)$ is nondegenerate. With $D$ as in \eqref{eqn:domain_d}, $G_\sigma$ the Stieltjes transform as usual, and $\overline{G}_\sigma$ the ``second branch'' of the Stieltjes transform from Lemma \ref{lem:overline_g}, define $I_\sigma : D \to [0,+\infty]$ by
\[
	I_\sigma(x) = \begin{cases} \frac{\beta}{2}\int_{r(\sigma)}^x [\overline{G}_\sigma(u) - G_\sigma(u)] \diff u & \text{if } x \in D, \\ +\infty & \text{otherwise.} \end{cases}
\]
\end{defn}

Our main theorem holds under either of the following two assumptions, recalling that $\mu$ is the common distribution of the entries of $Z$.

\begin{assnA}
\label{assn:ssg}
The measure $\mu$ is sharp sub-Gaussian, and the support of $\rho$ is in $[0,\infty)$.
\end{assnA}

\begin{assnA}
\label{assn:g}
The measure $\mu$ is Gaussian.
\end{assnA}

\begin{thm}
\label{thm:main}
\textbf{(Main theorem, real version)} 
If Assumption \ref{assn:no_outliers} holds, the pair $(\rho,\alpha)$ is nondegenerate (in the sense of Definition \ref{defn:degenerate}), and also either Assumption \ref{assn:ssg} or Assumption \ref{assn:g} holds, then $\lambda_{\textup{max}}(H_N)$ satisfies a large deviation principle at speed $N$ with the good rate function $I_{\sigma}$. This function is convex and strictly increasing on $D$ (in particular, it vanishes uniquely at $r(\sigma)$). If additionally $r(\rho) > 0$, then
\[
	\lim_{x \to +\infty} \frac{I_\sigma(x)}{x} = \frac{\theta_{\textup{max}}}{2}.
\]
\end{thm}

Theorem \ref{thm:main} will follow from the following two results.

\begin{prop}
\label{prop:main_infinite}
If Assumption \ref{assn:no_outliers} holds, the pair $(\rho,\alpha)$ is nondegenerate, and also either Assumption \ref{assn:ssg} or Assumption \ref{assn:g} holds, and
\[
	x_c(\rho) = +\infty,
\]
then $\lambda_{\textup{max}}(H_N)$ satisfies a large deviation principle at speed $N$ with the good rate function $I_{\sigma}$.
\end{prop}

\begin{prop}
\label{prop:main_finite}
Proposition \ref{prop:main_infinite} implies Theorem \ref{thm:main}.
\end{prop}

\begin{rem}
Propositions \ref{prop:main_infinite} and \ref{prop:main_finite} have fairly different proofs from one another. Proposition \ref{prop:main_infinite} is proved by tilting with spherical integrals; for every $x < x_c(\rho)$, we are able to find appropriate tilt that makes the deviations $\{\lambda_{\textup{max}} \approx x\}$ likely. But Proposition \ref{prop:main_finite} uses neither explicit tilting, nor almost anything else in the proof details of Proposition \ref{prop:main_infinite}; instead it goes by approximating models with $x_c(\rho) < +\infty$ using models with $x_c(\rho) = +\infty$, and textbook results on obtaining LDPs by taking limits in a sequence of approximating LDPs.
\end{rem}

\begin{rem}
\label{rem:ssg_negative_di}
In the non-Gaussian case, the requirement that $\rho$ be supported in $(0,\infty)$ is not just technical. When some $d_i$'s are negative, the rate function should likely be different from the Gaussian case. Indeed, suppose all the $d_i$'s are equal to $-d$ for some $d > 0$. Then of course $\lambda_{\textup{max}}(-\frac{d}{M}Z^TZ) = -\lambda_{\textup{min}}(\frac{d}{M}Z^TZ)$, but the left-hand deviations of the smallest eigenvalue of a Rademacher covariance matrix are not the same as the Gaussian analogue; the Rademacher rate function must be finite at zero since $\P(\text{two columns of $Z$ agree}) \geq 2^{-M}$.
\end{rem}


\subsection{Main result in the complex setting.}\
\label{subsec:complex}

Our main result also translates to the complex setting. In this, though, we need to take the entries of $Z$ to be i.i.d. distributed according the some centered probability measure $\mu$ on the complex plane such that $\int (\Im z)^2 \mu(dz) = \int (\Re z)^2 \mu(dz) = 1/2$ and $\int \Im z \Re z \mu(dz)=0$. Our model then becomes
\[ H_N = \frac{1}{M} Z^* \Gamma Z \]
where $Z^*$ denotes the Hermitian conjugate of $Z$. All the other assumptions on $M=M_N$ and $\Gamma$ remain the same. 

We can extend Definition \ref{defn:ssg} to complex random variables: 

\begin{defn}
	\label{defn:ssgc}
	A centered probability measure $\mu$ on $\C$ is called \emph{sharp sub-Gaussian in $\C$}  if for $X$ $\mu$-distributed, the random vector $(\Re X, \Im X)$ has covariance matrix $\frac{1}{2}\left(\begin{smallmatrix} 1 & 0 \\ 0 & 1 \end{smallmatrix}\right)$ (the real and imaginary parts must be uncorrelated, but do not have to be independent) and
	\[
	\int_\R e^{\Re (w \overline{z})} \mu(\diff w) \leq e^{\frac{|z|^2}{4}} \quad \text{ for all } z \in \C.
	\]
\end{defn}

We need then to slightly update Assumption \ref{assn:ssg} to our complex setting:
\begin{assnA}
	\label{assn:ssgc}
	The measure $\mu$ is sharp sub-Gaussian in $\C$, and the support of $\rho$ is in $[0,\infty)$.
\end{assnA}
Then we have for this model an LDP exactly identical to the one we obtain in the real case, except the rate function we have here is twice the rate function of  the real case.
\begin{thm} 
	\label{thm:mainc}
	\textbf{(Main theorem, complex version)} 
	If Assumption \ref{assn:no_outliers} holds, the pair $(\rho,\alpha)$ is nondegenerate, and also either Assumption \ref{assn:g} or Assumption \ref{assn:ssgc} holds, then $\lambda_{\textup{max}}(H_N)$ satisfies a large deviation principle at speed $N$ with the good rate function $ 2 I_{\sigma}$. 
\end{thm}
Since the proof remains mostly the same up to some tweaks in the computations of our spherical integrals, we will simply the adjustments we need to make in Section \ref{SecComplex}.


\section{Proof for infinite \texorpdfstring{$x_c$}{xc}}
\label{sec:infinite_xc}


\subsection{Outline of the proof.}\
\label{subsec:outline}
In this section, we prove Proposition \ref{prop:main_infinite}. The proof goes by introducing a variational formulation of the rate function, which is more opaque but technically more convenient, and then proving the LDP with this variational formulation via tilting by spherical integrals. The broad sketch of the proof in this section resembles previous works, but as explained in the introduction, new arguments allow us to bypass several technical assumptions present in previous works. 

\begin{defn}
\label{def:j}
For $\mu$ a compactly supported probability measure on $\R$, $\theta \geq 0$, and $\lambda \geq r(\mu)$, let
\begin{align*}
	v(\mu,\theta,\lambda) &\defeq \begin{cases} \lambda - \frac{1}{2\theta} & \text{if } G_\mu(\lambda) \leq 2\theta, \\ G_\mu^{-1}(2\theta) - \frac{1}{2\theta} & \text{if } G_\mu(\lambda) \geq 2\theta \geq 0, \end{cases} \\
	J(\mu,\theta,\lambda) &\defeq \theta v(\mu,\theta,\lambda) - \frac{1}{2} \int_\R \log (1+2\theta v(\mu,\theta,\lambda) - 2\theta y) \mu(\diff y),
\end{align*}
\end{defn}

\begin{defn}
\label{def:IF}
For $0 \leq \theta < \theta_{\textup{max}}$, let
\begin{align*}
	F(\rho,\theta) &\defeq -\frac{\alpha}{2} \int_\R \log\left(1-\frac{\theta t}{\alpha}\right) \rho(\diff t), \\
	I_\sigma(x,\theta) &\defeq J\left(\sigma,\frac{\theta}{2},x\right) - F(\rho,\theta).
\end{align*}
Using these, define $\widetilde{I_\sigma} : \R \to [0,+\infty]$ by 
\[
	\widetilde{I_\sigma}(x) \defeq \begin{cases} \sup_{0 \leq \theta < \theta_{\textup{max}}} I_\sigma(x,\theta) & \text{if } x \in D, \\ +\infty & \text{otherwise}. \end{cases}
\]
\end{defn}

\begin{lem}
\label{lem:ratesimp}
\textbf{(Simplification of the rate function)} 
If $x_c(\rho) = +\infty$, we have
\[
    I_\sigma(x) = \widetilde{I_\sigma}(x),
\]
and this function is concave on $[r(\sigma),\infty)$ and vanishes uniquely at $x = r(\sigma)$.
\end{lem}
\begin{proof}
Let us first look at the first case, that is, $r( \rho) > 0$ and $G_{\rho}( r( \rho))= + \infty$. We compute
\[
	\frac{\partial}{\partial \theta} J\left( \sigma, \frac{\theta}{2},x\right) = \begin{cases} \frac{G_{\sigma}^{-1}( \theta)}{2} - \frac{1}{2 \theta} & \text{ if } \theta \leq G_{\sigma}(x) \\ \frac{ x}{2} - \frac{1}{2 \theta} & \text{ if } \theta \geq G_{\sigma}(x) \end{cases}
\]
We also compute
\[
	\frac{\partial}{\partial \theta} F(\rho, \theta) = \frac{1}{2} \int_{\R} \frac{\alpha u }{ \alpha - \theta u} d \rho(u) = \frac{H_\rho(\theta)}{2} - \frac{1}{2\theta}.
\]
Using the Dyson equation \eqref{eqn:dyson}, we have
\[
	\frac{\partial}{\partial \theta} I_{\sigma}(x,\theta) = \begin{cases} 0 & \text{if } \theta \leq G_{\sigma}(x), \\
		\frac{ x}{2} - \frac{H_{\rho}(\theta)}{2 } & \text{if } \theta \geq G_{\sigma}(x). \end{cases}
\]
Therefore the function $I_{\sigma}(x,\cdot)$ is increasing on $[G_{\sigma}(x), \overline{G}_{\sigma}(x)]$ and decreasing on $[\overline{G}_{\sigma}(x), + \infty)$. Thus $\theta_x \defeq \overline{G}_{\sigma}(x)$ is the optimizing $\theta$ value, i.e., 
\[
	\widetilde{I_\sigma}(x) = I_\sigma(x,\theta_x).
\]
(When $x = r(\sigma)$, we define $\theta_{r(\sigma)} \defeq \theta_c$ by convention, and note that this argument shows $\widetilde{I_\sigma}(r(\sigma)) = 0$.) This lets us compute the derivative as
\[
	\frac{\diff}{\diff x} \widetilde{I_\sigma}(x) = \left. \frac{\partial}{\partial x} I_\sigma(x,\theta) \right|_{\theta = \theta_x} = \frac{1}{2}(\theta_x - G_\sigma(x)) = \frac{1}{2}(\overline{G}_\sigma(x) - G_\sigma(x)).
\]
Since $I_\sigma$ and $\widetilde{I_\sigma}$ have the same derivative and agree at $x = r(\sigma)$, the claim follows.
	
The case with $r( \sigma) < 0$ is essentially similar for values of $x$ between $r( \sigma)$ and $0$. 
\end{proof}

\begin{rem}
This is not necessary for our proof, but we note that when $x_c(\rho) < +\infty$, i.e. $r(\rho) > 0$ and $G_\rho(r(\rho)) < \infty$, one can check that $\int_\R \log(r(\rho) - t)\rho(\diff t) > -\infty$, hence one can make sense of $F(\rho,\theta_{\textup{max}})$ and $I_\sigma(x,\theta_{\textup{max}})$. In this case, one can define $I^\dagger_\sigma : \R \to [0,+\infty]$ by
\[
	I^\dagger_\sigma(x) \defeq \begin{cases} I_\sigma(x,\overline{G}_\sigma(x)) & \text{if } x \geq r(\sigma) \\ +\infty & \text{otherwise} \end{cases} = \begin{cases} \sup_{0 \leq \theta < \theta_{\textup{max}}} I_\sigma(x,\theta) & \text{if } r(\sigma) \leq x < x_c(\rho), \\ I_\sigma(x,\theta_{\textup{max}}) & \text{if } x \geq x_c(\rho), \\ +\infty & \text{otherwise}. \end{cases}
\]
The point of this remark is that one can check 
\[
	I_\sigma(x) = I^\dagger_\sigma(x).
\]
Indeed, the proof of Lemma \ref{lem:ratesimp} already checked this for $x < x_c(\rho)$. For $x \geq x_c(\rho)$, one can compute
\[
	\partial_x I_\sigma(x,\theta_{\textup{max}}) = \partial_x J(\sigma,\theta_{\textup{max}}/2,x) = \frac{1}{2}(\theta_{\textup{max}} - G_\sigma(x)) = \frac{1}{2}(\overline{G}_{\sigma}(x) - G_\sigma(x)),
\]
meaning that $I_\sigma$ and $I^\dagger_\sigma$ have the same derivative. However, the definition of $I_\sigma$ is technically more convenient, so we use it in the proof. 
\end{rem}

Assumptions \ref{assn:no_outliers} and \ref{assn:ssg} require $\supp(\rho) \subset [0,\infty)$ but permit a handful of negative $d_i$'s at finite $N$, which must tend to zero in the limit of large dimension. However, it is technically more convenient to work with the case when all $d_i$'s are nonnegative at finite $N$. The following result allows us to restrict to this case; we omit its proof, since it is essentially the same as that of Proposition \ref{prop:degenerate}. 
\begin{lem}
Under Assumptions \ref{assn:no_outliers} and \ref{assn:ssg}, define $\Gamma^+ = \diag(\max(d_1,0),\ldots,\max(d_M,0))$, and $H_N^+ = M^{-1}Z^T \Gamma^+ Z$. Then $\lambda_{\textup{max}}(H_N)$ and $\lambda_{\textup{max}}(H^+_N)$ are exponentially equivalent. In particular, LDPs for one automatically hold for the other.
\end{lem}
In the remainder, we tacitly replace $H_N$ with $H_N^+$ when necessary.

\begin{lem}
\label{lem:exponential_tightness}
\textbf{(Exponential tightness)}
Under either Assumption \ref{assn:ssg} or Assumption \ref{assn:g}, we have
\[
    \lim_{K \to \infty} \limsup_{N \to \infty} \frac{1}{N} \log \P(\abs{\lambda_{\textup{max}}(H_N)} > K) = -\infty.
\]
\end{lem}
\begin{proof}
If $M$ is large enough that $d_i \in (l(\rho)-1, r(\rho)+1)$ for all $i$, then it is elementary that
\[
    (l(\rho)-1) \lambda_{\textup{max}}(M^{-1}Z^TZ) \leq \lambda_{\textup{max}}(H_N) \leq (r(\rho)+1) \lambda_{\textup{max}}(M^{-1} Z^TZ)
\]
and thus
\[
    \abs{\lambda_{\textup{max}}(H_N)} \leq \max(\abs{l(\rho)-1},\abs{r(\rho)+1}) \lambda_{\textup{max}}(M^{-1} Z^TZ).
\]
Exponential tightness of $\lambda_{\max{}}(M^{-1}Z^TZ)$ was proved in \cite[Lemma 1.9]{GuiHus2020}.
\end{proof}

\begin{defn}\label{def:tilt}
\textbf{(Tilted measures)}
For $\theta \geq 0$, consider the tilted measure $\P^{\theta}(Z)$ on $M \times N$ matrices with density
\[
	\frac{\diff \P^\theta}{\diff \P}(Z) = \frac{\E_e[e^{N\frac{\theta}{2}\ip{e,\frac{1}{M}Z^T\Gamma Ze}}]}{\E_{e,H_N}[e^{N\frac{\theta}{2}\ip{e,H_Ne}}]}.
\]
\end{defn}

\begin{prop}
\label{prop:wkldpub}
\textbf{(Weak LDP upper bound for tilted measures)}
For $0 \leq \theta < \theta_{\textup{max}}$,
\[
	\limsup_{\delta \downarrow 0} \limsup_{N \to \infty} \frac{1}{N} \log \P^\theta(\abs{\lambda_{\textup{max}}(H_N) - x} \leq \delta) \begin{cases} \leq -(\widetilde{I_\sigma}(x) - I_\sigma(x,\theta)) & \text{if } x \in D, \\ = -\infty & \text{otherwise.} \end{cases}
\]
Notice that $\P^0 = \P$, and $I_\sigma(x,0) = 0$, so in particular we have the weak LDP upper bound for the measure we care about.
\end{prop}

\begin{prop}
\label{prop:wkldplb}
\textbf{(Weak LDP lower bound)}
\[
    \liminf_{\delta \downarrow 0} \liminf_{N \to \infty} \frac{1}{N} \log \P(\abs{\lambda_{\textup{max}}(H_N) - x} < \delta) \geq -\widetilde{I_\sigma}(x).
\]
\end{prop}

Proposition \ref{prop:main_infinite} follows in the classical way from Lemma \ref{lem:ratesimp}, Lemma \ref{lem:exponential_tightness}, Proposition \ref{prop:wkldpub}, and Proposition \ref{prop:wkldplb}. 

\begin{rem}
Actually, the proof given in this section is slightly more general than just stated: It works ``up to $x_c(\rho)$'' in the sense that it shows
\[
	-I_\sigma(x) \leq \liminf_{\delta \downarrow 0} \liminf_{N \to +\infty} \frac{1}{N} \log \P(\abs{\lambda_{\textup{max}}(H_N) - x} \leq \delta) \leq \limsup_{\delta \downarrow 0} \limsup_{N \to +\infty} \frac{1}{N} \log \P(\abs{\lambda_{\textup{max}}(H_N) - x} \leq \delta) \leq -I_\sigma(x)
\]
whenever $x < x_c(\rho)$, not just whenever $x_c(\rho) = +\infty$ and $x \in \R$ as stated. But since our final result holds regardless of $x_c(\rho)$, we will not need this level of generality here.
\end{rem}


\subsection{Annealed spherical integral.}\
\label{subsec:annealed}

The goal of this subsection is to prove the following lemma.

\begin{lem}
\label{lem:annealed}
Under either Assumption \ref{assn:ssg} or Assumption \ref{assn:g}, for every $0 \leq \theta < \theta_{\textup{max}}$, we have
\begin{equation}
\label{eqn:annealed}
	\lim_{N \to \infty} \frac{1}{N} \log \E_{e,H_N}[e^{N\frac{\theta}{2}\ip{e,H_Ne}}] = F(\rho,\theta).
\end{equation}
\end{lem}
\begin{proof}
For every unit vector $e$, we have 
\[
	\E_{H_N}[e^{N\frac{\theta}{2}\ip{e,H_Ne}}] = \E[e^{N\frac{\theta}{2}\ip{e,\left(\frac{1}{M}\sum_{\mu=1}^M d_\mu z_\mu z_\mu^T\right) e}}] = \E\left[ \prod_{\mu = 1}^M e^{\frac{N}{M}\frac{\theta}{2} d_\mu \ip{z_{\mu},e}^2} \right] = \prod_{\mu=1}^M \E[e^{\frac{N}{M}\frac{\theta}{2} d_\mu \ip{z_\mu,e}^2}]
\]
since the $z_\mu$'s are independent. Applying a Hubbard-Stratonovich transformation, we find
\[
	\E_{H_N}[e^{N\frac{\theta}{2}\ip{e,H_Ne}}] = \prod_{\mu=1}^M \frac{1}{\sqrt{\pi}} \int_{-\infty}^\infty \E\left[e^{2x \ip{z_\mu,e} \sqrt{\frac{N}{M}\frac{\theta}{2} d_\mu}} \right] e^{-x^2} \diff x,
\]
where we interpret $\sqrt{d_\mu} = \ii \sqrt{-d_\mu}$ for those $d_\mu$ which are negative.

In the Gaussian case (i.e., under Assumption \ref{assn:g}), the remainder of the proof is easy: We can exactly compute $\prod_{j=1}^N \E[e^{x\sqrt{2\frac{N}{M} \theta d_\mu} (z_\mu)_je_j}] = e^{x^2\frac{N}{M}\theta d_\mu}$, giving 
\[
	\frac{1}{N} \log \E_{e,H_N}[e^{N\frac{\theta}{2}\ip{e,H_Ne}}] = -\frac{M}{2N} \int_\R \log\left(1-\frac{N}{M}\theta t\right) \hat{\mu}_{\Gamma}(\diff t).
\]
The limiting replacement of $\hat{\mu}_{\Gamma}$ with $\rho$, and of $\frac{M}{N}$ with $\alpha$, is routine, but it requires $\theta < \theta_{\textup{max}}$. 

In the sharp sub-Gaussian case (i.e., under Assumption \ref{assn:ssg}), the upper bound is similar: For every \emph{real} $c$ and every unit $e$ we have
\[
	\E[\exp(c \ip{z_\mu,e})] = \prod_{j=1}^N \E[\exp(c (z_\mu)_j e_j)] \leq \prod_{j=1}^N \exp\left(\frac{c^2e_j^2}{2}\right) = \exp\left(\frac{c^2}{2}\right).
\]
Since each $d_\mu$ is positive, we can apply this with real $c = 2\sqrt{\frac{N}{M}\frac{\theta}{2}d_\mu}$ to find 
\[
	\frac{1}{N} \log \E_{e,H_N}[e^{N\frac{\theta}{2}\ip{e,H_Ne}}] \leq -\frac{M}{2N} \int_\R \log\left(1-\frac{N}{M}\theta t\right) \hat{\mu}_{\Gamma}(\diff t),
\]
which finishes the proof of the upper bound. For the lower bound, fix $0 < \epsilon < \frac{1}{4}$ and define
\[
	V^\epsilon_N = \{e : \|e\|_\infty \leq N^{-\frac{1}{4}-\epsilon}\} \subset \mathbb{S}^{N-1}.
\]
From the sharp sub-Gaussian assumption, we know that for every $\delta > 0$ there exists $\eta > 0$ such that, for any $\abs{t} \leq \eta$,
\[
	\int e^{tx} \mu(\diff x) \geq e^{\frac{(1-\delta)}{2}t^2}.
\]
In particular, if we fix $\delta$ so small that $f(t) = 1-\frac{(1-\delta)\theta t}{\alpha}$ is bounded below on the support of $\rho$ and write 
\[
	\eta' = \frac{\eta}{\sqrt{2(\alpha+1)\theta(r(\rho)+1)}},
\]
then whenever $e \in V^\epsilon_N$ and $\abs{x} \leq \eta' N^{\frac{1}{4}+\epsilon}$, for each $\mu$ we have
\begin{align*}
	\E\left[\exp\left(2x\ip{z_\mu,e}\sqrt{\frac{N}{M}\frac{\theta}{2} d_\mu}\right)\right] &\geq \exp\left((1-\delta)x^2\frac{N}{M}\theta d_\mu\right).
\end{align*}
Thus for such $e$ we have
\[
	\E[e^{\frac{N}{M} \frac{\theta}{2} d_\mu \ip{z_\mu,e}^2}] \geq \frac{1}{\sqrt{\pi}} \int_{-\eta' N^{\frac{1}{4}+\epsilon}}^{\eta'N^{\frac{1}{4}+\epsilon}} \exp\left(-\left(1-(1-\delta)\frac{N}{M}\theta d_\mu\right)x^2\right) \diff x.
\]
From standard Gaussian tail bounds, if $c,d > 0$ are independent of $N$ then
\[
	\frac{1}{\sqrt{\pi}} \int_{dN^{\frac{1}{4}+\epsilon}}^\infty \exp(-cx^2) \diff x \leq \sqrt{c}\exp(-2cd^2N^{\frac{1}{2}+2\epsilon})
\]
so that
\begin{align*}
	\E[e^{\frac{N}{M}\frac{\theta}{2} d_\mu\ip{z_\mu,e}^2}] &\geq \left(1-(1-\delta)\frac{N}{M}\theta d_\mu\right)^{-\frac{1}{2}} - 2\sqrt{1-\frac{N}{M}\theta d_\mu}\exp\left(-2\left(1-\frac{N}{M}\theta d_\mu\right)(\eta')^2 N^{\frac{1}{2}+2\epsilon}\right) \\
	&\geq \left(1-\frac{(1-\delta)\theta d_\mu}{\alpha}\right)^{-\frac{1}{2}}\left(1-C\abs{\frac{N}{M}-\frac{1}{\alpha}}\right) - C'\exp(-C''N^{\frac{1}{2}+2\epsilon})
\end{align*}
for some constants $C, C', C''$ depending on $\delta$ through $\eta$, but not depending on $\mu$. Therefore
\begin{align*}
	\frac{1}{N}\log \E_{e,H_N}[e^{N\frac{\theta}{2}\ip{e,H_Ne}}] \geq& \, \frac{M}{N} \int_\R \underbrace{\log\left[ \left(1-\frac{(1-\delta)\theta t}{\alpha}\right)^{-\frac{1}{2}}\left(1-C\abs{\frac{N}{M}-\frac{1}{\alpha}}\right) - C'\exp(-C''N^{\frac{1}{2}+2\epsilon})\right] }_{=:f_N(t)} \rho(\diff t) \\
	&+ \frac{1}{N} \log \P(e \in V^\epsilon_N)
\end{align*}
From \cite[Lemma 3.3]{GuiHus2020}, we have $\P(e \in V^\epsilon_N) \to 1$. Furthermore, $f_N(t) \to -\frac{1}{2}\log (1-\frac{(1-\delta)\theta t}{\alpha})$ pointwise as $N \to \infty$, and is bounded above on the support of $\rho$ from our choice of $\delta$, so dominated convergence gives
\[
	\liminf_{N \to \infty} \frac{1}{N} \log \E_{e,H_N}[e^{N\frac{\theta}{2}\ip{e,H_Ne}}] \geq -\frac{\alpha}{2} \int_\R \log\left(1-\frac{(1-\delta)\theta t}{\alpha}\right) \rho(\diff t).
\]
Letting $\delta \downarrow 0$ with dominated convergence again finishes the proof.
\end{proof}


\subsection{Concentration of measure.}\

The goal of this subsection is to prove the following proposition.
\begin{prop}
\label{prop:sub_Gaussian_concentration}
If either Assumption \ref{assn:ssg} or Assumption \ref{assn:g} holds, then for every $\epsilon > 0$ we have
\begin{equation}
\label{eqn:sub_Gaussian_concentration}
    \lim_{N \to \infty} \frac{1}{N} \log \P(d_{\textup{BL}}(\hat{\mu}_{H_N},\sigma) > \epsilon) = -\infty.
\end{equation}
\end{prop}
\begin{rem}
Shortly prior to the posting of our paper, we learned that similar arguments to those in this subsection and in Appendix \hyperlink{sec:compact_or_log_sobolev}{B} will also appear in independent, soon-to-be-posted work of Cook, Ducatez, and Guionnet \cite{CoDuGuTBD}.
\end{rem}
\begin{proof}
The structure of the proof is common between the different cases of Assumption \ref{assn:ssg} and Assumption \ref{assn:g}. However, one technical estimate is proved quite differently for the different cases; we will mention this at the appropriate moment below, and otherwise tacitly treat both cases simultaneously.

Consider the decomposition
\[
    Z = A + B,
\]
where
\begin{equation}
\label{eqn:han}
    A_{ij} = Z_{ij} \mathds{1}_{\abs{Z_{ij}} \leq N^{\gamma}}
\end{equation}
for some $\gamma = \gamma(\epsilon) > 0$ to be chosen. (Recall the $Z_{ij}$'s are order one, so $B$ is typically sparse.) Define the matrix
\[
    H_N^A = \frac{1}{M} A^T\Gamma A
\]
and, for large positive $L$, the event
\[
    \mc{E}_L = \{\supp(\hat{\mu}_{H_N}) \subset (-L,L)\} \cap \{\supp(\hat{\mu}_{H_N^A}) \subset (-L,L)\}.
\]
To prove \eqref{eqn:sub_Gaussian_concentration}, it suffices to check
\begin{align}
	\lim_{N \to \infty} \frac{1}{N} \log \P(d_{\textup{BL}}(\hat{\mu}_{H_N}, \hat{\mu}_{H_N^A}) > \epsilon, \mc{E}_L) &= -\infty \quad \text{for every } L > 100r(\sigma), \label{eqn:truncating} \\
	\lim_{L \to \infty} \limsup_{N \to \infty} \frac{1}{N} \log \P(\mc{E}_L^c) &= -\infty, \label{eqn:elc} \\
	\lim_{N \to \infty} \frac{1}{N} \log \P(d_{\textup{BL}}(\hat{\mu}_{H_N^A}, \E[\hat{\mu}_{H_N^A}]) > \epsilon) &= -\infty,\label{eqn:guizei} \\
	\lim_{N \to \infty} d_{\textup{BL}}(\E[\hat{\mu}_{H_N^A}], \sigma) &= 0. \label{eqn:deterministic}
\end{align}
We start with \eqref{eqn:truncating}, which we prove by adapting arguments of Bordenave, Caputo, and Chafa\"{i} (namely \cite[Lemma C.2]{BorCapCha2011} and \cite[Lemma 2.2]{BorCap2014}). Whenever $f$ is a $C^1$ test function with $\|f\|_{\textup{Lip}} + \|f\|_\infty \leq 1$, integration by parts gives
\begin{align*}
	\abs{\int f(\lambda) (\hat{\mu}_{H_N} - \hat{\mu}_{H_N^A})(\diff \lambda)}\mathbf{1}_{\mc{E}_L} &= \abs{\int f'(\lambda) (F_{\hat{\mu}_{H_N}} - F_{\hat{\mu}_{H_N^A}})(\diff \lambda)} \mathbf{1}_{\mc{E}_L} \leq \|f'\|_\infty \|F_{\hat{\mu}_{H_N}} - F_{\hat{\mu}_{H_N^A}}\|_1 \mathbf{1}_{\mc{E}_L} \\
	&\leq 2L \|f\|_{\textup{Lip}} \|F_{\hat{\mu}_{H_N}} - F_{\hat{\mu}_{H_N^A}}\|_\infty \leq 2Ld_{\textup{KS}}(\hat{\mu}_{H_N}, \hat{\mu}_{H^A_N}).
\end{align*}
If $f$ just has $\|f\|_{\textup{Lip}} + \|f\|_\infty \leq 1$ but is not necessarily $C^1$, there is a $C^1$ function $g$ with $\|g\|_{\textup{Lip}} + \|g\|_\infty \leq 1$ and $\|f-g\|_{L^\infty([-L,L])} \leq 2Ld_{\textup{KS}}(\hat{\mu}_{H_N}, \hat{\mu}_{H_N^A})$; thus 
\[
	\P(d_{\textup{BL}}(\hat{\mu}_{H_N}, \hat{\mu}_{H_N^A}) > \epsilon, \mc{E}_L) \leq \P\left(d_{\textup{KS}}(\hat{\mu}_{H_N}, \hat{\mu}_{H_N^A}) > \frac{\epsilon}{4L}\right).
\]
It is classical (a consequence of interlacing of singular values, see e.g. \cite[Theorem A.44]{BaiSil2010}) that
\[
    d_{\textup{KS}}(\hat{\mu}_{H_N},\hat{\mu}_{H_N^A}) \leq \frac{1}{N}\rank(Z - A) = \frac{1}{N}\rank(B) \leq \frac{1}{N} \sum_{i,j} \mathds{1}_{\abs{Z_{ij}} > N^\gamma}
\]
for any $\Gamma$. Now $(\mathds{1}_{\abs{Z_{ij}} > N^{\gamma}})_{1 \leq i \leq M, 1 \leq j \leq N}$ is a collection of $NM$ i.i.d. Bernoulli variables with mean
\[
	p_N := \P(\abs{Z_{ij}} > N^{\gamma}) \leq \exp(-cN^{2\gamma})
\]
for some $c$ depending on the sub-Gaussian norm of $\mu$. Writing
\[
	\sigma^2 = NM p_N(1-p_N) \leq \exp\left(-\frac{c}{2}N^{2\gamma}\right)
\]
and $h(x) = (x+1)\log(x+1)-x$, Bennett's inequality \cite{Ben1965} gives 
\[
	\P\left( \sum_{i \leq j} \mathds{1}_{\abs{Z_{ij}} > N^\gamma} - NM p_N \geq t \right) \leq \exp\left(-\sigma^2 h\left(\frac{t}{\sigma^2}\right)\right)
\]
for any $t > 0$. We will choose $t = N \frac{\epsilon}{4 L} - NM p_N \geq  N \frac{\epsilon}{8L}$, which has $\frac{t}{\sigma^2} \to +\infty$; since $h(x) \geq x \log x$ for sufficiently large arguments, we have
\[
	\P( d_{\textup{BL}}(\hat{\mu}_{H_N}, \hat{\mu}_{H_N^A}) > \epsilon, \mc{E}_L) \leq \exp\left(-t\log\left(\frac{t}{\sigma^2}\right)\right).
\]
Since $\gamma > 0$, this suffices for \eqref{eqn:truncating}.

The estimate \eqref{eqn:elc} is essentially a consequence of exponential tightness, Lemma \ref{lem:exponential_tightness}, which controls $\abs{\lambda_{\textup{max}}(H_N)}$; the same proof controls $\abs{\lambda_{\textup{min}}(H_N)}$ (which can be negative if some $d_i$'s are negative), as well as the extreme eigenvalues of $H_N^A$ (the proof of Lemma \ref{lem:exponential_tightness} references \cite[Lemma 1.9]{GuiHus2020}, which goes through for $A$).

The verification of \eqref{eqn:guizei} is fairly different in the sharp sub-Gaussian case (Assumption \ref{assn:ssg}) vs. the Gaussian case (Assumption \ref{assn:g}). The latter case is Lemma \ref{lem:herbstconcentration} below, while the former case is Lemma \ref{lem:guizeiexpansion} below (this is also where we select $\gamma$ as a function of $\epsilon$).

For \eqref{eqn:deterministic}, we first claim
\begin{equation}
\label{eqn:deterministic_between_matrices}
    \lim_{N \to \infty} d_{\textup{BL}}(\E[\hat{\mu}_{H_N^A}], \E[\hat{\mu}_{H_N}]) = 0.
\end{equation}
It suffices to show $\E[X_N] \to 0$, where $X_N = d_{\textup{BL}}(\hat{\mu}_{H_N^A},\hat{\mu}_{H_N})$. But $X_N$ is a random variable between zero and two, and \eqref{eqn:truncating} and \eqref{eqn:elc} show that for every $\epsilon > 0$ and $N \geq N_0(\epsilon)$ we have $\P(X_N > \epsilon) \leq \exp(-100N)$, which shows $\E[X_N] \to 0$ and hence \eqref{eqn:deterministic_between_matrices}. It remains only to show
\[
    \lim_{N \to \infty} d_{\textup{BL}}(\E[\hat{\mu}_{H_N}],\sigma) = 0.
\]
This is equivalent to the claim $\E[\hat{\mu}_{H_N}] \to \sigma$. The original result of this type is due to Mar\v{c}enko and Pastur \cite{MarPas1967}, but for the model where the $d_i$'s are i.i.d. draws from $\rho$, instead of being deterministic with the property $\frac{1}{M} \sum \delta_{d_i} \to \rho$; the result for our, latter variant of this model is due to Silverstein and Bai \cite{SilBai1995} (actually, their result holds in greater generality).
\end{proof}

\begin{lem}
\label{lem:herbstconcentration}
Under Assumption \ref{assn:g}, for each $\epsilon > 0$, there exists $\gamma = \gamma(\epsilon) > 0$ such that, if $H^A_N$ is defined in terms of $\gamma$ using \eqref{eqn:han}, then 
\begin{equation}
\label{eqn:herbstconcentration}
    \lim_{N \to \infty} \frac{1}{N} \log \P(d_{\textup{BL}}(\hat{\mu}_{H^A_N},\E[\hat{\mu}_{H^A_N}]) > \epsilon) = -\infty.
\end{equation}
\end{lem}
\begin{proof}
In this case, it is technically inconvenient that the truncation of $A$ is discontinuous; thus we further decompose
\[
	A = C + D,
\]
where
\[
	C_{ij} = \begin{cases} Z_{ij} & \text{if } \abs{Z_{ij}} \leq N^\gamma, \\ N^\gamma \text{sign}(Z_{ij}) & \text{if } \abs{Z_{ij}} > N^\gamma \end{cases}, \qquad D_{ij} = \begin{cases} 0 & \text{if } \abs{Z_{ij}} \leq N^\gamma, \\ -N^\gamma \text{sign}(Z_{ij}) & \text{if } \abs{Z_{ij}} > N^\gamma \end{cases},
\]
and define the matrix
\[
	H^C_N = \frac{1}{M} C^T \Gamma C.
\]
To prove \eqref{eqn:herbstconcentration}, it suffices to check
\begin{align}
	\lim_{N \to \infty} \frac{1}{N} \log \P(d_{\textup{BL}}(\hat{\mu}_{H^A_N},\hat{\mu}_{H^C_N}) > \epsilon, \mc{E}_L) &= -\infty \quad \text{for every } L > 100 r(\sigma), \label{eqn:shifting_c} \\
	\lim_{N \to \infty} \frac{1}{N} \log \P(d_{\textup{BL}}(\hat{\mu}_{H^C_N}, \E[\hat{\mu}_{H^C_N}]) > \epsilon) &= -\infty, \label{eqn:herbst_c} \\
	\lim_{N \to \infty} d_{\textup{BL}}(\E[\hat{\mu}_{H^C_N}],\E[\hat{\mu}_{H^A_N}]) &= 0. \label{eqn:deterministic_c}
\end{align}
The proof of \eqref{eqn:shifting_c} is a close copy of the proof of \eqref{eqn:truncating}. Namely, one shows in the same way as before that
\[
	\P(d_{\textup{BL}}(\hat{\mu}_{H^A_N}, \hat{\mu}_{H^C_N}) > \epsilon, \mc{E}_L) \leq \P\left(d_{\textup{KS}}(\hat{\mu}_{H^A_N}, \hat{\mu}_{H^C_N}) > \frac{\epsilon}{4L}\right) \leq \P\left( \frac{1}{N} \text{rank}(D) > \frac{\epsilon}{4L}\right).
\]
Earlier, we bounded the rank of $B$ by its number of nonzero entries. Here we do the same for $D$, but by construction $B$ and $D$ have the same number of nonzero entries, so the rest of the estimate is exactly the same. Similarly, the proof of \eqref{eqn:deterministic_c} is analogous to the proof of \eqref{eqn:deterministic_between_matrices}.

So it remains only to show \eqref{eqn:herbst_c}, and we start by showing 
\begin{equation}
\label{eqn:gaussian_concentration_one_f}
	\sup_{f \in \mc{F}_{\textup{Lip}}} \P\left( \abs{\int_\R f(\lambda) (\hat{\mu}_{H^C_N} - \E[\hat{\mu}_{H^C_N}])(\diff \lambda)} \geq \delta\right) \leq 2\exp\left(-\frac{\delta^2N^{2-2\gamma}}{2}\right).
\end{equation}
Indeed, we shift perspective slightly by defining $C : \R^{M \times N} \to \R^{M \times N}$ as
\[
	C(Z)_{ij} = \begin{cases} Z_{ij} & \text{if } \abs{Z_{ij}} \leq N^\gamma, \\ N^\gamma \text{sign}(Z_{ij}) & \text{if } \abs{Z_{ij}} > N^\gamma, \end{cases}
\]
and $H^{C(Z)}_N$ as $H^{C(Z)}_N = M^{-1} C(Z)^T \Gamma C(Z)$. Fix some $f \in \mc{F}_{\textup{Lip}}$, and consider the map $h = h_f : \R^{M \times N} \to \R$ defined by
\[
	h(Z) = \int_{\R} f(\lambda) \hat{\mu}_{H^{C(Z)}_N} (\diff \lambda).
\]
We want to show that $h$ is Lipschitz. Using an $L^p$ version of the Hoffman-Wielandt inequality with $p = 1$ (see e.g. \cite[Theorem II]{Kat1987}), and writing $\pi \in S_N$ for a permutation in the permutation group on $N$ letters, we have
\begin{align*}
	\abs{h(Z_1) - h(Z_2)} &\leq \min_{\pi \in S_N} \frac{1}{N} \sum_{i=1}^N \abs{f(\lambda_i(H^{C(Z_1)}_N)) - f(\lambda_{\pi(i)}(H^{C(Z_2)}_N))} \leq \min_{\pi \in S_N} \frac{1}{N} \sum_{i=1}^N \abs{\lambda_i(H^{C(Z_1)}_N) - \lambda_{\pi(i)}(H^{C(Z_1)}_N)} \\
	&\leq \frac{1}{N} \sum_{i=1}^N \abs{\lambda_i(H^{C(Z_1)}_N - H^{C(Z_2)}_N)} = \frac{1}{N} \|H^{C(Z_1)}_N - H^{C(Z_2)}_N\|_\ast \\
	&\leq \frac{1}{NM} (\|C(Z_1)^T\Gamma (C(Z_1)-C(Z_2))\|_\ast + \|(C(Z_1)-C(Z_2))^T\Gamma C(Z_2)\|_\ast)
\end{align*}
Recalling that, for a matrix $M$, $\|M\|_F= \sqrt{ \sum_{i,j} |M_{ij}|^2}$ denotes its Frobenius norm, we will now use that $\|T_1T_2\|_\ast \leq \|T_1\|_F \|T_2\|_F$; that $\|C(Z_1)^T\Gamma\|_F \leq \sqrt{ NM} d_{\textup{max}}N^{\gamma}$, since the entries have magnitude at most $d_{\textup{max}}N^\gamma$; and that 
\[
	\|C(Z_1)-C(Z_2)\|_F \leq \|Z_1-Z_2\|_F
\]
(this estimate is why we needed to define $C$; the analogue for $A$ is not true). These give
\[
	\abs{h(Z_1)-h(Z_2)} \leq 2d_{\textup{max}} \frac{N^\gamma}{\sqrt{NM}} \|Z_1-Z_2\|_F \leq \frac{4d_{\textup{max}}}{\sqrt{\alpha}} \frac{N^\gamma}{N} \|Z_1-Z_2\|_F.
\]
Since Gaussian measure on $\R^{M \times N} \cong \R^{MN}$ satisfies the log-Sobolev inequality with constant $1$, the Herbst argument then gives
\[
	\P\left( \abs{\int_\R f(\lambda) (\hat{\mu}_{H^C_N} - \E[\hat{\mu}_{H^C_N}])(\diff \lambda)} \geq \delta\right) = \P(\abs{h(Z) - \E[h(Z)]} \geq \delta) \leq 2\exp\left(-\frac{\delta^2N^{2-2\gamma}}{2}\right)
\]
which proves \eqref{eqn:gaussian_concentration_one_f}.

Now we want to upgrade by taking a supremum over $f$ inside the probability -- at first just over $\mc{F}_{\textup{Lip},\mc{K}}$, which we recall denotes the set of functions in $\mc{F}_{\textup{Lip}}$ supported in some compact set $\mc{K}$. Guionnet and Zeitouni give a very useful construction for this purpose (which we will also mimic in the proof fo the sub-Gaussian case below): For any $\Delta > 0$, they construct a set of $2(\lceil\frac{\abs{\mc{K}}}{\Delta}\rceil+1)$ functions $(h_k)_{k=1}^{2(\lceil\frac{\abs{\mc{K}}}{\Delta}\rceil+1)}$ in $\mc{F}_\textup{Lip}$ with the property that, for any given $f \in \mc{F}_{\textup{Lip},\mc{K}}$, one can choose $\lceil \frac{\abs{\mc{K}}}{\Delta} \rceil+1$ of the $h_k$'s whose sum, called $f_\Delta$, satisfies $\|f-f_\Delta\|_\infty \leq \Delta$. Since $f_\Delta$ is actually a sum of this finite count of functions, not a linear combination of them, we have
\begin{align*}
	&\P\left( \sup_{f \in \mc{F}_{\textup{Lip}, \mc{K}}} \abs{\int_\R f(\lambda) (\hat{\mu}_{H^C_N} - \E[\hat{\mu}_{H^C_N}])(\diff \lambda)} \geq \delta\right) \\
	&\leq 2 \left( \left\lceil \frac{\abs{\mc{K}}}{\Delta} \right\rceil+1\right) \max_{k=1}^{2\left(\lceil\frac{\abs{\mc{K}}}{\Delta}\rceil+1\right)}  \P\left( \abs{\int_\R h_k(\lambda) (\hat{\mu}_{H^C_N} - \E[\hat{\mu}_{H^C_N}])(\diff \lambda)} \geq \frac{\delta-2\Delta}{2(\lceil \abs{\mc{K}}/\Delta \rceil+1)}\right) \\
	&\leq 4 \left( \left\lceil \frac{\abs{\mc{K}}}{\Delta} \right\rceil+1\right) \exp\left(-\frac{N^{2-2\gamma}}{2} \left(\frac{\delta-2\Delta}{2\left(\lceil \abs{\mc{K}}/\Delta\rceil + 1\right)} \right)^2 \right)
\end{align*}
Like Guionnet and Zeitouni, we choose $\Delta = \delta/4$; then whenever $\delta < 1$ and $\abs{\mc{K}} > 1$, we have
\begin{equation}
\label{eqn:gaussian_concentration_some_fs}
	\P\left( \sup_{f \in \mc{F}_{\textup{Lip}, \mc{K}}} \abs{\int_\R f(\lambda) (\hat{\mu}_{H^C_N} - \E[\hat{\mu}_{H^C_N}])(\diff \lambda)} \geq \delta\right) \leq \frac{32\delta}{\abs{\mc{K}}} \exp\left(-\frac{N^{2-2\gamma}\delta^4}{128\abs{\mc{K}}^2}\right).
\end{equation}
Now we define, for $L > 0$, the event
\[
	\mc{E}_{C,L} = \{\supp(\hat{\mu}_{H^C_N}) \subset (-L,L)\},
\]
mimicking the event $\mc{E}_L$ from above. In the same way that we proved the $\mc{E}_L$ was likely in \eqref{eqn:elc}, one can prove
\begin{equation}
\label{eqn:eclc}
	\lim_{L \to \infty} \limsup_{N \to \infty} \frac{1}{N} \log \P(\mc{E}_{C,L}^c) = -\infty.
\end{equation}
On the other hand, fix large $L$. For any $f \in \mc{F}_{\textup{Lip}}$, there exists $\widetilde{f} \in \mc{F}_{\textup{Lip}}$ that agrees with $f$ on $(-L,L)$ and vanishes outside $(-2L,2L)$, say. Furthermore, on the event $\mc{E}_{C,L}$, the empirical measure $\hat{\mu}_{H^C_N}$ is supported on $(-L,L)$, although its expectation is not; thus with $\mc{K} := (-2L,2L)$ we have
\begin{align*}
	&\P\left( d_{\textup{BL}}(\hat{\mu}_{H^C_N}, \E[\hat{\mu}_{H^C_N}]) \geq \delta, \mc{E}_{C,L} \right) \\
	&= \P\left( \sup_{f \in \mc{F}_{\textup{Lip}}} \abs{\int_\R f(\lambda) (\hat{\mu}_{H^C_N} - \E[\hat{\mu}_{H^C_N}])(\diff \lambda)} \geq \delta, \mc{E}_{C,L} \right) \\
	&\leq \P\left( \sup_{\widetilde{f} \in \mc{F}_{\textup{Lip}, \mc{K}}} \abs{\int_\R \widetilde{f}(\lambda) (\hat{\mu}_{H^C_N} - \E[\hat{\mu}_{H^C_N}])(\diff \lambda)} \geq \frac{\delta}{2}, \mc{E}_{C,L} \right) + \P\left( \sup_{f \in \mc{F}_{\textup{Lip}}} \abs{\int_\R (f(\lambda)-\widetilde{f}(\lambda))\E[\hat{\mu}_{H^C_N}](\diff \lambda)} \geq \frac{\delta}{2} \right) \\
	&\leq \P\left( \sup_{\widetilde{f} \in \mc{F}_{\textup{Lip}, \mc{K}}} \abs{\int_\R \widetilde{f}(\lambda) (\hat{\mu}_{H^C_N} - \E[\hat{\mu}_{H^C_N}])(\diff \lambda)} \geq \frac{\delta}{2} \right) + \mathbf{1}\left\{ \int_\R \mathbf{1}_{\abs{\lambda} \geq L} \E[\hat{\mu}_{H^C_N}](\diff \lambda) \geq \frac{\delta}{4} \right\} \\
	&\leq \frac{8\delta}{L}\exp\left(-\frac{N^{2-2\gamma}\delta^4}{2048L^2}\right) + \mathbf{1}\left\{ \int_\R \mathbf{1}_{\abs{\lambda} \geq L} \E[\hat{\mu}_{H^C_N}](\diff \lambda) \geq \frac{\delta}{4} \right\}
\end{align*}
where we used \eqref{eqn:gaussian_concentration_some_fs} in the last line. To handle the indicator, we note
\[
	\int_L^\infty \E[\hat{\mu}_{H^C_N}](\diff \lambda) \leq \P(\lambda_{\textup{max}}(H^C_N) \geq L) \leq \P(\mc{E}_{C,N}^c) 
\]
and similarly for the left tail, \eqref{eqn:eclc} gives that the indicator vanishes for $L$ large enough depending on $\delta$; thus 
\[
	\P\left( \sup_{f \in \mc{F}_{\textup{Lip}}} \abs{\int_\R f(\lambda) (\hat{\mu}_{H^C_N} - \E[\hat{\mu}_{H^C_N}])(\diff \lambda)} \geq \delta, \mc{E}_{C,L} \right) \leq \frac{8\delta}{L}\exp\left(-\frac{N^{2-2\gamma}\delta^4}{2048L^2}\right) \quad \text{for all } L > L_0(\delta).
\]
Combined with \eqref{eqn:eclc}, this finishes the proof of \eqref{eqn:herbst_c}, and thus finishes the proof of Lemma \ref{lem:herbstconcentration}. 
\end{proof}

\begin{lem}
\label{lem:guizeiexpansion}
Under Assumption \ref{assn:ssg}, for each $\epsilon > 0$, there exists $\gamma = \gamma(\epsilon) > 0$ such that, if $H^A_N$ is defined in terms of $\gamma$ using \eqref{eqn:han}, then 
\[
    \lim_{N \to \infty} \frac{1}{N} \log \P(d_{\textup{BL}}(\hat{\mu}_{H^A_N},\E[\hat{\mu}_{H^A_N}]) > \epsilon) = -\infty.
\]
\end{lem}
\begin{proof}
This proof is essentially an exercise in filling out remarks made by Guionnet and Zeitouni \cite{GuiZei2000}, using their techniques. We will need several classes of test functions $f : \R \to \R$:
\begin{align*}
    \mc{F}_{\textup{Lip}} &= \left\{f : \|f\|_{\mc{L}} := \|f\|_\infty + \sup_{x \neq y} \frac{\abs{f(x) - f(y)}}{\abs{x-y}} \leq 1 \right\}, \\
    \text{for compact } \mc{K} \subset \R, \qquad \mc{F}_{\textup{Lip},\mc{K}} &= \{f \in \mc{F}_{\textup{Lip}} : \supp(f) \subset \mc{K}\}, \\
    \mf{C}_1 &= \{f : \text{ the function } x \mapsto g(x) := f(x^2) \text{ is Lipschitz and convex}\}
\end{align*}
and the corresponding norm
\[
    \|f\|_{\mf{C}_1} = \sup_{x \neq y} \abs{\frac{f(x^2)-f(y^2)}{x-y}}.
\]
Guionnet and Zeitouni \cite[Corollary 1.8]{GuiZei2000} showed that, for all $f \in \mf{C}_1$,
\[
    \P\left(\abs{\int_\R f(\lambda) (\hat{\mu}_{H^A_N}-\E[\hat{\mu}_{H^A_N}])(\diff \lambda)} \geq \delta \frac{N+M}{N}\right) \leq 4\exp\left(-\frac{1}{16d_{\textup{max}}N^{2\gamma}\|f\|_{\mf{C}_1}^2}(\delta-\delta_0(N+M))^2(N+M)^2\right)
\]
with $\delta_0(N+M) = \frac{8\sqrt{\pi d_{\textup{max}}}N^\gamma\|f\|_{\mf{C}_1}}{N+M}$.

Now fix compact $\mc{K} \subset \R$ and $\delta > 0$, and write $k_{\textup{min}} = \max(0,\min\{x : x \in \mc{K}\})$, $k_{\textup{max}} = \max(0,\max\{x : x \in \mc{K}\})$, and $\Delta = \frac{\delta}{8}$ (assume $\Delta < 1$). We will approximate test functions in $\mc{F}_{\textup{Lip},\mc{K}}$ by a finite linear combination of test functions in $\mf{C}_1$.

To do this, consider for any $a \geq 0$ the function
\[
    g_a(x) = \begin{cases} 0 & \text{if } x \leq a, \\ \frac{\Delta}{\sqrt{a+\Delta}-\sqrt{a}} (\sqrt{x}-\sqrt{a}) & \text{if } a \leq x \leq a + \Delta, \\ \Delta & \text{if } x \geq a + \Delta. \end{cases}
\]
Notice that $g_a$ is the difference of two functions in $\mf{C}_1$, namely $g_a = g_{a,1} - g_{a,2}$ where
\[
    g_{a,1}(x) = \begin{cases} 0 & \text{if } x \leq a, \\ \frac{\Delta}{\sqrt{a+\Delta}-\sqrt{a}} (\sqrt{x}-\sqrt{a}) & \text{if } x \geq a, \end{cases} \qquad g_{a,2}(x) = \begin{cases} 0 & \text{if } x \leq a+\Delta, \\ \frac{\Delta}{\sqrt{a+\Delta}-\sqrt{a}} (\sqrt{x}-\sqrt{a+\Delta}) & \text{if } x \geq a + \Delta, \end{cases}
\]
which have $\|g_{a,1}\|_{\mf{C}_1} = \|g_{a,2}\|_{\mf{C}_1} = \frac{\Delta}{\sqrt{a+\Delta}-\sqrt{a}}$.

Now let $f \in \mc{F}_{\textup{Lip},\mc{K}}$. We first note that, since $\hat{\mu}_{H^A_N}$ and $\E[\hat{\mu}_{H^A_N}]$ are supported on the right half-line as empirical measures of covariance matrices, if we define
\[
    \widetilde{f}(x) = (f(x)-f(0))\mathbf{1}\{x \geq 0\},
\]
then
\[
    \int_\R f(\lambda)(\hat{\mu}_{H^A_N} - \E[\hat{\mu}_{H^A_N}])(\diff \lambda) = \int_\R \widetilde{f}(\lambda)(\hat{\mu}_{H^A_N} - \E[\hat{\mu}_{H^A_N}])(\diff \lambda).
\]
Furthermore, $\widetilde{f}$ is supported in the right-half line, is $1$-Lipschitz, and while not in general compactly supported, is nevertheless constant for $x > k_{\textup{max}}$. We will actually not approximate $f$ directly, but rather approximate $\widetilde{f}$. Our approximating function will be
\[
    f_\Delta(x) = \sum_{i=0}^{\lceil \abs{\mc{K}}/\Delta \rceil} c_i^{(\Delta)} g_{k_{\textup{min}}+i\Delta}(x),
\]
where the coefficients $c_i^{(\Delta)}$ take values in $\{-1,+1\}$ and are defined recursively as follows:
\begin{itemize}
\item 
\[
    c_0^{(\Delta)} = \begin{cases} 1 & \text{if } \widetilde{f}(k_{\textup{min}} + \Delta) \geq 0, \\ -1 & \text{otherwise.} \end{cases}
\]
\item Given $c_{0,\Delta}, \ldots, c_{j,\Delta}$, let
\[
    c_{j+1}^{(\Delta)} = \begin{cases} 1 & \text{if } \widetilde{f}(k_{\textup{min}} + (j+2)\Delta) \geq \Delta \sum_{i=0}^j c_i^{(\Delta)} \\
    -1 & \text{otherwise.} \end{cases} 
\]
\end{itemize}
Notice that, on each interval 
\[
    I_j = [k_{\textup{min}} + j\Delta, k_{\textup{min}} + (j+1)\Delta],
\]
only the term $i = j$ in the sum defining $f_{\Delta}$ is non-trivial; the terms of smaller index are either $\Delta$ or $-\Delta$, and the terms of larger index vanish.

Now we claim 
\begin{equation}
\label{eqn:fDelta-approx}
    \|\widetilde{f}-f_\Delta\|_\infty \leq 2\Delta.
\end{equation}
Indeed, $\widetilde{f}$ and $f_\Delta$ both vanish on $(-\infty,k_{\textup{min}})$, and we will prove by induction on $j$ that 
\[
    A_j = \sup_{x \in I_j} \abs{\widetilde{f}(x) - f_\Delta(x)} \leq 2\Delta \quad \text{and} \quad B_j = \abs{\widetilde{f}(k_{\textup{min}}+(j+1)\Delta)-f_\Delta(k_{\textup{min}}+(j+1)\Delta)} \leq \Delta
\]
for $j = 0, \ldots, \lceil 2\abs{\mc{K}}/\Delta\rceil$. In the proof, we will regularly use the observation
\[
    f_{\Delta}(k_{\textup{min}}+j\Delta) = \Delta \sum_{i=0}^{j-1} c_i^{(\Delta)}.
\]
\begin{itemize}
\item The case $j = 0$ is simple, since, for $x \in I_0$, we have both $\abs{\widetilde{f}(x)} \leq \Delta$ (since $\widetilde{f}(k_{\textup{min}}) = 0$ and $\widetilde{f}$ is $1$-Lipschitz) and $\abs{f_\Delta(x)} = \abs{c_0^{(\Delta)}g_{k_{\textup{min}}}(x)} \leq \Delta$; this gives $A_0 \leq 2\Delta$, and $B_0 \leq \Delta$ from the definition of $c_0^{(\Delta)}$ and the $1$-Lipschitz property of $\widetilde{f}$.
\item Now suppose that $A_i \leq 2\Delta$ and $B_i \leq \Delta$ for $i = 0, \ldots, j-1$. For convenience let $c^{(\Delta)} = \sum_{i=0}^{j-2} c_i^{(\Delta)}$ if $j \geq 2$, or $c^{(\Delta)} = 0$ if $j = 1$. Notice that, whether $j = 1$ or $j \geq 2$, we have $f_{\Delta}(k_{\textup{min}}+j\Delta) = (c^{(\Delta)}+c_{j-1}^{(\Delta)})\Delta$, and that on $I_j$ we have $f_\Delta(x) = (c^{(\Delta)}+c_{j-1}^{(\Delta)})\Delta + c_j^{(\Delta)}g_{k_{\textup{min}}+j\Delta}(x)$. It will be useful to write $\ell_j = k_{\textup{min}}+j\Delta$ and $r_j = k_{\textup{min}}+(j+1)\Delta$ for the left and right endpoints, respectively, of the interval $I_j$. There are two cases:
\begin{itemize}
\item \textbf{Case 1 ($c_{j-1}^{(\Delta)}$ and  $c_j^{(\Delta)}$ are the same):} Assume $c_{j-1}^{(\Delta)} = c_j^{(\Delta)} = 1$, the other case being similar. This gives $f_\Delta(\ell_j) = (c^{(\Delta)}+1)\Delta$ and $f_\Delta(r_j) = (c^{(\Delta)}+2)\Delta$.

Suppose we can show that $\widetilde{f}$ lies in the ``correct windows'' at the endpoints, namely that $\widetilde{f}(\ell_j) \in [f_\Delta(\ell_j)-\Delta,f_\Delta(\ell_j)]$ and $\widetilde{f}(r_j) \in [f_\Delta(r_j)-\Delta,f_\Delta(r_j)]$. This would immediately show $B_j \leq \Delta$. Furthermore, since $\widetilde{f}$ is $1$-Lipschitz, this would tell us that $\widetilde{f}$ lies in some ``tube'' of ``height'' $\Delta$ connecting these windows; by comparing to the explicit formula for $f_\Delta$ on $I_j$ (it lies just above this tube), we would also obtain $A_j \leq 2\Delta$.

So we only need to understand $\widetilde{f}$ at the endpoints. But from the definition of $c_{j-1}^{(\Delta)}$ we find $\widetilde{f}(\ell_j) \geq c^{(\Delta)}\Delta = f_\Delta(\ell_j) - \Delta$. Since $B_j \leq \Delta$, this implies $\widetilde{f}(\ell_j) \leq f_\Delta(\ell_j)$. Since $\widetilde{f}$ is $1$-Lipschitz, this implies $\widetilde{f}(r_j) \leq \widetilde{f}(\ell_j) + \Delta \leq f_\Delta(\ell_j) + \Delta = f_\Delta(r_j)$. Next, from the definition of $c_j^{(\Delta)}$ we find that $\widetilde{f}(r_j) \geq (c^{(\Delta)}+1)\Delta = f_\Delta(r_j)-\Delta$, which finishes the localization of $\widetilde{f}$ in the correct windows at $x = \ell_j$ and $x = r_j$. 
\item \textbf{Case 2 ($c_{j-1}^{(\Delta)}$ and $c_j^{(\Delta)}$ are not the same):} Assume $c_{j-1}^{(\Delta)} = 1$ and $c_j^{(\Delta)} = -1$, the other case being similar. This gives $f_\Delta(\ell_j) = (c^{(\Delta)}+1)\Delta$ and $f_\Delta(r_j) = c^{(\Delta)}\Delta$.

Similar arguments to the first case localize $\widetilde{f}$ in slightly different windows at the endpoints: Here we find $\widetilde{f}(\ell_j) \in [f_\Delta(\ell_j) - \Delta, f_\Delta(\ell_j)]$ and $\widetilde{f}(r_j) \in [f_\Delta(r_j)-\Delta, f_\Delta(r_j)+\Delta]$. This immediately shows $B_j \leq \Delta$. The paths of all possible $1$-Lipschitz functions connecting these windows now make a pentagonal rather than a ``tube'' shape (if $\widetilde{f}(\ell_j) = f_\Delta(\ell_j)$, then $\widetilde{f}$ can increase on $[\ell_j,\ell_j+\Delta/2]$ before coming back down); but regardless, by comparing to the explicit formula for $f_\Delta$ on $I_j$, this still forces $A_j \leq 2\Delta$.
\end{itemize}
\end{itemize}
This shows $\sup_{x \leq k_{\textup{min}}+\Delta \lceil \abs{\mc{K}}/\Delta \rceil} \abs{\widetilde{f}(x) - f_\Delta(x)} \leq 2\Delta$. If $x \geq k_{\textup{min}}+\Delta\lceil \abs{\mc{K}}/\Delta\rceil \geq k_{\textup{min}}+\abs{\mc{K}}$, then $\widetilde{f}(x) = -f(0)$, and
\[
    \abs{f_\Delta(x)+f(0)} = \abs{f_\Delta(k_{\textup{min}} + \Delta\lceil\abs{\mc{K}}/\Delta \rceil)+f(0)} = \abs{f_\Delta(k_{\textup{min}} + \Delta\lceil\abs{\mc{K}}/\Delta \rceil) - \widetilde{f}(k_{\textup{min}} + \Delta\lceil\abs{\mc{K}}/\Delta \rceil)} \leq B_{\lceil \abs{\mc{K}}/\Delta \rceil} \leq \Delta
\]
by our work above. This completes the proof of \eqref{eqn:fDelta-approx}. 

Regardless of $f$, the approximation $f_\Delta$ is a linear combination of at most $2(\abs{\mc{K}}/\Delta+1)$ elements $h_k$ of $\mf{C}_1$, each of $\mf{C}_1$-norm at most $\frac{\Delta}{\sqrt{k_{\textup{max}}+\Delta}-\sqrt{k_{\textup{max}}}}$. Thus
\begin{align*}
    &\P\left( \sup_{f \in \mc{F}_{\textup{Lip}, \mc{K}}} \abs{\int f(\lambda)(\hat{\mu}_{H_N^A} - \E[\hat{\mu}_{H_N^A}])(\diff \lambda)} \geq \delta\right) \\
    &\leq 2\left(\frac{\abs{\mc{K}}}{\Delta}+1\right)\sup_k \P\left(\abs{\int h_k(\lambda)(\hat{\mu}_{H_N^A} - \E[\hat{\mu}_{H_N^A}])(\diff \lambda)} \geq \frac{\delta - 4\Delta}{2(\abs{\mc{K}}/\Delta+1)} \right) \\
    &\leq 8\left(\frac{\abs{\mc{K}}}{\Delta}+1\right)\exp\left( -\frac{1}{16d_{\textup{max}}N^{2\gamma} \frac{\Delta^2}{(\sqrt{k_{\textup{max}}+\Delta}-\sqrt{k_{\textup{max}}})^2}} \left( \frac{\delta-4\Delta}{2(\abs{\mc{K}}/\Delta+1)}\frac{N}{N+M} - \delta_0(N+M) \right)^2 (N+M)^2 \right)
\end{align*}
with $\delta_0(N+M) = \frac{8\sqrt{\pi d_{\textup{max}}} N^\gamma \frac{\Delta}{\sqrt{k_{\textup{max}}+\Delta}-\sqrt{k_{\textup{max}}}}}{N+M}$. Substituting $\Delta = \delta/8$ gives 
\begin{align*}
    &\P\left( \sup_{f \in \mc{F}_{\textup{Lip}, \mc{K}}} \abs{\int f(\lambda)(\hat{\mu}_{H_N^A} - \E[\hat{\mu}_{H_N^A}])(\diff \lambda)} \geq \delta\right) \\
    &\leq 8\left(\frac{8\abs{\mc{K}}}{\delta}+1\right)\exp\left( -\frac{4(\sqrt{k_{\textup{max}}+\delta/8} - \sqrt{k_{\textup{max}}})^2}{d_{\textup{max}}N^{2\gamma} \delta} \left( \frac{\delta}{4(8\abs{\mc{K}}/\delta+1)}\frac{N}{N+M} - \delta_0(N+M) \right)^2 (N+M)^2 \right).
\end{align*}

Now we approximate a general $f \in \mc{F}_{\textup{Lip}}$ by a compactly supported Lipschitz function $f^{L}$ for some large threshold $L$. Since
\[
    \frac{1}{N} \tr(H_N^A) = \frac{1}{NM} \sum_{i=1}^N \sum_{k=1}^M A_{ki}^2d_k \leq d_{\textup{max}}N^{2\gamma},
\]
for $L$ to be chosen we have $\int_{\abs{\lambda} > L} \hat{\mu}_{H_N^A}(\diff \lambda) \leq \frac{d_{\textup{max}}N^{2\gamma}}{L}$. If we define 
\[
    f^{L}(x) = \begin{cases} f(x) & \text{if } 0 \leq x \leq L, \\ f(L)-\text{sign}(f(L))(x-L) & \text{if } L \leq x \leq L + \abs{f(L)} \\ f(0) + \text{sign}(f(0))x & \text{if } -\abs{f(0)} \leq x \leq 0 \\ 0 & \text{otherwise} \end{cases}
\]
then
\[
    \abs{ \int f(\lambda) (\hat{\mu}_{H_N^A} - \E[\hat{\mu}_{H_N^A}])(\diff \lambda)} \leq \abs{\int f^{L}(\lambda) (\hat{\mu}_{H_N^A} - \E[\hat{\mu}_{H_N^A}])(\diff \lambda)} + \frac{2d_{\textup{max}}N^{2\gamma}}{L}. 
\] 
If we take $L = \frac{4d_{\textup{max}}N^{2\gamma}}{\epsilon}$ with $\epsilon$ from the statement of the lemma, then since $f^{L} \in \mc{F}_{\textup{Lip},\mc{K}}$ with, crudely, $\mc{K} = [-1,2L]$ we thus have
\begin{align*}
    &\P(d_{\textup{BL}}(\hat{\mu}_{H^A_N},\E[\hat{\mu}_{H^A_N}]) > \epsilon) = \P\left( \sup_{f \in \mc{F}_{\textup{Lip}}} \abs{\int f(\lambda)(\hat{\mu}_{H_N^A} - \E[\hat{\mu}_{H_N^A}])(\diff \lambda)} \geq \epsilon\right) \\
    &\leq \P\left( \sup_{f \in \mc{F}_{\textup{Lip}, \mc{K}}} \abs{\int f(\lambda)(\hat{\mu}_{H_N^A} - \E[\hat{\mu}_{H_N^A}])(\diff \lambda)} \geq \frac{\epsilon}{2} \right) \\
    &\leq 8\left(\frac{16\abs{\mc{K}}}{\epsilon}+1\right)\exp\left( -\frac{8(\sqrt{k_{\textup{max}}+\epsilon/16} - \sqrt{k_{\textup{max}}})^2}{d_{\textup{max}}N^{2\gamma} \epsilon} \left( \frac{\epsilon}{8(8\abs{\mc{K}}/\epsilon+1)}\frac{N}{N+M} - \delta_0(N+M) \right)^2 (N+M)^2 \right).
\end{align*}
This suffices by taking $\gamma$ sufficiently small.
\end{proof}


\subsection{Proof of the weak LDP upper bound}
\label{subsec:wk_ldp_ub}

\begin{defn}
Consider the following deterministic set of $M \times N$ real matrices:
\[
    \mc{A}_{x,\delta,\epsilon}^L = \left\{Z : \text{ with } H_N = H_N(Z) = \frac{1}{M}Z^T \Gamma Z, \quad \abs{\lambda_{\textup{max}}(H_N) - x} < \delta, \quad d_{\textup{BL}}(\hat{\mu}_{H_N}, \sigma) < \epsilon, \text{ and } \|H_N\| \leq L\right\}.
\]
\end{defn}

\begin{lem}
\label{lem:quenched}
For every $x \geq r(\sigma)$ and $0 \leq \theta < \theta_{\textup{max}}$, and every large enough $L$, we have
\begin{align*}
    &\lim_{\epsilon \downarrow 0} \limsup_{\delta \downarrow 0} \limsup_{N \to \infty} \sup_{Z \in \mc{A}^L_{x,\delta,\epsilon}} \abs{\frac{1}{N} \log \E_e[e^{N\frac{\theta}{2}\ip{e,H_Ne}}] - J\left(\sigma,\frac{\theta}{2},x\right)} \\
    &= \lim_{\delta \downarrow 0} \limsup_{\epsilon \downarrow 0} \limsup_{N \to \infty} \sup_{Z \in \mc{A}^L_{x,\delta,\epsilon}} \abs{\frac{1}{N} \log \E_e[e^{N\frac{\theta}{2}\ip{e,H_Ne}}] - J\left(\sigma,\frac{\theta}{2},x\right)} = 0.
\end{align*}
\end{lem}
\begin{proof}
This is an easy consequence of (stronger results from) \cite{GuiHus2022}. Fix (small) $t > 0$ and (large) $L > 0$; their Theorem 6.2 shows that there exists $N_0(t,L)$ such that for $N \geq N_0(t,L)$ we have
\[
    \sup_{Z \in \mc{A}_{x,\delta,\epsilon}^L} \abs{\frac{1}{N} \log \E[e^{N\frac{\theta}{2}\ip{e,H_Ne}}] - J\left(\hat{\mu}_{H_N},\frac{\theta}{2},x\right)} \leq t.
\]
On the other hand, their Theorem 21 shows that $J(\mu,\frac{\theta}{2},x)$ is a continuous function of $\mu$ in the set of probability measures compactly supported on $[-L,L]$; thus for $\epsilon \leq \epsilon_0(t)$ we have
\[
    \sup_{Z \in \mc{A}_{x,\delta,\epsilon}^L} \abs{ J\left(\hat{\mu}_{H_N},\frac{\theta}{2},x\right) - J\left(\sigma,\frac{\theta}{2},x\right)} \leq t.
\]
We combine these two estimates with the quantifiers in the right order, then take $t \downarrow 0$. 
\end{proof}

\begin{lem}
\label{lem:ub_good_event}
For each $x \geq r(\sigma)$, $0 \leq \theta < \theta_{\textup{max}}$, and $L$ large enough, we have
\[
    \limsup_{\delta \downarrow 0} \limsup_{\epsilon \downarrow 0} \limsup_{N \to \infty} \frac{1}{N} \log \P^\theta(\mc{A}_{x,\delta,\epsilon}^L) \leq -(\widetilde{I_\sigma}(x) - I_\sigma(x,\theta)).
\]
\end{lem}
\begin{proof}
For any $0 \leq \theta' < \theta_{\textup{max}}$,
\begin{align*}
    \P^\theta(\mc{A}_{x,\delta,\epsilon}^L) &= \frac{1}{\E_{e,H_N}[e^{N\frac{\theta}{2}\ip{e,H_Ne}}]} \E_{H_N}\left[\mathds{1}_{H_N \in \mc{A}_{x,\delta,\epsilon}^L} \E_e[e^{N\frac{\theta}{2}\ip{e,H_Ne}}] \frac{\E_e[e^{N\frac{\theta'}{2}\ip{e,H_Ne}}]}{\E_e[e^{N\frac{\theta'}{2}\ip{e,H_Ne}}]} \right] \\
    &\leq \frac{\E_{e,H_N}[e^{N\frac{\theta'}{2}\ip{e,H_Ne}}]}{\E_{e,H_N}[e^{N\frac{\theta}{2}\ip{e,H_Ne}}]} \left( \sup_{Z \in \mc{A}_{x,\delta,\epsilon}^L} \E_e[e^{N\frac{\theta}{2}\ip{e,H_Ne}}] \right) \left( \sup_{Z \in \mc{A}_{x,\delta,\epsilon}^L} \frac{1}{\E_e[e^{N\frac{\theta'}{2}\ip{e,H_Ne}}]}\right)
\end{align*}
By Lemmas \ref{lem:annealed} and \ref{lem:quenched}, this shows
\[
    \limsup_{\delta \downarrow 0} \limsup_{\epsilon \downarrow 0} \limsup_{N \to \infty} \frac{1}{N} \log \P^{\theta}(\mc{A}_{x,\delta,\epsilon}^L) \leq I_\sigma(x,\theta) - I_\sigma(x,\theta').
\]
Taking the infimum over $0 \leq \theta' < \theta_{\textup{max}}$ on the right-hand side completes the proof.
\end{proof}

\begin{lem}
\label{lem:unlikely_adding_theta}
Suppose $A_{N,M}$ is a doubly-indexed sequence of events with $\lim_{M \to \infty} \lim_{N \to \infty} \frac{1}{N} \log \P(A_{N,M}) = -\infty$. Then for every $\theta < \theta_{\textup{max}}$ we have
\[
    \lim_{M \to \infty} \lim_{N \to \infty} \frac{1}{N} \log \P^\theta(A_{N,M}) = -\infty.
\]
In particular (by taking $A_{N,M}$ independent of $M$), if $\lim_{N \to \infty} \frac{1}{N} \log \P(A_N) = -\infty$, then $\lim_{N \to \infty} \frac{1}{N} \log \P^\theta(A_N) = -\infty$.
\end{lem}

\begin{proof}[Proof of Proposition \ref{prop:wkldpub}]
If $x < r(\sigma)$, then for small enough $\delta$ we have $\{\abs{\lambda_{\max{}}(H_N) - x} \leq \delta\} \subset \{d_{\textup{BL}}(\hat{\mu}_{H_N},\sigma) > \epsilon\}$ for some $\epsilon = \epsilon(\delta)$; this suffices by Lemma \ref{lem:unlikely_adding_theta} and  Proposition \ref{prop:sub_Gaussian_concentration}. In the remainder we assume $x \geq r(\sigma)$.

For large $L$, we have
\[
    \P^{\theta}(\abs{\lambda_{\textup{max}}(H_N)-x} \leq \delta) \leq \P^\theta(\mc{A}_{x,\delta,\epsilon}^L) + \P^\theta(d_{\textup{BL}}(\hat{\mu}_{H_N},\sigma) > \epsilon) + \P^\theta(\|H_N\| > L).
\]
Then we take the normalized log of both sides; by taking $N \to +\infty$, and applying again Lemma \ref{lem:unlikely_adding_theta} and  Proposition \ref{prop:sub_Gaussian_concentration}, we find
\[
    \limsup_{N \to \infty} \frac{1}{N} \log \P^\theta(\abs{\lambda_{\textup{max}}(H_N) - x} \leq \delta) \leq \max\left\{ \limsup_{N \to \infty} \frac{1}{N} \log \P^\theta(\mc{A}_{x,\delta,\epsilon}^L), \limsup_{N \to \infty} \frac{1}{N} \log \P^\theta(\|H_N\| > L) \right\}.
\]
Then we take $\epsilon \downarrow 0$, then $\delta \downarrow 0$, and apply Lemma \ref{lem:ub_good_event} to get
\[
    \limsup_{\delta \downarrow 0} \limsup_{N \to \infty} \frac{1}{N} \log \P^\theta(\abs{\lambda_{\textup{max}}(H_N) - x} \leq \delta) \leq \max\left\{ -(\widetilde{I_\sigma}(x)-I_\sigma(x,\theta)),\limsup_{N \to \infty} \frac{1}{N} \log \P^\theta(\|H_N\| > L)\right\}.
\]
By taking $L \to \infty$ and applying Lemmas \ref{lem:unlikely_adding_theta} and \ref{lem:exponential_tightness}, we finish the proof.
\end{proof}

\begin{proof}[Proof of Lemma \ref{lem:unlikely_adding_theta}]
For every unit vector $e$, the proof of Lemma \ref{lem:annealed} shows 
\[
    \E_{H_N}[e^{N\frac{\theta}{2}\ip{e,H_Ne}}] = \prod_{\mu=1}^M \E[e^{\frac{N}{M}\frac{\theta}{2}d_\mu\ip{z_\mu,e}^2}] \geq 1,
\]
so that, if $A$ is an event and $\epsilon > 0$, we have
\[
    \P^\theta(A) \leq \E_{H_N}[\mathbf{1}_A \E_e[e^{N\frac{\theta}{2}\ip{e,H_Ne}}]] \leq \P(A)^{\frac{\epsilon}{1+\epsilon}} \E_{e,H_N}[e^{N\frac{(1+\epsilon)\theta}{2}\ip{e,H_Ne}}]^{\frac{1}{1+\epsilon}}
\]
If $\epsilon$ is so small that $(1+\epsilon)\theta < \theta_{\textup{max}}$, we apply Lemma \ref{lem:annealed} to finish the proof.
\end{proof}


\subsection{Proof of the weak LDP lower bound}
\label{subsec:wk_ldp_lb}

\begin{prop}
\label{prop:thetax}
Assume $(\rho,\alpha)$ is nondegenerate, and let $D^\circ$ be the interior of the set $D$ defined in \eqref{eqn:domain_d}. If $r(\rho) > 0$, assume additionally that $G_\rho(r(\rho)) = +\infty$. Then for every $x \in D^\circ$ there exists a unique $\theta_x$ with $0 \leq \theta_x < \theta_{\textup{max}}$ such that
\[
    \widetilde{I_\sigma}(x) = \sup_{0 \leq \theta < \theta_{\textup{max}}} I_\sigma(x,\theta) = I_\sigma(x,\theta_x).
\]
The map $x \mapsto \theta_x$ is injective.
\end{prop}
\begin{proof}
	This is a consequence of the proof of Lemma \ref{lem:ratesimp}. In particular $\theta_x = \overline{G}_{\sigma}(x)$. 
	\end{proof}

\begin{lem}
\label{lem:a_is_likely}
Let $x \geq r(\sigma)$, and let $\theta_x$ be as defined in Proposition \ref{prop:thetax}. Then for all $L$ sufficiently large and $\delta, \epsilon$ sufficiently small (depending on $x$), we have 
\[
    \lim_{N \to \infty} \frac{1}{N}\log \P^{\theta_x}(\mc{A}_{x,\delta,\epsilon}^L) = 0.
\]
\end{lem}
\begin{proof}
We claim that actually $\P^{\theta_x}(\mc{A}_{x,\delta,\epsilon}^L)$ tends to one. As shown in the proof of Proposition \ref{prop:wkldpub}, for all $0 \leq \theta < \theta_{\textup{max}}$, the quantities $\P^{\theta}(d_{\textup{BL}}(\hat{\mu}_{H_N},\sigma) \geq \epsilon)$ and $\P^{\theta}(\|H_N\| \geq L)$ tend to zero (actually exponentially quickly), for $L$ large enough and all $\epsilon > 0$. Thus it suffices to show
\begin{equation}
\label{eqn:likelyunderthetax}
    \P^{\theta_x}(\abs{\lambda_{\textup{max}}(H_N) - x} \geq \delta) = \oo(1).
\end{equation}
But Proposition \ref{prop:wkldpub} gives a weak LDP upper bound for the variable $\lambda_{\textup{max}}(H_N)$ under the measures $\P^{\theta_x}$, which is exponentially tight, with rate function $J_x(y)$ that is infinite for $y < r(\sigma)$ and otherwise equal to $J_x(y) = \widetilde{I_\sigma}(y) - I_\sigma(y,\theta_x)$. Proposition \ref{prop:thetax} shows that $J_x$ is nonnegative and vanishes uniquely at $x$; indeed, if $r(\sigma) \leq y \neq x$, then $J_x(y) = \sup_{0 \leq \theta < \theta_{\textup{max}}} I(y,\theta) - I(x,\theta_x)$, and the supremum is achieved uniquely at $\theta_y$, which is different from $\theta_x$ since $y \neq x$. Combined with exponential tightness of $\P^{\theta_x}$ (see Lemmas \ref{lem:exponential_tightness} and \ref{lem:unlikely_adding_theta}), this gives \eqref{eqn:likelyunderthetax}.
\end{proof}

\begin{proof}[Proof of Proposition \ref{prop:wkldplb}]
We have 
\begin{align*}
    \P(\abs{\lambda_{\textup{max}}(H_N) - x} \leq \delta) &\geq \P(\mc{A}_{x,\delta,\epsilon}^L) \geq \frac{\E[\mathbf{1}_{\mc{A}_{x,\delta,\epsilon}^L} \E_e[e^{N\frac{\theta_x}{2}\ip{e,H_Ne}}]}{\E[e^{N\frac{\theta_x}{2}\ip{e,H_Ne}}]} \E[e^{N\frac{\theta_x}{2}\ip{e,H_Ne}}] \left(\inf_{Z \in \mc{A}_{x,\delta,\epsilon}^L} \frac{1}{\E_e[e^{N\frac{\theta_x}{2}\ip{e,H_Ne}}]}\right). \\
    &= \P^{\theta_x}(\mc{A}_{x,\delta,\epsilon}^L)\E[e^{N\frac{\theta_x}{2}\ip{e,H_Ne}}] \left(\inf_{Z \in \mc{A}_{x,\delta,\epsilon}^L} \frac{1}{\E_e[e^{N\frac{\theta_x}{2}\ip{e,H_Ne}}]}\right).
\end{align*}
From Lemmas \ref{lem:annealed}, \ref{lem:quenched}, and \ref{lem:a_is_likely}, we have
\[
    \liminf_{\delta \downarrow 0} \liminf_{N \to \infty} \frac{1}{N} \log \P(\abs{\lambda_{\textup{max}}(H_N) - x} \leq \delta) \geq F(\rho,\theta_x) - J\left(\sigma,\frac{\theta_x}{2},x\right) = -I_\sigma(x,\theta_x) \geq -\widetilde{I_\sigma}(x).
\]
\end{proof}


\subsection{Degenerate cases}
\label{subsec:r_sigma_0}

\begin{proof}[Proof of Lemma \ref{lem:r_sigma_0}]
Let $\rho$ be a compactly supported measure on $\R$ such that $r(\rho) \leq 0$. Theorem \ref{thm:sigma} already shows that $r(\sigma) \leq 0$; as in the proof of that theorem, we consider a sequence $(\Gamma_M)_{M=1}^\infty$ chosen without outliers, and the matrices $H_N$ defined using these $\Gamma_M$'s. There are two cases:
	\begin{enumerate} 
		\item If $\alpha \leq 1$, since  $H_N \geq - K Z^T Z $ where $K = \max_M ( - \lambda_{\min}( \Gamma_M))$. The empirical measure of $ Z^T Z$ converges toward the Mar\v{c}enko-Pastur distribution $\text{MP}_{\alpha}$ and therefore the cumulative distribution function of $\sigma$ is everywhere smaller than the distribution function of $(-K) * \text{MP}_{\alpha}$, where $(-K) * \text{MP}_{\alpha}$ is the push-forward of $\text{MP}_{\alpha}$ by the multiplication by $(-K)$. Since $r(- K * \text{MP}_{\alpha}) = - K \ell( \text{MP}_{\alpha})=0$, we have that $r(\sigma) \geq 0$.
		\item If $\alpha > 1$, we can find a $\delta > 1$ such that $\alpha - \delta > 1$. If we call $\delta_M = \lfloor \delta M \rfloor$, we let $\Gamma'_M$ be the submatrix of $\Gamma_M$ with the first $M - \delta_M$ rows and columns, and $\Gamma''_M$ be the submatrix of $\Gamma_M$ with the last $\delta_M$ rows and columns. We also let $Z'$ be the submatrix of $Z$ with the first $M - \delta_N$ rows, and $Z''$ be the submatrix of $Z$ with the last $\delta_N$ rows. We have $H_N = H'_N + H''_N$ with $H_N'= Z'^T \Gamma_M' Z'$ and $H_N''= Z''^T \Gamma_M'' Z''$. Since $\rho( \{ 0 \}) =0$, we have that $\limsup_N \lambda_{\max}( \Gamma'_M) =c < 0$, and since $H''_N \leq 0$, we have for $N$ large enough $H_N \leq H'_N \leq c Z'^T Z'$. So $r(\sigma) \leq c \ell( MP_{\alpha - \delta}) < 0 $. 
	\end{enumerate}
\end{proof}

\begin{proof}[Proof of Proposition \ref{prop:degenerate}]
The same proof as in the nondegenerate case shows that $\lambda_{\textup{max}}$ cannot push into the bulk at this speed, which handles small-ball probabilities $\P(\lambda_{\textup{max}} \approx x)$ for $x < 0$. 

If $\Gamma$ is negative semidefinite, then the rest of the argument is trivial, since $H_N \leq 0$ and thus $\lambda_{\textup{max}}(H_N) \leq 0$ deterministically.

So suppose that $\Gamma$ has a handful of positive eigenvalues at finite $N$, meaning that $H_N$ is not necessarily negative semidefinite. Define 
$\Gamma^{\textup{nsd}} \defeq \diag(d_1^{(\textup{nsd})}, \ldots, d_M^{(\textup{nsd})})$, where $d_i^{(\textup{nsd})} = \min(d_i,0)$, set $\epsilon_N = \|\Gamma - \Gamma^{(\textup{nsd})}\|$ which tends to zero by Assumption \ref{assn:no_outliers}, and set $H_N^{(\textup{nsd})} \defeq \frac{1}{M} Z^T \Gamma^{(\textup{nsd})} Z$, which we couple with $H_N$ by using the same noise $Z$ to define both. We note that $\|H_N - H_N^{(\textup{nsd})}\| \leq \frac{\epsilon_N}{M} \|Z\|^2$ and thus, for every $\delta > 0$,
\[
	\lim_{N \to \infty} \P(|\lambda_{\textup{max}}(H_N^{(\textup{nsd})}) - \lambda_{\textup{max}}(H_N)| > \delta) \leq \lim_{N \to \infty} \P(\sqrt{M}^{-1}\|Z\| > \delta/\epsilon_N) = -\infty
\]
(the details of this are given in the proof of Lemma \ref{lem:Approx} below). This means that the sequences $(\lambda_{\textup{max}}(H_N))_{N=1}^\infty$ and $(\lambda_{\textup{max}}(H_N^{(\textup{nsd})}))_{N=1}^\infty$ are \emph{exponentially equivalent}; since the latter sequence satisfies the desired LDP by the above argument, it is classical (see, e.g., \cite[Theorem 4.2.13]{DemZei2010}) that the former does, as well.
\end{proof}


\subsection{Second branch of the Stieltjes transform}
\label{subsec:overline_g}

\begin{proof}[Proof of Lemma \ref{lem:overline_g}]
If $x_c(\rho)$ is finite, then
\begin{align*}
	x_c(\rho) &= r(\rho)^2G_\rho(r(\rho)) + \left(\frac{1}{\alpha}-1\right) r(\rho) = \frac{1}{\theta_{\textup{max}}} + r(\rho) \left( \int_\R \frac{r(\rho)-(r(\rho)-u)}{r(\rho)-u} \rho(\diff u) \right) \\
	&= \frac{1}{\theta_{\textup{max}}} + \int_\R \frac{\alpha u}{\alpha - \frac{\alpha}{r(\rho)} u} \rho(\diff u) = H_\rho(\theta_{\textup{max}}).
\end{align*}
The claim $x_c(\rho) \geq r(\sigma)$ will be shown along the course of the proof. We will eventually need three cases. Common to them is the computation of
\[
	f_\rho(\theta) \defeq \theta^2H'_\rho(\theta) = -1+\alpha \int_\R \frac{u^2\theta^2}{\left(\alpha-u\theta\right)^2} \rho(\diff u).
\]
Notice that $\lim_{\theta \downarrow 0} f_\rho(\theta) = -1$. We claim $f_\rho$ is strictly increasing for $\theta \in (0,\theta_{\textup{max}})$. Indeed, it is (a constant plus) an average over $u$ of the functions $f_{u,\rho}(\theta) \defeq \frac{u^2\theta^2}{(\alpha - u\theta)^2}$; the function $f_{0,\rho}$ is constant, and the functions $f_{u,\rho}$ are strictly increasing for each $u \neq 0$, since their derivatives $\frac{2(\frac{\alpha}{u})\theta(\frac{\alpha}{u}-\theta)}{(\frac{\alpha}{u}-\theta)^4}$ have the same sign as $\frac{\alpha}{u}(\frac{\alpha}{u}-\theta)$, and the sign of the latter can be checked by hand depending on the sign of $u$ (in the case $u > 0$, this relies on $\theta < \theta_{\textup{max}} = \frac{\alpha}{r(\rho)}$). 

The three cases are:
\begin{enumerate}
		\begin{figure}[h]
			\includegraphics[scale=0.8]{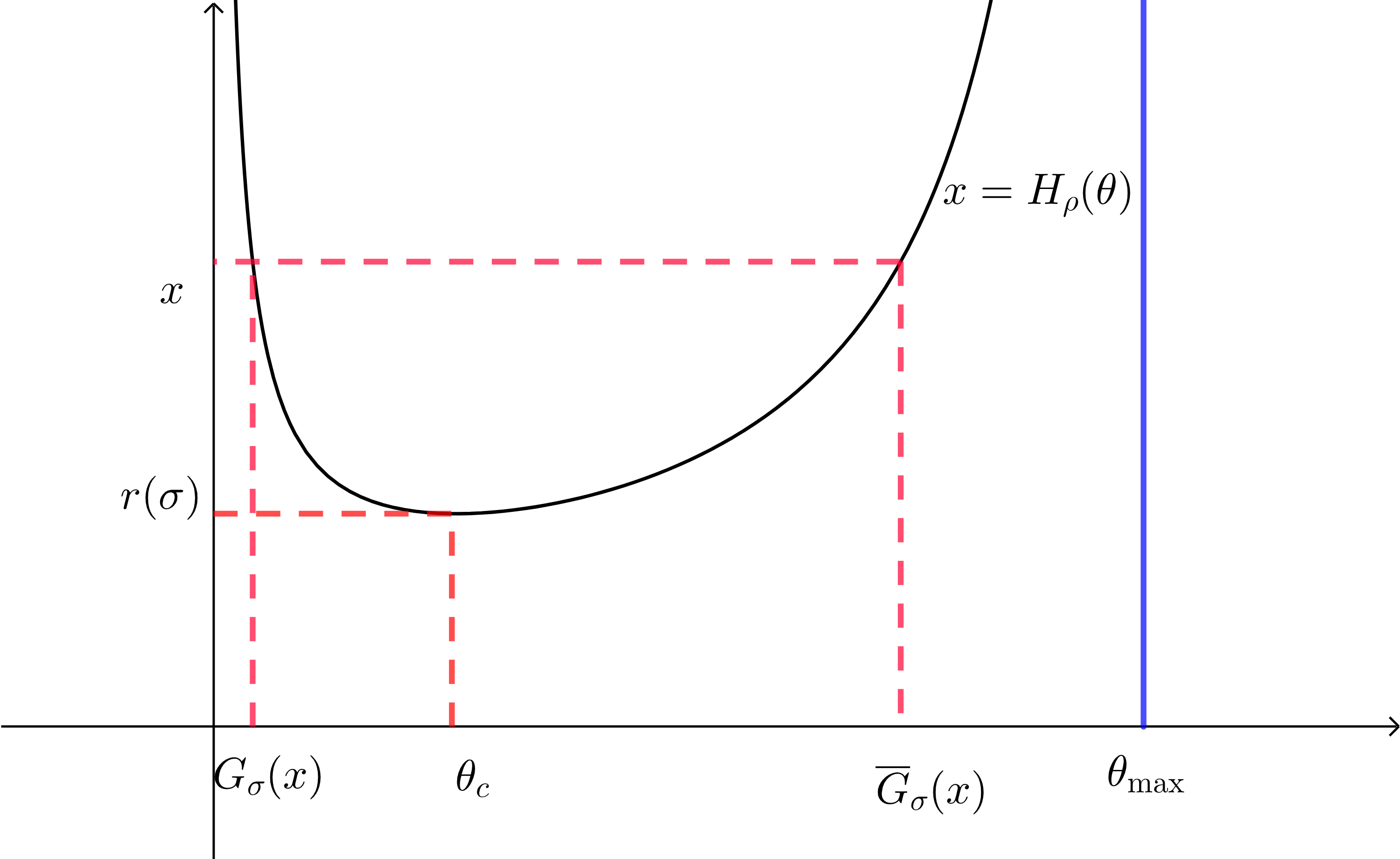}	
			\centering
			\caption{Graph of the function $H_{\rho}$ when $r(\rho)>0$ and $G_{\rho}(r(\rho))= + \infty$ with representations of $r(\sigma),\theta_c,G_{\sigma}(x),\overline{G}_{\sigma}(x)$.}
			\label{fig:case1}
		\end{figure}
	\item \textbf{Case 1 ($r(\rho) > 0$ and $G_{\rho}(r(\rho)) = +\infty$), shown in Figure~\ref{fig:case1}:} It is easy to see that $H_{\rho}( \theta_{\max}) = + \infty$. Thus $f_\rho(\theta)$ is positive for some $\theta$; since it is also strictly increasing and tends to $-1$ at zero, there exists a unique $\theta_c \in (0,\theta_{\textup{max}})$ where it vanishes, i.e., there exists a unique $\theta_c \in (0,\theta_{\textup{max}})$ such that $H_{\rho}$ is decreasing on $(0,\theta_c)$ and increasing on $(\theta_c,\theta_{\textup{max}})$.
	Using the uniqueness of the analytic continuation, one can argue that $H_{\rho}( \theta_c) = r( \sigma)$ and $G_{\sigma}( r(\sigma))= \theta_c$. Furthermore, one sees that the equation $H_{\rho}(y) = x$, considered as a function of $y \in (0, \theta_{\max})$ parametrized by $x \in \R$,
	\begin{enumerate}
		\item has no solution if $x < r( \sigma)$. 
		\item has one solution if $x = r(\sigma)$. That solution is $\theta_c$, and we set $\overline{G}_\sigma(r(\sigma)) \defeq \theta_c$.
		\item has two solutions $y_1$ and $y_2$ such that $ 0 < y_1 < \theta_c < y_2 < \theta_{\max}$ if $ x > r( \sigma)$. Furthermore, due to the Dyson equation \eqref{eqn:dyson}, we clearly have $y_1 = G_{\sigma}( x)$. We write $\overline{G}_\sigma(x)$ for the second solution $y_2$. In particular, $\overline{G}_{\sigma}$ defined this way on $[ r(\sigma), + \infty)$ is analytic increasing and $\lim_{x \to \infty} \overline{G}_{\sigma}(x) = \theta_{\max}$. 
		\end{enumerate}
	\begin{figure}[ht!]
			\includegraphics[scale=4]{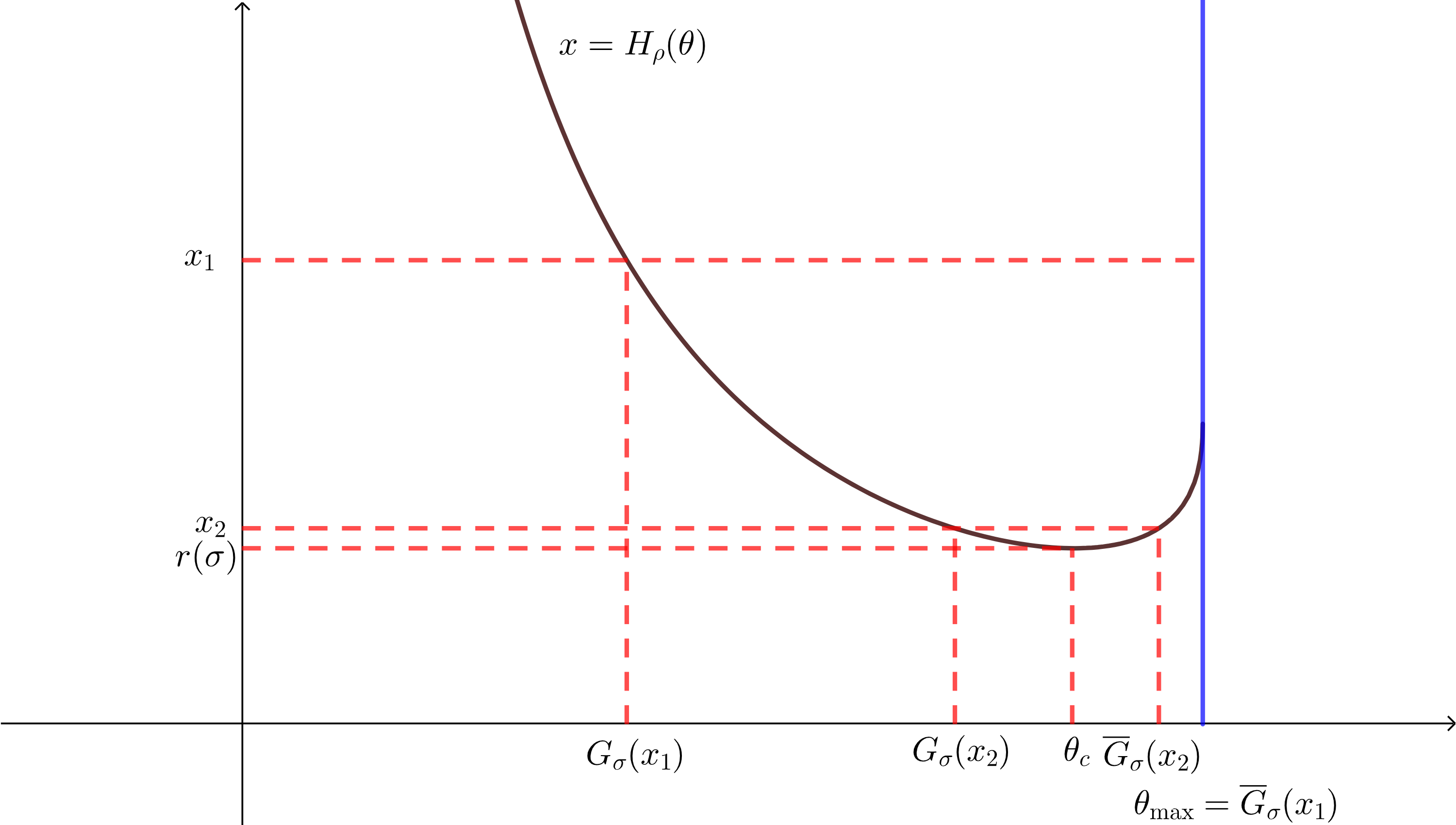}	
			\centering
			\caption{Graph of the function $H_{\rho}$ when $r(\rho)>0$ and $G_{\rho}(r(\rho))< + \infty$ with representations of $r(\sigma),\theta_c,G_{\sigma}(x),\overline{G}_{\sigma}(x)$ with $x_1 > x_c$ and $ x_2 < x_c$ .}
			\label{fig:case2}
	\end{figure}
	\item \textbf{Case 2 ($r(\rho) > 0$ and $G_\rho(r(\rho)) < + \infty$), shown in Figure~\ref{fig:case2}:} Once again, $H_{\rho}$ is decreasing on $(0, \theta_c)$ and increasing $(\theta_c , \theta_{\textup{max}})$, but $\theta_c = \theta_{\max}$ if and only if $x_c = r( \sigma)$. As before, the equation $H_{\rho}(y) = x$, considered as a function of $y \in (0, \theta_{\max})$ parametrized by $x \in \R$,
		\begin{enumerate}
			\item has no solution if $x < r( \sigma)$. 
			\item has one solution if $x = r(\sigma)$. That solution is $\theta_c$, and we set $\overline{G}_\sigma(r(\sigma)) \defeq \theta_c$.
			\item has two solutions $y_1$ and $y_2$ such that $ 0 < y_1 < \theta_c < y_2 < \theta_{\max}$ if $r( \sigma) < x \leq x_c$. Once again, we have $y_1= G_{\sigma}( x)$ and  we will set $\overline{G}_{\sigma} \defeq y_2$. 
			\item has one solution $y$ such that $0 < y < \theta_c$ if $x > x_c(\rho)$. However, in this case we will define $\overline{G}_{\sigma}(x) \defeq \theta_{\max}$. 
			\end{enumerate} 
	\begin{figure}[h!]
	\includegraphics[scale=2]{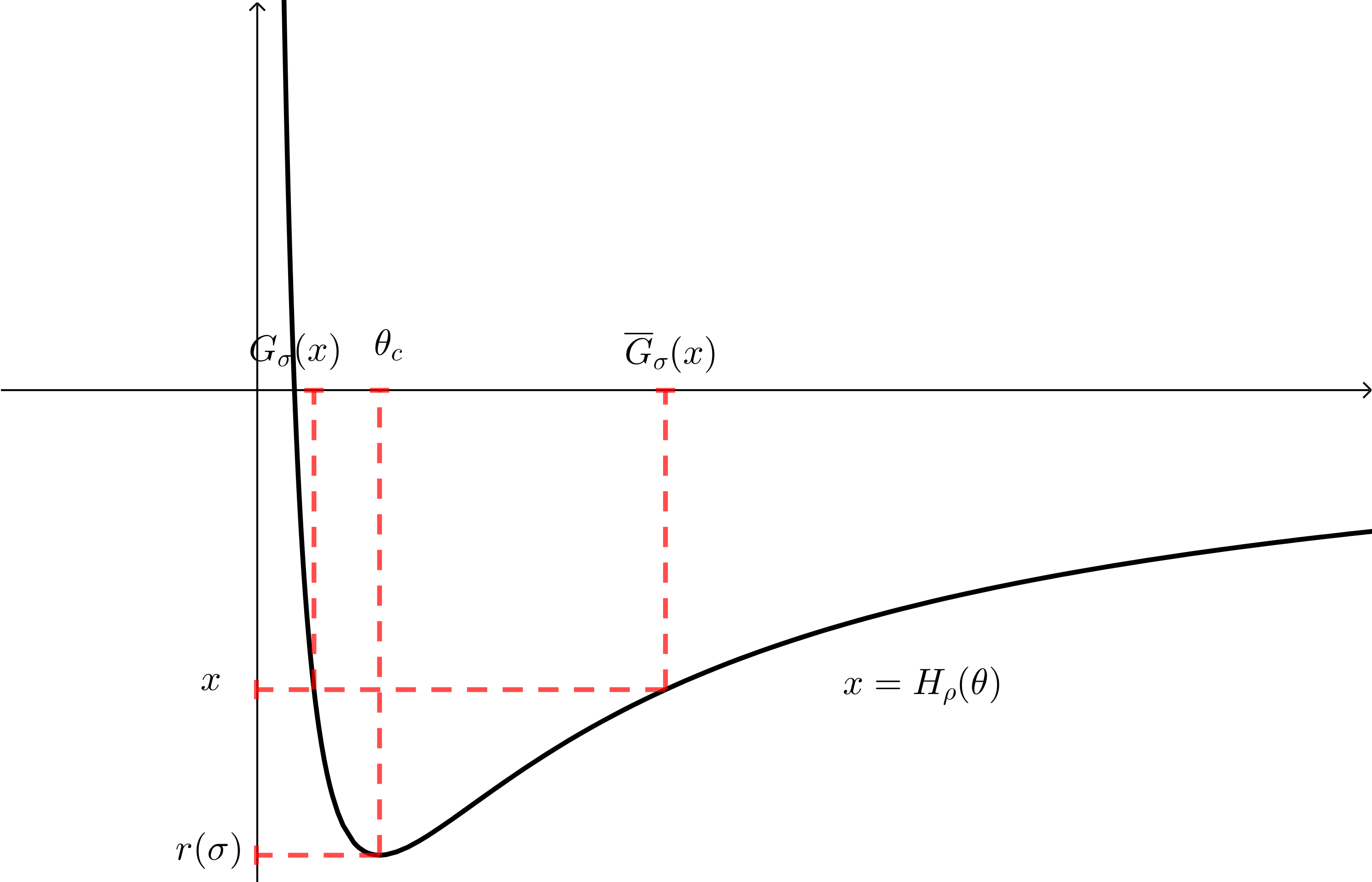}	
	\centering
	\caption{Graph of the function $H_{\rho}$ when $r(\rho) <0$ and $\alpha > 1$ with representations of $r(\sigma),\theta_c,G_{\sigma}(x),\overline{G}_{\sigma}(x)$.}
	\label{fig:case3}
	\end{figure}
	\item \textbf{Case 3 ($r(\rho) \leq 0$):} Using Remark \ref{rem:r_sigma_0} and Lemma \ref{lem:r_sigma_0}, we need only consider the case where $\rho( \{0\}) =0$ and $\alpha > 1$ (see Figure~\ref{fig:case3}). Since $f_\rho$ is strictly increasing, either $H'_\rho(\theta)$ is negative for all $\theta$, or there exists $\theta_c \in (0,\theta_{\textup{max}})$ such that $H'_\rho(\theta)$ is negative for $\theta \in (0,\theta_c)$ and positive for $\theta \in (\theta_c,\theta_{\textup{max}})$. But since
		\begin{equation}
		\label{eqn:h_rho_asymptotics}
			H_{\rho}( \theta) = (1- \alpha) \theta^{-1} + \oo_{\theta \to +\infty} ( \theta^{-1})
		\end{equation}
		and $\alpha > 1$, we must be in the latter case, and indeed must have $H_{\rho}( \theta_c) < 0$ and $H_{\rho} < 0$ on $(\theta_c, +\infty)$. Then the equation $H_{\rho}(y) = x$, considered as a function of $y \in (0, + \infty)$ parametrized by $x \in \R$,
	\begin{enumerate}
		\item has one solution for $x \geq 0$. This solution is equal to $G_{\sigma}(x)$. 
		\item has two solutions $y_1, y_2$ for $0 > x > r( \sigma)$, with $ 0 < y_1 < \theta_c < y_2$. We have that $y_1 = G_{\sigma}(x)$ and we denote the second solution $y_2$ by $\overline{G}_{\sigma}(x)$. Once again $\overline{G}_{\sigma}$ is analytic increasing between $r( \sigma)$ and $0$, and $\lim_{x \uparrow 0} \overline{G}_\sigma = +\infty$.
	\item has one solution for $x = r(\sigma)$, namely $y= \theta_c$.
	\item has no solutions for $x < r( \sigma)$. 
	\end{enumerate} 
\end{enumerate}
\end{proof}


\section{Proof for finite \texorpdfstring{$x_c$}{xc}}
\label{SecReg}

In this section, we prove Proposition \ref{prop:main_finite}. We remark that, from the definition \eqref{eqn:def_xc_rho}, we can only have $x_c(\rho) < +\infty$ if $r(\rho) > 0$.

First, from its definition and Lemma \ref{lem:overline_g}, we see that the rate function has the form $I_\sigma(x) = \frac{1}{2} \int_{r(\sigma)}^x g(y) \diff y$, where $g$ is strictly increasing, positive for arbitrarily small arguments, and $\lim_{x \to +\infty} \frac{g(x)}{\theta_{\textup{max}}} = 1$; this proves the claimed properties of $I_\sigma$.

Now we fix once and for all some $\rho$ with $x_c(\rho) < +\infty$, and try to prove the associated LDP, assuming that we know the LDP for every model with $x_c = +\infty$. The proof goes by approximation. Precisely, we are going to discretize the right edge of $\rho$ by replacing the $d_i$'s greater than $r(\rho) - \epsilon$ by $r(\rho)$. 

\begin{defn} 
	For $\epsilon > 0$, we define $\Gamma^{(\epsilon)} \defeq \diag(d^{(\epsilon)}_1,...,d^{(\epsilon)}_M)$ where $d_i^{(\epsilon)} =d_i$ if $d_i \leq r(\rho) - \epsilon$ and $d_i^{(\epsilon)} = r(\rho)$ for $d_i >  r(\rho) - \epsilon$. 
	The same way, we define $H_N^{(\epsilon)}$ as: 
	\[ H_N^{(\epsilon)} \defeq \frac{1}{M} Z^{T} \Gamma^{(\epsilon)} Z. \]
	It will be important later that we couple $H^{(\epsilon)}_N$ with $H_N$, by using the same noise $Z$ to define both.
	We define also $\rho^{(\epsilon)}$ to be the probability measure on $\R$ given by
	\begin{equation}
	\label{eqn:rho_epsilon}
	\rho^{(\epsilon)}(A) \defeq \rho(A \cap ] - \infty, r(\rho) - \epsilon]) + \rho( ] r(\rho) - \epsilon, r(\rho)]) \delta_{r(\rho)}
	\end{equation}
	for Borel $A$.
\end{defn}
Let us remark that 
\[ \lim_{M \to \infty}\frac{1}{M} \sum_{i=1}^M \delta_{d^{(\epsilon)}_i} = \rho^{(\epsilon)} \]
as long as $\rho$ does not have an atom at $r(\rho) -\epsilon$. Since a probability measure can have at most countably many atoms, we can take $\epsilon \to 0$ along some $\rho$-dependent sequence avoiding such atoms, which we will do implicitly in the rest of the proof. 

Then, if we assume we have avoided such atoms, the empirical measure of $H_N^{(\epsilon)}$ converges toward a measure $\sigma^{(\epsilon)}$ characterized by the fact that its Stieltjes transform is the inverse function of $H_{\rho^{(\epsilon)}}$. 

Since $\rho^{(\epsilon)}$ has an atom at its right endpoint, we have $G_{\rho^{(\epsilon)}}(r(\rho^{(\epsilon)})) = + \infty$ and therefore we have that $\lambda_{\max}(H_N^{(\epsilon)})$ satisfies a large deviations principle with rate function $I^{(\epsilon)}$ defined as 

\[ I^{(\epsilon)}(x) =\begin{cases}  \frac{1}{2}  \int_{r(\sigma^{(\epsilon)})}^x \Big( \overline{G_{\sigma^{(\epsilon)}}}(t) - G_{\sigma^{(\epsilon)}}(t) \Big) \diff t  & \text{if } x \geq r(\sigma^{(\epsilon)} ), \\
	+ \infty & \text{otherwise. }  \end{cases}
\]

To prove our result we will need the following three lemmas.

\begin{lem}
\label{lem:CvRatef}
The function $\epsilon \mapsto r(\sigma^{(\epsilon)})$ is non-decreasing, and
	\begin{equation}
	\label{eqn:r_sigma_epsilon}
	\lim_{\epsilon \to 0} r(\sigma^{(\epsilon)}) = r(\sigma).
	\end{equation}
Furthermore, the functions $I^{(\epsilon)}$ converge uniformly on all compact subsets of $(r(\sigma), + \infty)$ toward $I$ as $\epsilon \to 0$. 
\end{lem}

\begin{lem}\label{lem:Approx}
For every $K > 0$, if the ratio $\frac{\eta}{\epsilon}$ is large enough depending on $K$, then 
\[
	\limsup_{N \to \infty} \frac{1}{N} \log \P[\|H_N - H_N^{(\epsilon)}\| \geq \eta] \leq -K
\]
where we recall $\|\cdot\|$ is the operator norm (or spectral radius in this case). 
\end{lem}

\begin{lem}
\label{lem:dz}
Define $J : \R \to \R$ by
\begin{equation}
\label{eqn:dz4.2.17}
	J(x) = \sup_{\delta > 0} \liminf_{\epsilon \downarrow 0} \inf_{y \in (x-\delta,x+\delta)} I^{(\epsilon)}(y).
\end{equation}
Then $I = J$. Furthermore, $I$ is a good rate function, and for every closed set $F \subset \R$, we have
\begin{equation}
\label{eqn:dz4.2.18}
	\inf_{y \in F} I(y) \leq \limsup_{\epsilon \downarrow 0} \inf_{y \in F} I^{(\epsilon)}(y).
\end{equation}
\end{lem}

Let us assume these three lemmas momentarily, and prove that they imply the large deviation principle. 

\begin{proof}[Proof of Proposition \ref{prop:main_finite}]
This will be an immediate consequence of Theorem 4.2.16 of \cite{DemZei2010}, which explains how to recover an LDP for $\lambda_{\textup{max}}(H_N)$ from LDPs for $\lambda_{\textup{max}}(H_N^{(\epsilon)})$ in the $\epsilon \downarrow 0$ limit.\footnote{
Translating the notation: Their $m$ is our $\epsilon^{-1}$, and their $\epsilon$ is our $N^{-1}$. Thus their $\widetilde{\mu_\epsilon}$ is the law of $\lambda_{\textup{max}}(H_{\epsilon^{-1}})$, and their $\mu_{\epsilon,m}$ is the law of $\lambda_{\textup{max}}(H_{\epsilon^{-1}}^{(m^{-1})})$.} The condition they define as ``exponentially good approximations,'' translated into our notation, reads
\[
	\lim_{\epsilon \downarrow 0} \limsup_{N \to \infty} \frac{1}{N} \log \P\left(\abs{\lambda_{\textup{max}}(H_N^{(\epsilon)}) - \lambda_{\textup{max}}(H_N)} > \delta\right) = -\infty,
\]
which follows from Lemma \ref{lem:Approx}. We checked the remaining conditions of their result in Lemma \ref{lem:dz} above.
\end{proof}

 \begin{proof}[Proof of Lemma \ref{lem:CvRatef}]
	By construction, $r(\rho^{(\epsilon)}) = r(\rho)$, so that $\theta_{\textup{max}}$ is independent of $\epsilon$. Thus
	\[
		r(\sigma^{(\epsilon)}) = \min_{0 < \theta < \theta_{\textup{max}}} H_{\rho^{(\epsilon)}}(\theta).
	\]
 	Since $\rho^{(\epsilon)}$ converges toward $\rho$, and thus $H_{\rho^{(\epsilon)}}$ converges to $H_{\rho}$ uniformly on all compact subsets of $(0, \theta_{\max})$, this implies
 	
 	\[ \limsup_{\epsilon \to 0} r(\sigma^{(\epsilon)}) \leq r(\sigma). \]
	On the other hand, we know more about this convergence: We claim that, for each $\theta \in (0,\theta_{\textup{max}})$, 
	\begin{equation}
	\label{eqn:H_vs_Heps}
		H_\rho(\theta) \leq H_{\rho^{(\epsilon)}}(\theta)
	\end{equation}
	and that the function $\epsilon \mapsto H_{\rho^{(\epsilon)}}(\theta)$ is actually non-decreasing for $\epsilon \in (0,r(\rho))$. (Notice this implies that $\epsilon \mapsto r(\sigma^{(\epsilon)})$ is non-decreasing.) Indeed, we can write
	\[
		H_{\rho^{(\epsilon)}}(\theta) = \frac{1}{\theta} + \int_{-\infty}^{r(\rho)-\epsilon} \frac{\alpha u}{\alpha-\theta u} \rho(\diff u) + \frac{\alpha r(\rho)}{\alpha - \theta r(\rho)} \rho((r(\rho)-\epsilon,r(\rho)]) = \frac{1}{\theta} + \int_{\R} f_{\alpha,\theta}^{(\epsilon)}(u) \rho(\diff u)
	\]
	where $f_{\alpha,\theta}^{(\epsilon)} : \supp(\rho) \to \R$ is defined by
	\[
		f_{\alpha,x}^{(\epsilon)}(u) = \begin{cases} \frac{\alpha u}{\alpha-\theta u} & \text{if } u < r(\rho) - \epsilon, \\ \frac{\alpha r(\rho)}{\alpha - \theta r(\rho)} & \text{if } u \geq r(\rho) - \epsilon. \end{cases}
	\]
	Since $u \mapsto \frac{\alpha u}{\alpha - \theta u}$ is strictly increasing on the support of $\rho$ and positive for $u > 0$, the map $\epsilon \mapsto f_{\alpha,\theta}^{(\epsilon)}(u)$ is, for each $u \in \supp(\rho)$, non-decreasing on the set $\epsilon \in (0,r(\rho))$. This shows that $\epsilon \mapsto H_{\rho^{(\epsilon)}}(\theta)$ is non-decreasing for small enough $\epsilon$ (uniformly in $\theta$), and thus that 
	\[
		\liminf_{\epsilon \downarrow 0} r(\sigma^{(\epsilon)}) \geq r(\sigma),
	\]
	finishing the proof of \eqref{eqn:r_sigma_epsilon}.
 	
 	Now we prove uniform convergence of $I^{(\epsilon)}$ to $I$ on compact sets of $(r(\sigma),+\infty)$. Recall that $\theta_{\textup{max}}$ does not depend on $\epsilon$. If $x > r(\sigma)$, then for $\epsilon$ sufficiently small we have $x > r(\sigma^{(\epsilon)})$ and thus
 	\begin{align*}
 	I^{(\epsilon)}(x) &= \frac{1}{2} \int_{ r(\sigma^{(\epsilon)})}^x \Big( \overline{G}_{\sigma^{(\epsilon)}}(t) - G_{\sigma^{(\epsilon)}}(t) \Big) \diff t = \frac{1}{2} \int_{ r(\sigma^{(\epsilon)})}^x \int_{0}^{\theta_{\max}} \mathds{1}_{ G_{\sigma^{(\epsilon)}}(t) \leq u \leq \overline{G}_{\sigma^{(\epsilon)}}(t)} \diff u \diff t \\
	&=  \frac{1}{2} \int_{ 0 }^x \int_{0}^{\theta_{\max}} \mathds{1}_{ H_{\rho^{(\epsilon)}}(u) \leq t 
 		}  \diff u \diff t.
 	\end{align*} 
	Similarly, 
	\[
		I(x) = \frac{1}{2} \int_0^x \int_0^{\theta_{\textup{max}}} \mathds{1}_{ H_{\rho}(u) \leq t 
 		}  \diff u \diff t.
	\]
Define
 \begin{align*}
 	D^{(\epsilon)} &\defeq \{(t,u) \in (0,+\infty) \times (0, \theta_{\textup{max}}) : H_{\rho^{(\epsilon)}}(u) \geq t > H_{\rho}(u)\}, \\
	R_x &\defeq (0,x) \times (0,\theta_{\textup{max}}), \\
	D^{(\epsilon)}_x &\defeq D^{(\epsilon)} \cap R_x.
\end{align*}
From \eqref{eqn:H_vs_Heps}, we actually have $I^{(\epsilon)}(x) \leq I(x)$ and
\[
	I(x) - I^{(\epsilon)}(x) = \frac{1}{2} \int_0^x \int_0^{\theta_{\textup{max}}} \mathds{1}_{H_\rho(u) < t \leq H_{\rho^{(\epsilon)}}(u)} \diff u \diff t = \frac{1}{2} \Leb(D^{(\epsilon)}_x).
\]
Therefore, if $[a,b] \subset (r(\sigma),+\infty)$ then for all $x \in [a,b]$ and all $\epsilon < \epsilon_0(a)$ we have
 \[ | I^{(\epsilon)}(x) - I(x) |  \leq \Leb( D^{(\epsilon)}_b). \]
 Since $H_{\rho^{(\epsilon)}}$ decreases to $H_{\rho}$, the sets $D_b^{(\epsilon)}$ are nested, and their intersection over all $\epsilon > 0$ is empty. Since their Lebesgue measures are bounded above by $\Leb(R_b) < \infty$, we have $\lim_{\epsilon \downarrow 0} \Leb(D_b^{(\epsilon)}) = 0$, proving the uniform convergence of $I^{(\epsilon)}$ towards $I$ on compact sets of $(r(\sigma),+\infty)$.
 \end{proof} 

\begin{proof}[Proof of Lemma \ref{lem:Approx}]
	Deterministically, we have
	\[ \| H_N - H_N^{(\epsilon)} \| = \frac{1}{M} \| Z^{T} ( \Gamma - \Gamma^{(\epsilon)}) Z\| \leq \frac{1}{M} \|Z\|^2 \| \Gamma - \Gamma^{(\epsilon)} \| \leq \frac{\epsilon \|Z\|^2}{M}. \]
Therefore it is sufficient to prove that for every $K > 0$ there exists $t_K > 0$ such that, for all $t > t_K$, 
\[
	\limsup_{N \to \infty} \frac{1}{N} \log \P[ \sqrt{M}^{-1}\|Z\|\geq t] \leq -K.
\]
This can be deduced from the sub-Gaussian character of the entries of $Z$ and $\lim_N \frac{M}{N} = \alpha$ using for instance the arguments of \cite[Section 2]{GuiHus2020}.
\end{proof}

\begin{proof}[Proof of Lemma \ref{lem:dz}]
Define
\[
	J_\delta(x) = \liminf_{\epsilon \downarrow 0} \inf_{y \in (x-\delta,x+\delta)} I^{(\epsilon)}(y),
\]
which is non-increasing in $\delta$. If $x < r(\sigma)$, then by Lemma \ref{lem:CvRatef} there exists $\delta > 0$ such that $x + \delta < r(\sigma^{(\epsilon)})$ for all sufficiently small $\epsilon$, so $J_\delta(x) = +\infty$, and thus $J(x) = +\infty$. If $x > r(\sigma)$, then there exists $\delta > 0$ with $x - \delta > r(\sigma^{(\epsilon)})$ for all sufficiently small $\epsilon$. Since each $I^{(\epsilon)}$ is non-decreasing, this gives $J_\delta(x) = \liminf_{\epsilon \downarrow 0} I^{(\epsilon)}(x-\delta) = I(x-\delta)$, again by Lemma \ref{lem:CvRatef}, and thus $J(x) = I(x)$. Finally, we let $x = r(\sigma)$. Then for every $\delta > 0$ and all $\epsilon < \epsilon_0(\delta)$, we have $r(\sigma^{(\epsilon)}) \in (x-\delta,x+\delta)$, so that $J_\delta(x) = 0 = J(x)$. This completes the proof that $I = J$, and $I$ is clearly a good rate function, since it is infinite on $(-\infty,r(\sigma))$, vanishes uniquely at $r(\sigma)$, and is strictly increasing.

Now we check \eqref{eqn:dz4.2.18}, splitting into cases according to whether $\alpha(F) = \inf_{y \in F} I(y)$ is infinite or finite. If $\alpha(F) = +\infty$, then necessarily $F \subset (-\infty,r(\sigma))$, with $\sup\{y : y \in F\} < r(\sigma)$. Then Lemma \ref{lem:CvRatef} gives $\inf_{y \in F} I^{(\epsilon)}(y) = +\infty$ for all $\epsilon$ sufficiently small. If $\alpha(F) = 0$, there is nothing to prove. If $0 < \alpha(F) < \infty$, then $F \subset (r(\sigma)+\delta,+\infty)$ for some $\delta > 0$, and with $y_F = \min\{y : y \in F\}$ we have $\alpha(F) = I(y_F)$. Whenever $\epsilon$ is small enough that $y_F > r(\sigma^{(\epsilon)})$, we have  $\inf_{y \in F} I^{(\epsilon)}(y) = I^{(\epsilon)}(y_F)$; as $\epsilon \downarrow 0$ this tends to $I(y_F)$, by Lemma \ref{lem:CvRatef}.
\end{proof}


\section{The complex case}\label{SecComplex}
In this section, we will review the changes needed to adapt the proof of Theorem \ref{thm:main} to Theorem \ref{thm:mainc}.
\begin{itemize}
	
	\item We keep Definition \ref{def:j} and Definition \ref{def:IF}. However we need to modify Definition \ref{def:tilt} by replacing $\theta/2$ by $\theta$: 
	\[
	\frac{\diff \P^\theta}{\diff \P}(Z) = \frac{\E_e[e^{N\theta\ip{e,\frac{1}{M}Z^*\Gamma Ze}}]}{\E_{e,H_N}[e^{N\theta\ip{e,H_Ne}}]}.
	\]
	\item In Propositions \ref{prop:wkldpub} and \ref{prop:wkldplb} we multiply by $2$ both right hand sides:
	\[
	\limsup_{\delta \downarrow 0} \limsup_{N \to \infty} \frac{1}{N} \log \P_N^\theta(\abs{\lambda_{\textup{max}}(H_N) - x} \leq \delta) \begin{cases} \leq - 2(\widetilde{I_\sigma}(x) - I_\sigma(x,\theta)) & \text{if } x \in D, \\ = -\infty & \text{otherwise,} \end{cases}
	\] 
	and
	\[
	\liminf_{\delta \downarrow 0} \liminf_{N \to \infty} \frac{1}{N} \log \P_N(\abs{\lambda_{\textup{max}}(H_N) - x} < \delta) \geq -2 \widetilde{I_\sigma}(x).
	\]

	\item In Lemma \ref{lem:annealed} the equation \eqref{eqn:annealed} is replaced by
	\[
		\lim_{N \to \infty} \frac{1}{N} \log \E_{e,H_N}[e^{N\theta\ip{e,H_Ne}}] = 2 F(\rho,\theta).
	\]
	In the proof of this Lemma, the Hubbard-Stratonovich transformation becomes
	\[
	\E_{H_N}[e^{N\theta \ip{e,H_Ne}}] = \prod_{\mu=1}^M \frac{1}{\pi} \int_{w \in \C} \E\left[e^{2\Re(\overline{w}  \ip{z_\mu,e}) \sqrt{\frac{N}{M}\theta d_\mu}} \right] e^{-|w|^2} \diff w.
	\]
	In the Gaussian case  we have:
	\[ \prod_{j=1}^N \E[e^{2  \Re(\overline{w}\overline{(z_\mu)}_je_j) \sqrt{\frac{N}{M} \theta d_\mu)} }] = e^{|w|^2\frac{N}{M}\theta d_\mu} \]
	and then:
	\[
	\frac{1}{N} \log \E_{e,H_N}[e^{N\theta\ip{e,H_Ne}}] = -\frac{M}{N} \int_\R \log\left(1-\frac{N}{M}\theta t\right) \hat{\mu}_{\Gamma}(\diff t).
	\]
	
	In the sharp sub-Gaussian case we get for any $w \in \C$ and $c \in \R$:
	\[
	\E[\exp(c \Re( \overline{w} \ip{z_\mu,e}))] = \prod_{j=1}^N \E[\exp(c \Re( \overline{w(z_\mu)_j} e_j))] \leq \prod_{j=1}^N \exp\left(\frac{c^2 |w|^2 |e_j|^2}{4}\right) = \exp\left(\frac{c^2 |w|^2}{4}\right),
	\]
	leading to:
	\[
	\frac{1}{N} \log \E_{e,H_N}[e^{N\theta\ip{e,H_Ne}}] \leq -\frac{M}{N} \int_\R \log\left(1-\frac{N}{M}\theta t\right) \hat{\mu}_{\Gamma}(\diff t).
	\]

	Similar modifications happen  for the lower bound. 
	
	\item In Lemma \ref{lem:quenched}, we modify the equations to
	\begin{align*}
		&\lim_{\epsilon \downarrow 0} \limsup_{\delta \downarrow 0} \limsup_{N \to \infty} \sup_{Z \in \mc{A}^L_{x,\delta,\epsilon}} \abs{\frac{1}{N} \log \E_e[e^{N \theta\ip{e,H_Ne}}] - 2 J\left(\sigma,\frac{\theta}{2},x\right)} \\
		&= \lim_{\delta \downarrow 0} \limsup_{\epsilon \downarrow 0} \limsup_{N \to \infty} \sup_{Z \in \mc{A}^L_{x,\delta,\epsilon}} \abs{\frac{1}{N} \log \E_e[e^{N\theta\ip{e,H_Ne}}] - 2 J\left(\sigma,\frac{\theta}{2},x\right)} = 0.
	\end{align*}
	The proof is actually identical, we merely use the complex version (that is, $\beta=2$ of Theorem 6.2 in \cite{GuiHus2022}). One has to careful that the conventions for the function $J$ differ between this paper and \cite{GuiHus2022}. If we denote $J^{GH}$ the function $J$ used in \cite{GuiHus2022},
	\[ J^{GH}(\mu,\theta,x) = 2 J\Big( \mu, \frac{\theta}{2},x \Big). \]
	\item In Lemma \ref{lem:ub_good_event} we once again have to multiply by $2$ the right hand side:
	\[
	\limsup_{\delta \downarrow 0} \limsup_{\epsilon \downarrow 0} \limsup_{N \to \infty} \frac{1}{N} \log \P^\theta(\mc{A}_{x,\delta,\epsilon}^L) \leq -2(\widetilde{I_\sigma}(x) - I_\sigma(x,\theta)).
	\]
	The modifications made to Lemmas \ref{lem:annealed} and \ref{lem:quenched} carry over to the proof which otherwise remains identical.
	\item In Proposition \ref{prop:thetax}, the expression of $\theta_{\max}$ stays the same.
	\item In the proof of Proposition \ref{prop:wkldplb}, once again the modifications made to Lemmas \ref{lem:annealed}, \ref{lem:quenched} carry over and we get 
	\[
	\liminf_{\delta \downarrow 0} \liminf_{N \to \infty} \frac{1}{N} \log \P(\abs{\lambda_{\textup{max}}(H_N) - x} \leq \delta) \geq 2 F(\rho,\theta_x) - 2 J\left(\sigma,\frac{\theta_x}{2},x\right) = - 2 I_\sigma(x,\theta_x) \geq - 2\widetilde{I_\sigma}(x).
	\].
	
\end{itemize}


\setcounter{equation}{0}
\setcounter{thm}{0}
\renewcommand{\theequation}{A.\arabic{equation}}
\renewcommand{\thethm}{A.\arabic{thm}}
\appendix
\setcounter{secnumdepth}{0}
\hypertarget{sec:talagrand}{}
\section[Appendix A \ \ \ Concentration for multidimensional product measures]
{Appendix A \ \ \ Concentration for multidimensional product measures}

This section deals with a straightforward extension of classic results of Talagrand \cite{Tal1996} and Guionnet-Zeitouni \cite{GuiZei2000} on concentration for product measures, in order to consider complex-Hermitian random matrices with real and imaginary parts that are not necessarily independent of one another.

In the 1990s, Talagrand developed a theory of concentration for products of compactly-supported measures, obtaining results of the form ``If $f : [-1,1]^N \to \R$ is Lipschitz and has convex sublevel sets, and $(X_i)_{i=1}^N$ are independent random variables each valued in $[-1,1]$, then the random variable $f(X_1,\ldots,X_N)$ concentrates about its median'' \cite[Theorem 6.6]{Tal1996}. Guionnet and Zeitouni translated his results into random matrices, using them to show results of the form ``If the real-symmetric or complex-Hermitian Wigner matrix $W_N$ has compactly-supported entries, then $\hat{\mu}_{W_N}$ concentrates about its mean $\E[\hat{\mu}_{W_N}]$ in Wasserstein-$1$ distance'' \cite[Corollary 1.4(b)]{GuiZei2000}. However, since Talagrand's result was written for the most digestible case of $f : [-1,1]^N \to \R$, the complex-Hermitian case Guionnet and Zeitouni's result required the entries of $W_N$ to have \emph{independent} real and imaginary parts; then linear statistics of $W_N$ could indeed be nice functions of the $N^2$ independent random variables $(\re W_{ij},\im W_{ij})_{1 \leq i < j \leq N} \cup (W_{ii})_{i=1}^N$. 

We want to prove results about slightly more general Wigner matrices $W_N$, where the real and imaginary parts of each $W_{ij}$ are allowed to be correlated with each other, as long as the entries $W_{ij} \in \C$ remain independent for different upper-triangular values of $i$ and $j$. In order to do this, we need to extend Corollary 1.4(b) of \cite{GuiZei2000}, which in turn requires the following extension of Theorem 6.6 of \cite{Tal1996}. We copy Talagrand's language and most of his notation, so that the reader can more easily compare, but we introduce the $d$-dimensional Euclidean unit balls
\[
	B_d = \{x \in \R^d: \|x\|_2^2 \leq 1\}.
\]

\begin{prop}
\label{prop:talagrand}
Consider a real-valued function $f$ defined on $(B_d)^N$. We assume that, for each real number $a$,
\[
	\text{the set } \{f \leq a\} \text{ is convex.}
\]
Consider a convex set $C \subset (B_d)^N$, consider $\sigma > 0$ and assume that the restriction of $f$ to $C$ has a Lipschitz constant at most $\sigma$; that is, 
\[
	\forall \, x, y \in C, \qquad \abs{f(x) - f(y)} \leq \sigma \|x-y\|,
\]
where $\|x\|$ denotes the norm $\|(x_1, \ldots, x_N)\|^2 = \sum_{i=1}^N \|x_i\|_2^2$.

Consider independent random variables $(X_i)_{i \leq N}$ valued in $B_d$, and consider the random variable
\[
	h = f(X_1, \ldots, X_N).
\]
Then, if $M$ is a median of $h$, we have, for all $t > 0$, that
\begin{equation}
\label{eqn:talagrand}
	\P(\abs{h-M} \geq t) \leq 4c + \frac{4}{1-2c} \exp\left(-\frac{t^2}{16\sigma^2}\right)
\end{equation}
where we assume
\[
	c = \P((X_1, \ldots, X_N) \not\in C) < \frac{1}{2}.
\]
\end{prop}

We omit the proof, since it is a very straightforward update of Talagrand's $d = 1$ original. We remark that \eqref{eqn:talagrand} has no dependence on $d$, which may be initially surprising, since we make no assumptions about the correlations between the $d$ entries of each $X_i$. However, the lack of $d$-dependence is essentially because we have chosen to extend Talagrand's $d=1$ compact set $[-1,1]$ to $B_d$, which has Euclidean diameter $2$ for each $d$, rather then, e.g., to replace $[-1,1]$ with $[-1,1]^d$, which has Euclidean diameter $2\sqrt{d}$. (If we instead considered $f : ([-1,1]^d)^N \to \R$ and variables $X_i \in [-1,1]^d$, we would obtain a variant of \eqref{eqn:talagrand} with right-hand side $4c + \frac{4}{1-2c} \exp\left(-\frac{t^2}{16d\sigma^2}\right)$.) 

We suspect that an extension of this form has already appeared in the literature, perhaps more than once, but we have not been able to find it.

By thinking $\C \cong \R^2$ and using the $d = 2$ case of Proposition \ref{prop:talagrand}, we obtain the following extension of Corollary 1.4(b) of \cite{GuiZei2000}. Again we copy their language and notation for ease of comparison. We consider inhomogeneous complex-Hermitian random matrices $X_A$ given by
\[
	X_A = ((X_A)_{ij})_{1 \leq i, j \leq N}, \quad X_A = X^\ast_A, \quad (X_A)_{ij} = \frac{1}{\sqrt{N}} A_{ij} \omega_{ij}
\]
with
\begin{align*}
	\omega &= (\omega^R + \ii \omega^I) = (\omega_{ij})_{1 \leq i, j \leq N} = (\omega_{ij}^R + \sqrt{-1}\omega_{ij}^I)_{1 \leq i, j \leq N}, \quad \omega_{ij} = \overline{\omega_{ji}}, \\
	A &= (A_{ij})_{1 \leq i, j \leq N}, \quad A_{ij} = \overline{A_{ji}}.
\end{align*}
Here $\{\omega_{ij}, 1 \leq i \leq j \leq N\}$ are independent complex random variables with laws $\{P_{ij}, 1 \leq i \leq j \leq N\}$, and the $P_{ij}$'s are probability measures on $\C$, but now with no assumptions on the relationship between their real and imaginary marginals, except the condition that $P_{ii}$ is supported on $\R$ in order to keep $X_A$ Hermitian. Here $A$ is a non-random complex matrix with entries $\{A_{ij}, 1 \leq i \leq j \leq N\}$ uniformly bounded by, say, $a$. (By choosing all the $P_{ij}$'s to be supported on $\R$ and all entries of $A$ real, we can of course obtain results about real-symmetric random matrices.)

\begin{prop}
\label{prop:appendix_gui_zei}
Assume that the $(P_{ij}, i \leq j)$ are uniformly compactly supported, that is that there exists a compact set $K \subset \C$ so that for any $1 \leq i \leq j \leq N$, $P_{ij}(K^c) = 0$. Write
\[
	\|K\| = \max\{\|x\| : x \in K\}.
\]
Fix $\delta_1(N) = 8\|K\|\sqrt{\pi} a/N$ and $M = a\|K\|\sqrt{8}$. For any $\delta > (128(M+\sqrt{\delta})\delta_1(N))^{2/5}$, 
\[
	\P({\rm W}_1(\hat{\mu}_{X_A}, \E[\hat{\mu}_{X_A}]) > \delta) \leq \frac{128(M+\sqrt{\delta})}{\delta^{3/2}} \exp\left(-N^2 \frac{1}{16\|K\|^2a^2} \left[ \frac{\delta^{5/2}}{128(M+\sqrt{\delta})} - \delta_1(N)\right]^2 \right).
\]
\end{prop}

We again omit the proof, which just mimics that of Guionnet and Zeitouni; the only observation is that here $\|K\|$ is defined in such a way that $\frac{K}{\|K\|} \subset B_2$, so that we can use Proposition \ref{prop:talagrand} when Guionnet and Zeitouni use \cite[Theorem 6.6]{Tal1996}.


\setcounter{equation}{0}
\setcounter{thm}{0}
\renewcommand{\theequation}{B.\arabic{equation}}
\renewcommand{\thethm}{B.\arabic{thm}}
\appendix
\setcounter{secnumdepth}{0}
\hypertarget{sec:compact_or_log_sobolev}{}
\section[Appendix B \ \ \ The ``compact or log-Sobolev'' assumption]
{Appendix B \ \ \ The ``compact or log-Sobolev'' assumption}

In this section we explain how to remove a certain technical assumption from previous tilting results on top-eigenvalue LDPs in the sub-Gaussian case, which we will call the ``compact or log-Sobolev'' assumption. This assumption appears in various forms throughout the literature: earlier papers tend to literally require some underlying measure to have either compact support or to satisfy the log-Sobolev inequality, whereas later papers tend to require some statement about concentration of the empirical measure which is easy to verify in the compact-or-log-Sobolev case.

Our techniques to remove this assumption use a recent strengthening of the continuity properties of spherical integrals, due to Guionnet and the first author \cite{GuiHus2022}. This works as follows: With $\hat{\mu}_N$ the empirical spectral measure of the matrix one is studying as defined in \eqref{eqn:muhat}, $\mu_\infty$ its deterministic limit, and $d_{\textup{BL}}$ the bounded-Lipschitz distance from \eqref{eqn:distances}, previous results required estimates of the form
\begin{equation}
\label{eqn:intro_N_delta}
	\lim_{N \to \infty} \frac{1}{N} \log \P(d_{\textup{BL}}(\hat{\mu}_N, \mu_\infty) > N^{-\delta}) = -\infty
\end{equation}
for some small $\delta > 0$. Under the better continuity properties, it suffices to show
\begin{equation}
\label{eqn:intro_eps}
	\lim_{N \to \infty} \frac{1}{N} \log \P(d_{\textup{BL}}(\hat{\mu}_N,\mu_\infty) > \epsilon) = -\infty
\end{equation}
for every $\epsilon > 0$. We see two main benefits of \eqref{eqn:intro_eps} over \eqref{eqn:intro_N_delta}: First, we are about to show that \eqref{eqn:intro_eps} is often provable without the compact-or-log-Sobolev assumption, essentially by carefully truncating the matrix entries and using concentration results of Guionnet and Zeitouni, in the style of Talagrand, for compactly supported product measures. Second, one can typically show \eqref{eqn:intro_eps} without relying on local laws, which had been used in previous results to verify \eqref{eqn:intro_N_delta} (see, e.g., \cite{McK2021, Hus2022}). Since local laws are only available in some cases, we think it may be useful for future LDP results that their use can be bypassed.

We demonstrate the ideas in the simplest setting of sharp sub-Gaussian Wigner matrices, showing the following result, which removes the ``compact-or-log-Sobolev'' assumption from the main result of \cite{GuiHus2020}. 

\begin{thm}
Let $\mu$ be a centered probability measure on $\R$ with unit variance, and let $X_N$ be the corresponding Wigner matrix, i.e., $X_N$ is an $N \times N$ real-symmetric random matrix with i.i.d. entries up to symmetry distributed according to $\frac{\mu}{\sqrt{N}}$. If $\mu$ is sharp sub-Gaussian, then $\lambda_{\textup{max}}(X_N)$ satisfies an LDP at speed $N$ with the good rate function
\[
	I(x) = \begin{cases} +\infty & \text{if } x < 2, \\ \frac{1}{2} \int_2^x \sqrt{y^2-4} \diff y & \text{if } x \geq 2. \end{cases}
\]
\end{thm}

\begin{proof}
We claim that, for every $\epsilon > 0$, we have
\begin{equation}
\label{eqn:w_concentration}
	\lim_{N \to \infty} \frac{1}{N} \log \P(d_{\textup{BL}}(\hat{\mu}_{X_N}, \rho_{\textup{sc}}) > \epsilon) = -\infty.
\end{equation}
Lemma \cite{GuiHus2020} shows something slightly stronger, under the compact-or-log-Sobolev assumption, namely that there exists some small $\kappa > 0$ with $\lim_{N \to \infty} \frac{1}{N} \log \P(d_{\textup{BL}}(\hat{\mu}_{X_N}, \rho_{\textup{sc}}) > N^{-\kappa}) = -\infty$. However, mimicking the arguments in our Sections \ref{subsec:wk_ldp_ub} and \ref{subsec:wk_ldp_lb} shows that \eqref{eqn:w_concentration} suffices.

To prove \eqref{eqn:w_concentration}, we mimic the proof of Proposition \ref{prop:sub_Gaussian_concentration}, decomposing $X = A+B$ with $A_{ij} = X_{ij} \mathds{1}\{\abs{X_{ij}} \leq N^{\gamma-1/2}\}$ for some $\gamma = \gamma(\epsilon) > 0$ to be chosen, and then prove analogues of \eqref{eqn:truncating}, \eqref{eqn:elc}, \eqref{eqn:guizei}, and \eqref{eqn:deterministic}. The proof of \eqref{eqn:truncating} is as in the sample-covariance case, except that we use the classical result $d_{\textup{KS}}(\hat{\mu}_{X_N}, \hat{\mu}_{A_N}) \leq \frac{1}{N} \rank(X_N - A_N)$ \cite[Theorem A.43]{BaiSil2010}, which comes from interlacing of eigenvalues, instead \cite[Theorem A.44]{BaiSil2010} from interlacing of singular values as in the main text. The estimates \eqref{eqn:elc} and \eqref{eqn:deterministic} are just as in the main text, and the estimate \eqref{eqn:guizei} is actually easier in the Wigner case: Define $P = \sqrt{8} N^\gamma$ and $\epsilon_1(N) = 16\sqrt{\pi} N^{\gamma-1}$. Since $\sqrt{N} A_N$ has entries compactly supported in $[-N^\gamma,N^\gamma]$, \cite[Theorem 1.4(a)]{GuiZei2000} gives
\[
	\P({\rm W}_1(\hat{\mu}_{A_N}, \E[\hat{\mu}_{A_N}]) > \epsilon) \leq \frac{128(P+\sqrt{\epsilon})}{\epsilon^{3/2}} \exp\left(-N^2 \frac{1}{64N^{2\gamma}} \left[ \frac{\epsilon^{5/2}}{128(P+\sqrt{\epsilon})} - \epsilon_1(N)\right]^2 \right)
\]
as long as $\epsilon$ satisfies the implicit equation $\epsilon > (128(P+\sqrt{\epsilon})\epsilon_1(N))^{2/5}$, the right-hand side of which is order $N^{(2\gamma-1)(2/5)}$, so that the implicit equation is satisfied for all $N$ large enough. Then the argument of the exponential is order $N^{2-2\gamma+2(-\gamma)}$, and the power in the exponential is at least $1$ for $\gamma$ small enough, which suffices.
\end{proof}

\begin{rem}
\label{rem:new_ssg_example}
	This proof allows the laws of the entries of $X_N$ to be sharp sub-Gaussian without having to be compactly supported or satisfying a log-Sobolev inequality. Let us provide an example of such a law. For this, let us consider $Z$ a standard Gaussian variable, $\mathcal{F} = \sigma( \{ Z \in [ n, n+1[ \} : n \in \Z ) $ and let us define
	\[ \tilde{Z} = \text{sign}(Z) \sqrt{ \E[ Z^2 | \mathcal{F}]} .\]
	The variable $\tilde{Z}$ is centered of variance $1$. It is obviously not compactly supported, and since it is supported on a discrete set of points, it cannot satisfy any log-Sobolev inequality. It remains to see that $\tilde{Z}$ is sharp sub-Gaussian. For this, let us look at its moments. Since the law of $\tilde{Z}$ is symmetric for $k \in \N$, we have:
	\[ \E[ \tilde{Z}^{2k+1}] =0 \]
	and for even moments, applying Jensen's inequality, we have 
	\[ \E[ \tilde{Z}^{2k}] =\E[ \E[ Z^2 | \mathcal{F}]^k] \leq \E[ Z^{2k}] \]
	and $\Var(\tilde{Z}) = \E[\tilde{Z}^2] = \E[\E[Z^2 | \mathcal{F}]] = \E[Z^2] = 1$. 
	So since $Z$ is Gaussian, we have for $t \in \R$
	\[ \E[ \exp (t \tilde{Z}) ] \leq \E[ \exp(t Z)] \leq \exp \Big( \frac{t^2}{2} \Big), \]
	meaning $\tilde{Z}$ is indeed sharp sub-Gaussian.
\end{rem}


\setcounter{equation}{0}
\setcounter{thm}{0}
\renewcommand{\theequation}{C.\arabic{equation}}
\renewcommand{\thethm}{C.\arabic{thm}}
\appendix
\setcounter{secnumdepth}{0}
\hypertarget{sec:wigner}{}
\section[Appendix C \ \ \ Deformed Wigner matrices]
{Appendix C \ \ \ Deformed Wigner matrices}

In this appendix, we use our techniques to improve previous results of the second author for so-called ``deformed Wigner matrices'' \cite{McK2021}. In this model, one considers an $N \times N$ random matrix
\[
	X_N = \frac{W_N}{\sqrt{N}} + D_N,
\]
either real-symmetric or complex-Hermitian, where $W_N$ has i.i.d. entries up to symmetry each distributed according to some centered probability measure $\mu$ (on $\R$ if $\beta = 1$ or on $\C$ if $\beta = 2$), and where the deterministic matrix $D_N$ satisfies the following assumption, which is weaker than the corresponding assumption from \cite{McK2021}.

\begin{assn}
\label{assn:deformed_wigner}
The matrix $D_N$ is real, diagonal, and deterministic, and its empirical measure $\hat{\mu}_{D_N}$ tends weakly as $N \to \infty$ to a compactly supported probability measure $\mu_D$, and there are asymptotically no external outliers, in the sense that
\begin{align*}
	\lambda_{\textup{max}}(D_N) &\to r(\mu_D), \\
	\lambda_{\textup{min}}(D_N) &\to \ell(\mu_D).
\end{align*}
\end{assn}

\begin{rem}
A word on notation: Many of the objects we used in studying the generalized sample-covariance-case have analogues here. We have chosen to overload the notation rather than cluttering it. For example, the generalized-sample-covariance case has a threshold $x_c$, defined in \eqref{eqn:def_xc_rho}. We will need an analogous threshold here, and even though the definition \eqref{eqn:dw_xc} is different, we still call the new version $x_c$ rather than, e.g., $x_c^{(\textup{dw})}$. We have similarly overloaded $H$, the rate function $I$, and so on, with one notable exception: The special function $J(\mu,\theta,\lambda)$ given in Definition \ref{def:j} is exactly the same in both models.
\end{rem}

The typical limiting behavior of the empirical measure for this model is classical \cite{Pas1972, Voi1991}: We have
\[
	\hat{\mu}_{X_N} \to \rho_{\textup{sc}} \boxplus \mu_D \eqdef \mu^{\textup{sc}}_D,
\]
where $\rho_{\textup{sc}}$ is the semicircle law $\rho_{\textup{sc}}(\diff x) = \frac{\sqrt{(4-x^2)_+}}{2\pi} \diff x$, the notation $\boxplus$ denotes the free (additive) convolution of two compactly supported probability measures, and $\mu^{\textup{sc}}_D$ is shorthand. We recall that $\boxplus$ is defined in terms of the Voiculescu $R$-transform, which will be important for us: For a compactly supported probability measure $\mu$, we recall the Stieltjes transform $G_\mu$, which is a decreasing bijection from $(r(\mu),+\infty)$ to $(0,G_\mu(r(\mu)))$. We write its inverse as $K_\mu : (0,G_\mu(r(\mu))) \to (r(\mu),+\infty)$, and its $R$-transform as $R_\mu(y) = K_\mu(y) - \frac{1}{y}$, which linearizes free convolution in the sense that $R_{\mu \boxplus \nu} = R_\mu + R_\nu$.

The following lemma defines functions we need to write the rate function.

\begin{lem}
\label{lem:dw_h}
The function $H : (0,\infty) \to \R$ defined by
\[
	H(y) \defeq \begin{cases} y + K_{\mu_D}(y) & \text{if } 0 \leq y \leq G_{\mu_D}(r(\mu_D)), \\ y+r(\mu_D) & \text{if } y \geq G_{\mu_D}(r(\mu_D)), \end{cases}
\]
has the following properties:
\begin{itemize}
\item $H$ is continuous and convex.
\item $H$ is uniquely minimized at $y_c = G_{\rho_{\mu^{\textup{sc}}_D}}(r(\mu^{\textup{sc}}_D))$, which is in $(0,G_{\mu_D}(r(\mu_D))]$, and $H(y_c) = r(\mu^{\textup{sc}}_D)$.
\item There exists a continuous, strictly increasing function $\overline{\mf{G}} : [r(\mu^{\textup{sc}}_D),+\infty) \to \R$ with the following properties: 
\begin{itemize}
\item If $0 \leq y \leq G_{\mu_D}(r(\mu_D))$, then $H(\overline{\mf{G}}(y)) = y$, and $\{w : H(w) = y\} = \{G_{\mu^{\textup{sc}}_D}(y), \overline{\mf{G}}(y)\}$. 
\item We have $\overline{\mf{G}}(y) > G_{\mu^{\textup{sc}}_D}(y)$ for $y > r(\mu^{\textup{sc}}_D)$, and $\overline{\mf{G}}(r(\mu^{\textup{sc}}_D)) = G_{\mu^{\textup{sc}}_D}(r(\mu^{\textup{sc}}_D))$. 
\end{itemize}
\end{itemize}
\end{lem}

\begin{defn}
Define $I : \R \to [0,+\infty]$ by
\[
	I(x) = \begin{cases} +\infty & \text{if } x < r(\rho_{\textup{sc}} \boxplus \mu_D), \\ \frac{1}{2} \int_{r(\rho_{\textup{sc}} \boxplus \mu_D)}^x (\overline{\mf{G}}(y) - G_{\rho_{\textup{sc}} \boxplus \mu_D}(y)) \diff y & \text{if } x \geq r(\rho_{\textup{sc}} \boxplus \mu_D). \end{cases} 
\]
\end{defn}

\begin{thm}
\label{thm:dw}
Suppose that Assumption \ref{assn:deformed_wigner} holds, and that $\mu$ is sharp sub-Gaussian (in the sense of Definition \ref{defn:ssg} if $\beta = 1$, or in the sense of Definition \ref{defn:ssgc} if $\beta = 2$). Then $\lambda_{\textup{max}}(X_N)$ satisfies a large deviation principle at speed $N$ with the good rate function $I^{(\beta)} = \beta I$. This function is convex, strictly increasing on $[r(\rho_{\textup{sc}} \boxplus \mu_D),+\infty)$ (in particular, vanishes uniquely at $r(\rho_{\textup{sc}} \boxplus \mu_D)$), and 
\[
	\lim_{x \to +\infty} \frac{I^{(\beta)}(x)}{\frac{\beta x^2}{4}} = 1.
\]
If $\mu$ is actually Gaussian, then by rotational invariance we do not need to assume that $D_N$ is diagonal; we can just assume it is symmetric (if $\beta = 1$) or Hermitian (if $\beta = 2$) and satisfies the rest of Assumption \ref{assn:deformed_wigner}.
\end{thm}

\begin{rem}
This improves upon the result of \cite{McK2021}, which showed a weaker version of Theorem \ref{thm:dw} requiring three additional assumptions: First, that either $\mu$ actually be Gaussian measure (i.e., that $\frac{W_N}{\sqrt{N}}$ be GOE/GUE), or that the important threshold
\begin{equation}
\label{eqn:dw_xc}
	x_c = x_c(\mu_D) \defeq \begin{cases} r(\mu_D) + G_{\mu_D}(r(\mu_D)) & \text{if } G_{\mu_D}(r(\mu_D)) < +\infty, \\ +\infty & \text{otherwise} \end{cases} 
\end{equation}
be infinite; second, that $\mu$ be either compactly supported or satisfy the log-Sobolev inequality; third, that the deformation tend to its limit at some mild polynomial speed, $d_{\textup{BL}}(\hat{\mu}_{D_N},\mu_D) \lesssim N^{-\epsilon}$ for some $\epsilon > 0$. However, the rate function was given there in a different form; below we show that the forms are equivalent.

We get rid of the second and third assumptions using the methods of Appendix \hyperlink{sec:compact_or_log_sobolev}{B}. We get rid of the first assumption as in the main text, namely by approximating $x_c < +\infty$ models with a sequence of $x_c = +\infty$ models, then using textbook results about approximating LDPs.
\end{rem}

\begin{defn}
For any $x \geq r(\rho_{\textup{sc}} \boxplus \mu_D)$ and any $\theta \geq 0$, define
\[
	I(x,\theta) \defeq J(\rho_{\textup{sc}} \boxplus \mu_D,\theta,x) - \theta^2 - J(\mu_D,\theta,r(\mu_D)).
\]
Using this, define the rate function $\widetilde{I} : \R \to [0,+\infty]$ by
\[
	\widetilde{I}(x) \defeq \begin{cases} + \infty & \text{if } x < r(\rho_{\textup{sc}} \boxplus \mu_D), \\ \sup_{\theta \geq 0} I(x,\theta) & \text{if } x \geq r(\rho_{\textup{sc}} \boxplus \mu_D). \end{cases}
\]
\end{defn}

\begin{lem}
\label{lem:dw_i_tilde}
If $x_c(\mu_D) = +\infty$, then $I = \widetilde{I}$. 
\end{lem}

\begin{lem}
\label{lem:dw_infinite}
If Assumption \ref{assn:deformed_wigner} holds, and $\mu$ is sharp sub-Gaussian, and
\[
	x_c(\mu_D) = +\infty, 
\]
then $\lambda_{\textup{max}}(X_N)$ satisfies a large deviation principle at speed $N$ with the good rate function $I$.
\end{lem}

\begin{lem}
\label{lem:dw_finite}
Lemma \ref{lem:dw_infinite} implies Theorem \ref{thm:dw}.
\end{lem}

\begin{proof}[Proof of Lemma \ref{lem:dw_h}]
Continuity is easy to check. Since $G_{\mu_D}$ is strictly convex on $(r(\mu_D),+\infty)$, the function $K_{\mu_D}$ is strictly convex on $(0,G_{\mu_D}(r(\mu_D)))$, so that $H$ is convex, indeed strictly on $(0,G_{\mu_D}(r(\mu_D)))$. Since the boundary behavior is $\lim_{y \to 0} H(y) = \lim_{y \to +\infty} H(y) = +\infty$, we have that $H$ is uniquely minimized at some $y_c$, which must be in $(0,G_{\mu_D}(r(\mu_D))]$. On the other hand, let $y_s \defeq G_{\mu^{\textup{sc}}_D}(r(\mu^{\textup{sc}}_D))$. Lemma 6.1 of \cite{GuiMai2020} gives $y_s \leq \min(G_{\rho_{\textup{sc}}}(r(\rho_{\textup{sc}})), G_{\mu_D}(r(\mu_D))$, and a computation in the proof of Proposition 6.1 in \cite{McK2021} shows $G'_{\mu_D}(y_s) = -1$, and thus, via differentiating $y = K_{\mu_D}(G_{\mu_D}(y))$ in $y$ and evaluating at $y_s$, that $H'(y_s) = 0$. Thus $y_s = y_c$. Another computation in the proof of Proposition 6.1 in \cite{McK2021} shows $H(y_s) = r(\mu_D^{\textup{sc}})$. 

Now we study $\overline{\mf{G}}$. Since $R_{\rho_{\textup{sc}}}(y) = y$ in our normalization, we have $H(y) = R_{\rho_{\textup{sc}}}(y) + R_{\mu_D}(y) + \frac{1}{y} = K_{\rho_{\textup{sc}} \boxplus \mu_D}(y)$ for $y < y_c$. Since $H$ is continuous, strictly convex on $(0,b)$ for some $b$, affine increasing on $(b,+\infty)$, and minimized at $y_c \in (0,b]$, it is easy to see that $\#\{w : H(w) = y\}$ is zero for $y < H(y_c)$, one for $y = H(y_c)$, and two for $y > H(y_c)$; that the smaller of these elements is $G_{\rho_{\textup{sc}} \boxplus \mu_D}(y)$; and that, if $\overline{\mf{G}}$ denotes the inverse of $H$ on $(y_c,+\infty)$, then $\overline{\mf{G}}$ has the claimed properties.
\end{proof}

\begin{proof}[Proof of Lemma \ref{lem:dw_i_tilde}]
The proof goes as in the generalized-sample-covariance case, i.e., by showing that $I$ and $\widetilde{I}$ have the same derivative on $(r(\rho_{\textup{sc}} \boxplus \mu_D),+\infty)$, and both vanish at $r(\rho_{\textup{sc}} \boxplus \mu_D)$. We rely on two computations already carried out in Proposition 6.1 of \cite{McK2021}. The first of these shows that $\widetilde{I}$ vanishes at $r(\rho_{\textup{sc}} \boxplus \mu_D)$. The second shows that, for each $x > r(\rho_{\textup{sc}} \boxplus \mu_D)$, we have $\widetilde{I}(x) = I(x,\theta_x)$, where $\theta_x$ is the unique solution to the constrained problem
\[
	H(2\theta_x) = 2\theta_x + K_{\mu_D}(2\theta_x) = x \quad \text{subject to} \quad 2\theta_x \in (G_{\rho_{\textup{sc}} \boxplus \mu_D}(r(\rho_{\textup{sc}} \boxplus \mu_D)), G_{\mu_D}(r(\mu_D))),
\]
which in the new language of branches of Stieltjes transforms we recognize as $\theta_x = \frac{1}{2}\overline{\mf{G}}(x)$;
and that this maximizer is unique, i.e., $I(x,\theta) < \widetilde{I}(x)$ if $\theta \neq \theta_x$. Thus 
\[
	\frac{\diff}{\diff x} \widetilde{I}(x) = \left. \frac{\partial}{\partial x} I(x,\theta) \right|_{\theta = \theta_x} = \left. \frac{\partial}{\partial x} J(\rho_{\textup{sc}} \boxplus \mu_D,\theta,x) \right|_{\theta = \theta_x} = \theta_x - \frac{1}{2}G_{\rho_{\textup{sc}} \boxplus \mu_D}(x) = \frac{1}{2}(\overline{\mf{G}}(x) - G_{\rho_{\textup{sc}} \boxplus \mu_D}),
\]
which completes the proof. 
\end{proof}

\begin{proof}[Proof of Lemma \ref{lem:dw_infinite}]
As stated above, this is the main result of \cite{McK2021}, except that (a) the rate function was written there as $\beta \widetilde{I}$ (which Lemma \ref{lem:dw_i_tilde} shows is irrelevant), and (b) that paper required $\mu$ to be either compactly supported or to satisfy log-Sobolev, and required 
\begin{equation}
\label{eqn:dw_dspeed}
	d_{\textup{BL}}(\hat{\mu}_{D_N}, \mu_D) \leq CN^{-\epsilon}
\end{equation}
for some $C$ and $\epsilon$. Under the additional (b) assumptions, Lemma 5.3 of \cite{McK2021} showed 
\[
	\lim_{N \to \infty} \frac{1}{N} \log \P(d_{\textup{BL}}(\hat{\mu}_{X_N}, \rho_{\textup{sc}} \boxplus \mu_D) > N^{-\kappa}) = -\infty
\]
for $\kappa > 0$ sufficiently small. Obtaining this small polynomial speed (a) required (a very weak consequence) of the local law \cite{ErdKruSch2019}, and (b) was the essential reason for requiring \eqref{eqn:dw_dspeed}. But Appendix \hyperlink{sec:compact_or_log_sobolev}{B} explains that it actually suffices to show
\begin{equation}
\label{eqn:dw_distance}
	\lim_{N \to \infty} \frac{1}{N} \log \P(d_{\textup{BL}}(\hat{\mu}_{X_N}, \rho_{\textup{sc}} \boxplus \mu_D) > \epsilon) = -\infty
\end{equation}
for every $\epsilon > 0$; this allows us to drop the requirement of \eqref{eqn:dw_dspeed}, and to give a proof bypassing the local law. To do this, as in Appendix \hyperlink{sec:compact_or_log_sobolev}{B}, we split $W_N N^{-1/2} = A + B = A_N + B_N$ with $A_{ij} = W_{ij} N^{-1/2} \mathds{1}\{\abs{W_{ij} N^{-1/2}} \leq N^{\gamma - 1/2}\}$ for some $\gamma = \gamma(\epsilon) > 0$ to be chosen, then decompose $X_N = (A + D_N) + (B + D_N)$ and show analogues of \eqref{eqn:truncating}, \eqref{eqn:elc}, \eqref{eqn:guizei}, and \eqref{eqn:deterministic}. The analogues of \eqref{eqn:truncating}, \eqref{eqn:elc}, and \eqref{eqn:deterministic} go through exactly as before. The analogue of \eqref{eqn:guizei} is the estimate
\[
	\lim_{N \to \infty} \frac{1}{N} \log \P(d_{\textup{BL}}(\hat{\mu}_{A + D_N}, \E[\hat{\mu}_{A + D_N}]) > \epsilon) = -\infty.
\]
In Appendix \hyperlink{sec:compact_or_log_sobolev}{B}, when $D_N = 0$, we noted that $\sqrt{N}A$ had entries compactly supported in $[-N^\gamma,N^\gamma]$, and applied results of \cite{GuiZei2000}. When $D_N = 0$, the matrix $\sqrt{N}(A + D_N)$ has entries compactly supported in boxes of size order $N^\gamma$, but whose centers have shifted away from zero. This requires a straightforward modification of the results of \cite{GuiZei2000}, which was already stated as Lemma 5.9 of \cite{McK2021}. This completes the proof of \eqref{eqn:dw_distance}.
\end{proof}

\begin{proof}[Proof of Lemma \ref{lem:dw_finite}]
This proof also goes as in the generalized-sample-covariance case. We use throughout results of Lemma \ref{lem:dw_h}, and only write the case $\beta = 1$ for simplicity.

From its definition and Lemma \ref{lem:dw_h}, we see that the rate function has the form $I(x) = \frac{1}{2} \int_{r(\rho_{\textup{sc}} \boxplus \mu_D)}^x g(y) \diff y$, where $g$ is strictly increasing, positive for arbitrarily small arguments, and $\lim_{x \to +\infty} \frac{g(x)}{x} = 1$; this proves the claimed properties of $I$. 

Fix once and for all some $\mu_D$ with $x_c(\mu_D) < +\infty$, and write $d_i = d_i^{(N)}$ for the diagonal entries of $D_N$, i.e., $D_N = \diag(d_1, \ldots, d_N)$. For $\epsilon > 0$, define $D_N^{(\epsilon)} = \diag(d_1^{(\epsilon)}, \ldots, d_N^{(\epsilon)})$, where $d_i^{(\epsilon)} = d_i$ if $d_i \leq r(\mu_D) - \epsilon$ and $d_i^{(\epsilon)} = r(\mu_D)$ otherwise. Then set
\[
	X_N^{(\epsilon)} = N^{-1/2} W_N + D_N^{(\epsilon)},
\]
coupled with $X_N$ by using the same randomness $W_N$. Set $\mu_D^{(\epsilon)}$ as in \eqref{eqn:rho_epsilon}, and notice that if $\epsilon \to 0$ avoids any atoms present near $r(\mu_D)$, we have that the empirical measure of $X_N^{(\epsilon)}$ tends to $\rho_{\textup{sc}} \boxplus \mu_D^{(\epsilon)}$, with a corresponding function $H^{(\epsilon)}$. By construction, $r(\mu_D^{(\epsilon)}) = r(\mu_D)$, but $G_{\mu_D^{(\epsilon)}}(r(\mu_D^{(\epsilon)})) = +\infty$ for each $\epsilon$, so that $\lambda_{\textup{max}}(X_N^{(\epsilon)})$ satisfies an LDP at speed $N$ with the good rate function $I^{(\epsilon)}$ defined as
\[
	I^{(\epsilon)}(x) = \begin{cases} +\infty & \text{if } x < r(\rho_{\textup{sc}} \boxplus \mu_D^{(\epsilon)}), \\ \frac{1}{2} \int_{r(\rho_{\textup{sc}} \boxplus \mu_D)}^x (\overline{\mf{G}}^{(\epsilon)}(y) - G_{\rho_{\textup{sc}} \boxplus \mu_D^{(\epsilon)}}(y)) \diff y & \text{if } x \geq r(\rho_{\textup{sc}} \boxplus \mu_D^{(\epsilon)}), \end{cases}
\]
where $\overline{\mf{G}}^{(\epsilon)}$ is defined as in Lemma \ref{lem:dw_h} for the measure $\mu_D^{(\epsilon)}$. 

To finish the proof, we just need analogues of Lemmas \ref{lem:CvRatef}, \ref{lem:Approx}, and \ref{lem:dz}. The analogue of Lemma \ref{lem:Approx} is even easier here, since $\|X_N - X_N^{(\epsilon)}\| \leq \epsilon$ deterministically. Lemma \ref{lem:dz} was essentially a consequence of Lemma \ref{lem:CvRatef}, and this remains true here, so we only need the analogue of Lemma \ref{lem:CvRatef}.

Towards this result: Similar arguments as in the main text show that, for each $x > r(\mu_D) = r(\mu_D^{(\epsilon)})$, the function $\epsilon \mapsto G_{\mu_D^{(\epsilon)}}(x)$ is non-decreasing for small $\epsilon > 0$, so that $\epsilon \mapsto K_{\mu_D^{(\epsilon)}}(y)$ is non-decreasing, and thus $\epsilon \mapsto H^{(\epsilon)}(y)$ is non-decreasing. Thus $\epsilon \mapsto r(\rho_{\textup{sc}} \boxplus \mu_D^{(\epsilon)})$ is non-decreasing, and $\liminf_{\epsilon \downarrow 0} r(\rho_{\textup{sc}} \boxplus \mu_D^{(\epsilon)}) \geq r(\rho_{\textup{sc}} \boxplus \mu_D)$. For the other inequality, it is clear that $H^{(\epsilon)}$ converges to $H$ uniformly on all compact subsets of $(0,G_{\mu_D}(r(\mu_D))]$. But $H^{(\epsilon)}(y) - y$ is strictly decreasing in $y$, and tends to $r(\mu_D)$ as $y \to +\infty$; thus $H^{(\epsilon)}$ tends to $H$ uniformly on all compact subsets of $(0,+\infty)$, and hence $\lim_{\epsilon \to 0} r(\rho_{\textup{sc}} \boxplus \mu_D^{(\epsilon)}) = r(\rho_{\textup{sc}} \boxplus \mu_D)$. Finally, for $r > r(\rho_{\textup{sc}} \boxplus \mu_D)$ and $\epsilon$ small enough we have the representations
\[
	I^{(\epsilon)}(x) = \frac{1}{2} \int_0^x \int_0^\infty \mathds{1}_{H^{(\epsilon)}(u) \leq t} \diff u \diff t \quad \text{and} \quad I(x) = \frac{1}{2} \int_0^x \int_0^\infty \mathds{1}_{H(u) \leq t} \diff u \diff t, 
\]
so that again $I^{(\epsilon)}(x) \leq I(x)$ with $\sup_{x \in [a,b]} \abs{I^{(\epsilon)}(x) - I(x)} \leq \Leb(D_b^{(\epsilon)})$, with $D_b^{(\epsilon)}$ redefined appropriately. The sets $D_b^{(\epsilon)}$ are again nested with empty intersection, so their Lebesgue measure tends to zero if it is finite. Before this finiteness was immediate, but here takes a moment's thought: Since $\lim_{y \to +\infty} H(y) = +\infty$, there exists $c$ with $H(y) \geq b$ for all $y \geq c$, and then $D_b^{(\epsilon)} \subset [0,b] \times [0,c]$, so $\Leb(D_b^{(\epsilon)}) < +\infty$. The rest of the proof goes as in the main text.
\end{proof}

\addcontentsline{toc}{section}{References}
\bibliographystyle{alpha-abbrvsort}
\bibliography{ldpbib}

\end{document}